\documentclass[reqno,11pt]{amsart}
\usepackage{tkz-euclide}
\usepackage{hyperref}
\usepackage{amsmath} 
\usepackage{amsfonts}
\usepackage{amssymb} 
\usepackage{mathrsfs} 
\usepackage{graphicx}  
\usepackage{bm}
\usepackage{tikz} 
\usepackage{centernot}
\usetikzlibrary{svg.path} 
\usepackage{longtable}
\usepackage{makecell}

\usepackage{geometry}
\usetikzlibrary{calc} 
\usepackage{xcolor}  
\usetikzlibrary{patterns}
\usetikzlibrary{arrows.meta} 
\usepackage{amsthm} 
\usepackage{float}
\usepackage{enumitem}
\setenumerate{ ref={\normalfont(\roman*)},label={\normalfont(\roman*)},itemsep=3pt}
\usepackage[english]{babel}
\usepackage[font=footnotesize, labelfont=bf]{ caption}
\usepackage{ esint }
\usepackage{stackengine,scalerel,graphicx}
\stackMath

\definecolor{trueblue}{rgb}{0.0, 0.45, 0.81}
\definecolor{truegreen}{rgb}{0.13, 0.55, 0.13}

\newcommand{\e}{\varepsilon}

\theoremstyle{plain}
\begingroup
\newtheorem{theorem}{Theorem}[section]
\newtheorem{lemma}[theorem]{Lemma}

\newtheorem{proposition}[theorem]{Proposition}
\newtheorem{corollary}[theorem]{Corollary}
\newtheorem*{ltheorem}{Liouville's Theorem}
\endgroup

\theoremstyle{definition}
\begingroup

\newtheorem{remark}[theorem]{Remark}
\endgroup

\newcommand{\B}{\mb B^n}
\renewcommand{\tilde}{\widetilde}

\renewcommand{\t}{\nabla_T}
\renewcommand{\d}{ \mathrm{d}}

\DeclareMathOperator{\dv}{div}

\numberwithin{equation}{section}
\newcommand{\N}{\mathbb{N}}
\newcommand{\Z}{\mathbb{Z}}

\newcommand{\E}{\mathcal{E}}
\newcommand{\R}{\mathbb{R}}
\newcommand{\C}{\mathbb{C}}

\newcommand{\EEE}{\color{black}}

\renewcommand{\S}{{\mathbb{S}^{n-1}}}

\renewcommand{\L}{\mathcal{L}}
\renewcommand{\H}{\mathcal{H}}

\newcommand{\dy}{\,\mathrm{d}y}

\usepackage[normalem]{ulem}

\def \mb{\mathbb}

\def \n{\nabla}
\def \p{\partial}
\def \d{\,\mathrm{d}}
\def \mc{\mathcal}

\def \haus{\mathcal H^{n-1}}
\def \Mob{\textup{M\"ob}}
\def \mob{\mathfrak{mob}}

\def \SO{\textup{SO}}
\def \tp{\textup}
\def \wstar{\overset{\ast}{\rightharpoonup}}

\DeclareMathOperator{\cof}{cof}
\DeclareMathOperator{\ddiv}{div}
\DeclareMathOperator{\id}{id}

\begin{document}
	
\title[Optimal   quantitative   stability of the M\"obius group]
{Optimal   quantitative   stability 
\\
of  the M\"obius group of the sphere 
 in all dimensions 
}
	
\author[Andr\'e Guerra]{Andr\'e Guerra} 
\address[Andr\'e Guerra]{Institute for Theoretical Studies, ETH Zürich, CLV, Clausiusstrasse 47, 8006 Zürich, Switzerland}
\email{andre.guerra@eth-its.ethz.ch}

\author[Xavier Lamy]{Xavier Lamy} 
\address[{Xavier Lamy}]{Institut de Mathématiques de Toulouse, Université Paul Sabatier, 118, route de Narbonne, F-31062 Toulouse Cedex 8, France}
\email{Xavier.Lamy@math.univ-toulouse.fr}	
	
\author{Konstantinos Zemas}
\address[Konstantinos Zemas]{Institute for Applied Mathematics, University of Bonn\\
Endenicher Allee 60, 53115 Bonn, Germany}
\email{zemas@iam.uni-bonn.de}
	


\begin{abstract}
In any dimension $n\geq 3$,
we prove an optimal stability estimate for the M\"obius group among maps $u\colon \mathbb S^{n-1} \to \mathbb R^n$,
of the form
\begin{align*}
\inf_{\lambda>0,\phi\in \textup{M\"ob}(\mathbb S^{n-1})}
 \int_{\mathbb S^{n-1}}\left|\frac 1\lambda \nabla_{ T} u -\nabla_{ T}\phi\right|^{n-1}
 \! \d\mathcal H^{n-1}
\lesssim  
\mathcal E_{n-1}(u)\,.
\end{align*}
Here, $\mathcal E_{n-1}(u)$ is a conformally invariant deficit which measures 
 simultaneously lack of conformality and the deviation of $u(\mathbb S^{n-1})$ from being a round sphere in an isoperimetric sense.
This entails in particular
 the following qualitative statement: sequences with vanishing deficit, once appropriately normalized by the action of the M\"obius group, are compact. Both the qualitative and the quantitative results are new for all dimensions $n\geq 4$.
\end{abstract} 
	
\maketitle
\thispagestyle{empty}

	
\section{Introduction}\label{sec:intro}

In this work we investigate  stability properties
 of the conformal group of the sphere. To be precise, 
a map $u\in W^{1,n-1}(\S;\R^n)$ is said to be \textit{weakly conformal} if and only if it satisfies  
\begin{equation}
\label{eq:weakconf}
\left(\n_T u\right)^t\n_T u = \frac{|\n_T u|^2}{n-1} I_x
\end{equation}
for $\mc H^{n-1}$-a.e.\ $x\in \S$. Here and in the rest of the paper we take $n\geq 3$, $\nabla_Tu
$ denotes the tangential gradient of $u$ and $I_x\colon T_x\S\to T_x\S$ is the identity map, see Sections~\ref{Notation} and   \ref{sec:prelims} for further details and notation.
The prototypical solution of \eqref{eq:weakconf} is a \textit{M\"obius transformation} of $\S$, \textit{i.e.},  an $\S$-valued map obtained as the composition of rotations and spherical inversions. The class of such maps forms a finite dimensional Lie group, which we denote by $\Mob(\S)$.

A classical theorem in conformal geometry,  which is usually credited to Liouville, classifies completely the solutions to the nonlinear system \eqref{eq:weakconf} in the case of $\S$-valued maps:

\begin{ltheorem}\label{liouville_thm}
Let $u\in W^{1,n-1}(\mb S^{n-1};\mb S^{n-1})$ be a solution to  \eqref{eq:weakconf}. 
\begin{enumerate}
\item If $n=3$, then $u$ is holomorphic,  after identifying 
 $\mathbb S^2$ with the Riemann sphere 
 $\widehat\C$.  In particular, if $|\deg u|=1$, then $u\in \Mob(\mb S^{2})$.
\item\label{it:n>3} If $n\geq 4$, then $u\in \Mob(\mb S^{n-1})$.
\end{enumerate}
\end{ltheorem}

Liouville's Theorem makes manifest that system \eqref{eq:weakconf} behaves very differently depending on the dimension: for maps $u\colon \mb S^2 \to \mb S^2$, it reduces to the Cauchy--Riemann equations, which form a \textit{linear determined system}, while in higher dimensions it is genuinely \textit{nonlinear} and \textit{overdetermined}. In fact, the rigidity in \ref{it:n>3} above holds even on subdomains of $\S$,  cf.\ \cite[Chapter 5]{Iwaniec2001}. We will discuss this difference in behavior  in more detail below.

The constraint on the image of the maps in Liouville's Theorem is crucial, as there is an enormous family of solutions $u\colon \S \to \R^n$ to \eqref{eq:weakconf}: for instance, the Uniformization Theorem provides us with many examples of conformal maps from $\mathbb{S}^2$ to $\R^3$ which are not rescaled and translated copies of M\"obius transformations. Moreover, according to the Nash--Kuiper theorem  \cite{Nash1954,Kuiper1955} and its refinements \cite{Borisov2004,Conti2012,DeLellis2018,Cao2022}, there is an abundance of $C^{1,\alpha}$-isometric embeddings of $\S$, for $\alpha>0$ small enough:  in particular, there are many wild Sobolev solutions to \eqref{eq:weakconf}.  
 
 \subsection{Main results}
In this paper we address the  question of  stability of the M\"obius group among maps  $u\colon \S \to \R^n$: if $u$ is \textit{nearly conformal} and its image $u(\S)$ is \textit{nearly spherical}, is $u$ close to a M\"obius transformation in a strong sense? While the isoperimetric inequality and its stability properties \cite{Fusco2008,Figalli2010} give a natural characterization of nearly spherical sets, a fundamental aspect of this problem  also lies in deciding how to measure the deviation from conformality, and  in this regard  we will pursue the strategy of \cite{zemas2022rigidity}.

For a map $u\in W^{1,n-1}(\S;\R^n)$, we consider the \textit{conformal $(n-1)$-Dirichlet energy}
\begin{equation}\label{eq:Dirichlet}
\mc D_{n-1}(u)
:=\fint_{\mathbb S^{n-1}}\bigg(\frac{|\nabla_T u|^2}{n-1}\bigg)^{\frac{n-1}{2}}\d \haus\,,
\end{equation}
and the \textit{extrinsic volume functional}
\begin{equation}\label{eq:vol}
\mc V_n(u)
:=\fint_{\mathbb S^{n-1}}\big\langle u ,\bigwedge_{i=1}^{n-1}\partial_{\tau_i} u\big\rangle \d \haus\,,
\end{equation}
studied in detail for $n=3$ by Wente \cite{Wente1969}.
 In \eqref{eq:vol},  we have identified an $(n-1)$-vector in $\R^n$ with its Hodge dual. 
It is well-known but not obvious at first sight 
that $\mathcal V_n(u)$ 
is well-defined for $u\in W^{1,n-1}(\mathbb S^{n-1};\R^n)$, and we refer the reader to Section~\ref{sec:degree} for further details.

Following \cite{zemas2022rigidity} (see the discussion after (1.9) therein),  we consider a \textit{combined conformal-isoperimetric deficit}, via
\begin{equation}\label{eq:deficit}
\mc E_{n-1}(u) := \begin{cases}\frac{\mc [D_{n-1}(u)]^{\frac{n}{n-1}}}{|\mc V_n(u)|} -1 & \text{ if } \mc V_n(u)\neq 0\,,\\[3pt]
+\infty& \text{ if } \mc V_n(u)= 0\,. 
\end{cases}
\end{equation}
We note that this deficit is invariant under translations, rotations and dilations of the maps, as well as precompositions with M\"obius transformations of $\S$, and its choice is precisely justified by the following:

\begin{proposition}[Conformal isoperimetric inequality]\label{prop:isopintro}
For all $u\in W^{1,n-1}(\mb S^{n-1};\R^n)$ we have 
$$\mc E_{n-1}(u)\geq 0\,,$$ with equality if and only if $(u-y_0)/|\mc V_n(u)|^{1/n}\in \Mob(\mb S^{n-1})$ for some $y_0\in \R^n$.
\end{proposition}

Our first theorem is a novel compactness result for sequences of maps with vanishing deficit, up to normalization and \textit{extraction of a single bubble,} \textit{i.e.}, up to precomposition with M\"obius transformations:

\begin{theorem}[Compactness]\label{thm:compactintro}
Let $(u_j)_{j\in \N}\subset W^{1,n-1}(\mb S^{n-1};\R^n)$ be a sequence such that 
\begin{equation}\label{eq:minim_seq}
\lim_{j\to \infty}\mc E_{n-1}(u_j)
 =
  0\,.
\end{equation}
Up to a non-relabeled subsequence, there exist $(\phi_j)_{j\in \N} \subset \Mob(\mb S^{n-1})$ and  $O\in \textup O(n)$  such that 
\begin{equation}\label{eq:compactness}
\frac{u_j  \circ \phi_j-\fint_{\mb S^{n-1}} u_j \circ \phi_j }{ |\mc V_n(u_j)|^{1/n}} \rightarrow O \id_{\S} \ \text{ strongly in } W^{1,n-1}(\mb S^{n-1};\R^n)\,.
\end{equation}
\end{theorem}

We note that if $u$ is $\S$-valued, then the conclusion of Theorem \ref{thm:compactintro} is a simple consequence of  a topological argument that allows us to pick $(\phi_j)_{j\in \N} \subset \Mob(\mb S^{n-1})$ so that $\fint_{\mb S^{n-1}} u_j \circ \phi_j=0$. This normalization, combined with the sharp Poincar\'e inequality on ${\mb S^{n-1}}$ and  the convexity of $\mc D_{n-1}$, easily yields the conclusion, see \cite[Lemma A.3]{zemas2022rigidity} and \cite[Lemma 2.4]{GuerraLamyZemas2023}. The main difficulty in proving Theorem \ref{thm:compactintro} thus resides in 
the lack of constraint on the image of $u$.

For $n=3$, Theorem \ref{thm:compactintro} was already known (cf.\ \cite[Lemma 2.1]{Caldiroli2006}), and it can also be retrieved as  a consequence of the seminal bubbling analysis of Brezis--Coron \cite[Theorem 0.3]{Brezis1985} and Struwe \cite[Proposition 3.7]{Struwe1985} for the $H$-system. Their tools, which we will explain in more detail below, do not apply in higher dimensions, where a new strategy is needed.

Having a \textit{qualitative} compactness result available,  one can hope for a \textit{quantitative} result. This is precisely the content of our main theorem.

\begin{theorem}[Optimal quantitative stability]\label{thm:qttiveintro}
There exists a  dimensional  constant $C_n>0$ such that, for all maps $u\in W^{1,n-1}(\mb S^{n-1};\R^n)$, we have
\begin{equation}\label{eq:qttive}
\inf_{\phi \in \Mob(\mb S^{n-1})} \fint_{\mb S^{n-1}} \bigg| \frac{\n_T u}{ |\mc V_n(u)|^{1/n}} - \n_T \phi\bigg|^{n-1}\, \mathrm{d}\haus \leq C_n \mc E_{n-1}(u)\,.
\end{equation}
\end{theorem}

For $n=3$, Theorem \ref{thm:qttiveintro} was proved in \cite[Theorem 1.4]{zemas2022rigidity}. 
In fact, in the same work a $W^{1,\infty}$-local version of  the theorem  
was established also in dimensions $n\geq 4$ (cf.\ Corollary 1.6 therein). In particular, 
Theorem \ref{thm:qttiveintro} provides the optimal quantitative version of these results in any dimension, global in $W^{1,n-1}(\S;\R^n)$ and with an explicit dilation factor of $u$. 

For sphere-valued maps $u\in W^{1,n-1}(\mb S^{n-1};\mb S^{n-1})$ we have 
$\mc V_n(u) = \deg(u)$.  In this case, one can consider the \textit{purely  conformal deficit}
$$\delta_{n-1}(u) :=\mc D_{n-1}(u)-1.$$  For maps of degree $\pm 1$ and whenever  $0\leq \delta_{n-1}(u)\ll 1$, we have
\begin{equation*}
\label{eq:deficit_R_n_def_on_S}	
\mc E_{n-1}(u)=(1+\delta_{n-1}(u))^{\frac{n}{n-1}}-1\lesssim \delta_{n-1}(u)\,,
\end{equation*}
and so Theorem \ref{thm:qttiveintro} immediately yields the following:

\begin{corollary}[Stability for sphere-valued maps]
\label{cor:sphere-val}
There is a  dimensional  constant $C'_n>0$ such that, for all $u\in W^{1,n-1}(\S;\S)$ with $|\tp{deg}(u)|=1$,  we have
\begin{equation}
\label{eq:sphereval}
\inf_{\phi \in \Mob(\mb S^{n-1})} \fint_{\mb S^{n-1}} | \n_T u- \n_T \phi|^{n-1}\, \mathrm{d}\haus \leq C'_n \delta_{n-1}(u)\,.
\end{equation}
\end{corollary}

Previously, we had proved Corollary~\ref{cor:sphere-val}  in \cite{GuerraLamyZemas2023}, although with a different proof. See also \cite{bernand2019quantitative,hirsch2022mobius2,topping2020,rupflin} for the case $n=3$.

\begin{remark}[Optimality]
As shown in \cite[Appendix A]{GuerraLamyZemas2023}, estimate \eqref{eq:sphereval} is optimal, in the sense that on its right-hand side the deficit $\delta_{n-1}(u)$ cannot be replaced with $\delta_{n-1}(u)^\beta$ for some $\beta>1.$  In particular, this shows that Theorem \ref{thm:qttiveintro} is also optimal, even if we restrict ourselves to $\S$-valued maps:  on the right-hand side of \eqref{eq:qttive} the deficit $\mc E_{n-1}(u)$ also cannot be raised to a higher power. Theorem \ref{thm:qttiveintro} is also optimal in terms of the distance used, in the sense that \eqref{eq:qttive} establishes stability for the gradients in the conformally-invariant norm: clearly one cannot hope for a better result without stronger a priori assumptions on the maps under consideration.
  \end{remark}

The problem of stability for Liouville's Theorem was first studied by Lavrentiev in  \cite{Lavrentiev1954}.  As mentioned above, a crucial feature of this problem resides in deciding how to measure  the deviation from conformality.  Lavrentiev considered maps $u\colon \Omega\to \Omega'$, where $\Omega, \Omega'\subset \S$ are proper domains, and he showed that if $u$ is $C^{1,\alpha}$ and has angle distortion uniformly close  to $1$, then $u$ is uniformly close to a M\"obius transformation.  There were many subsequent works in this line of investigation,  see for instance \cite{Friesecke2002, faraco2005geometric, bernand2019quantitative, hirsch2022mobius2, topping2020, rupflin}, the works  \cite{Reshetnyak1963,Reshetnyak1970} where a deficit somewhat reminiscent of $\mc E_{n-1}$ is considered,  and the monograph \cite{Reshetnyak1994}. Although these works consider both pointwise and integral measures of 
deviation from  conformality,  all of them study the stability of the M\"obius group among 
  maps for which the domain and the target have the same dimension. 
An important feature of this setting is that compactness properties of nearly conformal maps are a straightforward consequence of lower semicontinuity arguments; in other words, the analogue of Theorem \ref{thm:compactintro} in this setting is easily established. We note, however, that subtler issues arise if one works in Sobolev spaces with lower integrability \cite{Yan1996,Yan1998,Muller1999f}.

\subsection{Proof of Theorem \ref{thm:compactintro} and relation to \texorpdfstring{$H$}{H}-systems.}
 
Energies similar to $\mc E_{n-1}$  appear in the works of Heinz, Hildebrandt, and Wente,  among 
others, in relation to the Plateau problem for surfaces with \textit{prescribed mean curvature}, see \textit{e.g.}\ the survey \cite{Bethuel2002} and the references therein. This perspective is quite useful in our work, as we now explain. 
 
In the setting of Theorem \ref{thm:compactintro},  by using the invariances of the deficit and after identifying $ \overline{\R^{n-1}}:=\R^{n-1}\cup\{\infty\}$ with $\S$  via stereographic projection, it suffices to consider sequences $(u_j)_{ j\in \N }\subset W^{1,n-1}(\R^{n-1};\R^n)$ normalized so that  $\mc V_n(u_j)=1$. We are interested in sequences satisfying the minimizing property \eqref{eq:minim_seq}. By Ekeland's Variational Principle \cite{Ekeland1974},  we can move our attention to  such sequences  that further  satisfy a perturbed form of the Euler--Lagrange system for $\mc E_{n-1}$, which simply means that
\begin{equation}
\label{eq:approxHsystem}
\Delta_{n-1} u_j + H_j \p_{x_1} u_j \wedge \dots \wedge \p_{x_{n-1}} u_j = g_j\,, \qquad
\begin{cases} 
\sup_{j\in \N}\|\n u_j\|_{L^{n-1}(\R^{n-1})}
<+\infty\,,\\[5pt]
g_j \to 0 \text{  in } (W^{1,{n-1}})^*(\R^{n-1};\R^n)\,,
\end{cases}
\end{equation}
where $(H_j)_{ j\in \N }\subset \R_+$ is  an appropriate  bounded sequence and
$\Delta_p u :=\ddiv(|\n u|^{p-2} \n u)$ is the $p$-Laplace operator. Condition \eqref{eq:approxHsystem} asserts that $(u_j)_{j\in \N}$ is a \textit{Palais--Smale sequence} \cite[Section~II.2]{Struwe2008}. Up to a subsequence, we have $H_j\to H$ for some $H$ in $\R$, 
 $\nabla u_j\rightharpoonup \nabla u$ weakly in $L^{n-1}(\R^{n-1})$,
  and $u$ solves the $H$-system
\begin{equation}
\label{eq:H}
\Delta_{n-1} u + H\, \p_{x_1} u \wedge \dots \wedge \p_{x_{n-1}} u = 0 \,,
\end{equation}
cf. \cite[Theorem 1.1]{Strzelecki2004}. The interpretation of this system is that any weakly conformal solution  to \eqref{eq:H} is a (possibly branched) parametrization of a hypersurface with mean curvature equal to $H/(n-1)^{\frac{n-1}{2}}$.

In order to prove Theorem \ref{thm:compactintro} we need to show that $(\n u_j)_{j\in \N}\subset L^{n-1}(\R^{n-1})$ is strongly compact, at least up to pre-composition with  M\"obius maps. By a relatively standard argument, cf.\ \cite{Strzelecki2004},  \eqref{eq:approxHsystem} implies the subcritical convergence
\begin{equation}
\label{eq:subcritcompactness}
\n u_j \to \n u \text{ strongly in } L^p(\R^{n-1}) \ \ \text{for all }  1\leq  p<n-1\,,
\end{equation}
for some limit map $u \in W^{1,n-1}(\R^{n-1};\R^n)$. 
In \eqref{eq:approxHsystem} the perturbations $(g_j)_{j\in \N}$ are  vanishing  
only in an \textit{energy-critical space} and, as a consequence, strong compactness fails in general: by \eqref{eq:subcritcompactness}, this lack of compactness is due to the presence of \textit{concentration effects}.  Our task is to show that these effects result solely from the action of the non-compact group $\Mob(\S)$ on the domain.

\subsubsection{Bubbling analysis for $n=3$}

When $n=3$, Brezis--Coron \cite{Brezis1985} and Struwe \cite{Struwe1985} proved in two independent works that the possible concentrations in Palais--Smale sequences for the $H$-system can be characterized very precisely.  This characterization relies on the classification of the global solutions of \eqref{eq:H}, \textit{i.e.}, finite-energy solutions defined over $\R^{2}$: such solutions  turn out to be 
rational functions in the complex variable, see  \cite[Lemma 0.1]{Brezis1985}.  They then show that the sequence  $(u_j)_{j\in \N}$ possibly  concentrates only on a finite number of points, and moreover that at each concentration point it behaves exactly like a concentrating sequence of global solutions, usually known as a \textit{bubble}. 

Besides the classification of global solutions, the analysis of \cite{Brezis1985,Struwe1985} relies on two other ingredients. The first such ingredient is the special  \textit{multilinear structure} of the semilinear term in the $H$-system, which guarantees that the difference $u_j-u$ solves a system very similar to \eqref{eq:approxHsystem}; that this structure is crucial can be seen, for instance,  from the fact that there is no bubbling theorem for Palais--Smale sequences in the closely-related case of harmonic maps into spheres \cite{Parker1996, LinWang}. Its multilinear structure also means that the nonlinearity in the $H$-system lies in the Hardy space $\mathscr H^1(\R^{2};\R^3)$, which is a subspace of $L^1(\R^{2};\R^3)$ suited for harmonic analysis \cite{Coifman1993}. The second ingredient is the Wente-type result
\begin{equation}
\label{eq:wente}
\Delta u \in \mathscr H^{1}(\R^2) \quad \implies \quad u \in C^0(\R^2)\,,
\end{equation}
with a corresponding estimate. Such an estimate  leads  to the strong convergence
\begin{equation}
\label{eq:Linfty2d}
u_j\to u\ \ \text{ in } L^\infty (\R^2) \,,
\end{equation}
from which the above characterization of concentrations follows. 

\subsubsection{The higher dimensional case $n>3$: difficulties}

In the last decades there have been many works studying sequences of approximate solutions to conformally-invariant variational problems in higher dimensions, see \textit{e.g.}\
\cite{Hardt1994a,Toro1995a,Mou1996a,Courilleau2001,Wang2002,Strzelecki2004,Wang2005,Miskiewicz2016}.  
However,   neither  these works nor their methods can lead to a characterization of concentrations in Palais--Smale sequences, as in the case $n=3$. 
In fact,  currently it is not even known that the defect measure associated to the sequence $(|\n u_j|^{n-1})_{j\in \N}$ is concentrated on a finite number of points. We now briefly detail the difficulties associated with the higher dimensional case.

Although the $H$-system \eqref{eq:H} has a multilinear structure in all dimensions,  as soon as $n> 3$  the analogue of \eqref{eq:wente} does not hold, and we have
\begin{equation}
\label{eq:nowente}
\Delta_{n-1} u \in \mathscr H^{1}(\R^{n-1}) \quad \centernot \implies \quad u \in C^0(\R^{n-1})\,,
\end{equation}
see \cite{Firoozye1995,Iwaniec2007a} for counter-examples and an optimal result in this direction.  As far as we know, the only known way of proving strong convergence in $L^{n-1}$ of Palais--Smale sequences is to show the analogue of \eqref{eq:Linfty2d}, \textit{i.e.}, to show that
\begin{equation}
\label{eq:Linfty}
u_j\to u\ \  \text{ in } L^\infty (\R^{n-1}) \,,
\end{equation}
cf.\ \cite[Proposition 2.5]{Mou1996a} for a very general result in this direction. Nonetheless, due to \eqref{eq:nowente},  obtaining \eqref{eq:Linfty} for $n>3$ has remained elusive. 

A different, deeper difference between the cases $n=3$ and $n>3$ is that, in the latter, solutions to the $H$-system \eqref{eq:H} are \textit{not  known to be  necessarily weakly conformal} and, for non-conformal solutions, the geometric interpretation of this system as describing parametrized constant mean curvature hypersurfaces is lost. In particular, there is no characterization of global solutions to \eqref{eq:H}  for $n>3$.

The above difficulties are closely related to challenging regularity problems: the regularity theory of conformally invariant systems of quadratic growth in the plane is well-understood, after Rivi\`ere's foundational work  \cite{Riviere2007}, but extending these results to higher dimensions remains a difficult open problem.  We refer the reader to the survey \cite{Schikorra2017} for an extensive list of references on this problem, as well as  \cite{Wang1999,Duzaar2000a,Fusco2018,Martino2023} for some partial results.

\subsubsection{The higher dimensional case $n>3$: proof sketch}

We now sketch the proof strategy that we pursue. Instead of simply looking at Palais--Smale sequences as in \eqref{eq:approxHsystem},  in order to prove Theorem \ref{thm:compactintro} we will use the full strength of our minimality assumption \eqref{eq:minim_seq}. This approach is inspired by the work of Caldiroli--Musina who, for $n=3$, gave in \cite[Lemma~2.1]{Caldiroli2006}  a direct variational proof of Theorem \ref{thm:compactintro}, see also \cite[Lemma~2.1]{Cerami1984} and \cite[Lemma~2.2]{Struwe1985} in a different context. The proof of \cite{Caldiroli2006}, in particular, bypasses completely the perturbed Euler--Lagrange system  \eqref{eq:approxHsystem}. 

In our setting, we are forced to pass through the analysis of \eqref{eq:approxHsystem} in order to obtain the subcritical compactness \eqref{eq:subcritcompactness}. This compactness allows us to obtain the expansion
\begin{align*}
 \mc D_{n-1}(u_j) & =\mc D_{n-1}(u)+\mc D_{n-1}(u_j-u)+o(1)
\end{align*}
for the conformal Dirichlet energy, as a consequence of the Brezis--Lieb lemma. We emphasize that,  when $n=3$, this expansion is a direct consequence of the weak convergence 
$
\nabla u_j\rightharpoonup \nabla u $ in $ L^2(\R^2) $, since in this case  the conformal energy is quadratic, but in higher dimensions the expansion is not obvious.  The weak convergence alone also leads to a similar expansion for the volume,  namely
$$
\mc V_n(u_j)=\mc V_n(u)+  \mc V_n(u_j-u)+o(1)\,,
$$
thanks to the weak continuity properties of the minors.

Plugging the above expansions into \eqref{eq:minim_seq}
and using the non-negativity of the deficit ensured by Proposition~\ref{prop:isopintro} together with a convexity argument,
leads to the desired strong convergence  $\n u_j \to\n u$ in $L^{n-1}(\R^{n-1})$, provided the weak limit $u$ is non-constant.

In order to ensure that the limit $u$ is not a constant map, one needs to construct maps  $(\phi_j)_{j\in \N}\subset  \Mob(\S)$ which spread out the energy of the modified sequence $(u_j\circ \phi_j)_{j\in \N}$,  
guaranteeing that the  limiting map is non-constant.  These are precisely the maps in \eqref{eq:compactness}.


We note that the above argument relies crucially on the minimality assumption: unlike the case $n=3$,  in general we cannot prove strong convergence in the critical space $L^{n-1}(\R^{n-1})$ if \eqref{eq:minim_seq}  is replaced with
$$\lim_{j\to \infty} \mc E_{n-1}(u_j) 
\leq
 \e\,,$$
where $\e>0$ is arbitrarily small.



\subsection{
Quantitative stability and the proof of Theorem~\ref{thm:qttiveintro}
}\label{quantitative_description}

We now turn to the proof of  the main quantitative result of this work. 
We first note that,   by a standard contradiction argument and the compactness result of Theorem \ref{thm:compactintro}, it suffices to prove Theorem \ref{thm:qttiveintro} for maps $u\in W^{1,n-1}(\S;\R^n)$ such that
\begin{equation*}
\E_{n-1}(u)\ll 1\,, \qquad \fint_{\S}u=0\,, \qquad \big\|u-\mathrm{id}_{\S}\big\|_{W^{1,n-1}(\S)}\ll 1\,.
\end{equation*}
Ultimately,  as in many stability problems, the nonlinear estimate of Theorem \ref{thm:qttiveintro} relies on a suitable \textit{linear inequality}.  After the above reduction,  if one rescales $u$ appropriately and sets $$w:=u-\mathrm{id}_{\S}\,,$$ 
then the crucial inequality concerns the \textit{coercivity} of the second variation of the deficit at the identity,
$$Q_n(w) := \mc E''_{n-1}(\id_{\S})[w,w]\,,$$
see \eqref{eq:Qn} below for the explicit formula of this quadratic form. 
Note that, due to the invariance of the deficit under the action of $\Mob(\S)$, one can only hope for coercivity of $Q_n$ in the $W^{1,2}$-orthogonal complement of the \textit{Lie algebra of infinitesimal M\"obius transformations}, which we denote by $\mob(\S)$. One of the main results in \cite{zemas2022rigidity} asserts precisely that this coercivity holds, see already \eqref{eq:lin_stab}, and the proof relies on a fine interplay between the Fourier decomposition of a $W^{1,2}(\S;\R^n)$-vector field into $\R^n$-valued spherical harmonics and the properties of the quadratic form $\mc V_n''(\id_{\S})$.

When $n=3$, it is not difficult to pass from a linear estimate to a nonlinear one. To do so,  in \cite{zemas2022rigidity} the authors consider a formal Taylor expansion that turns out to be of the form
\begin{equation*}
\mathcal{E}_2(u)=Q_3(w)+o\left(\fint_{\mathbb{S}^2} |\t w|^2\right)\,
\end{equation*}
and in order to conclude,  by the coercivity of $Q_3$ in the orthogonal complement of $\mob( \mathbb{S}^2)$, it suffices to choose $\phi\in  \Mob( \mathbb{S}^2)$ such that $(u\circ \phi - \id_{ \mathbb{S}^2}) \bot \mob( \mathbb{S}^2)$. This is achieved through a topological argument involving the Inverse Function Theorem. \EEE
 
For $n\geq 4$, a similar formal Taylor expansion of the deficit in \eqref{eq:deficit} only gives 
\begin{equation*}
\mathcal{E}_{n-1}(u)=Q_n(w)+\mathcal{O}\left(\fint_{\S} |\t w|^3\right)\,.
\end{equation*} 
Since the higher order terms are now cubic in $\t w$, the linear estimate alone implies the nonlinear estimate
 \eqref{eq:qttive}
 only in a $W^{1,\infty}$-neighborhood of
 $\mathrm{id}_{\S}$, cf.\ \cite[Corollary 1.6]{zemas2022rigidity}.  Generically, it  does not provide a control in the desired optimal conformally invariant $W^{1,n-1}$-norm. 
Moreover, even a suboptimal estimate with an $L^2$-norm in the left hand side of \eqref{eq:qttive} would be hindered by 
 the presence of the higher-order cubic terms in the expansion.

To overcome these issues, we take inspiration from the work of Figalli--Zhang \cite{figalli2022sobolev},
where they prove a sharp quantitative version of the Sobolev inequality  in $\R^n$.
There, the same difficulty is present: the linearized estimate only provides subcritical control.
In order to overcome it, instead of the exact expansion of the $p$-Dirichlet energy, they use a Taylor-type lower inequality, which 
can be thought of as an interpolation between the  first up to quadratic  terms and the highest-order $p$-term. It includes a term controlling the critical norm,
but the quadratic terms  have to be replaced by a lower, nonquadratic quantity.
The core insight allowing to conclude is that, for small perturbations, this lower nonquadratic quantity still provides a control similar to the quadratic terms of the exact expansion.

In our case, we combine the same Taylor-type inequality for the conformal $(n-1)$-Dirichlet energy in \eqref{eq:Dirichlet} (see Lemma \ref{lem_Dirichlet_exp}) with a careful analysis of the terms which arise from expanding the volume $\mathcal V_n(\id_\S +w)$.
This eventually leads to a nonquadratic quantity $\widetilde Q_n(w)$ 
which is similar to but lower than the quadratic form $Q_n(w)$, and to an estimate of the form
\begin{align}
\begin{split}
\mathcal E_{n-1}(\mathrm{id}_{\S}+w)
&
\gtrsim 
\fint_{\S}|\n_T w|^{n-1}\,\mathrm{d}\haus 
\\
&\quad
+ \widetilde Q_n(w) - B[w,\nabla_T w]-\|\nabla_T w\|_{L^2(\S)}^{2+\alpha}\,,
\end{split}
\label{eq:quantitative_proof}
\end{align}
for some bilinear form $B$ and some $\alpha>0$.
The proof is completed by showing that the last line in \eqref{eq:quantitative_proof} is nonnegative for $w\bot\ \mob(\S)$.
We obtain this as a consequence of the linear stability estimate
for $Q_n(w)$ from \cite[Theorem 1.5]{zemas2022rigidity} and a compactness argument similar to \cite[Proposition~3.8]{figalli2022sobolev}.

It is worth noting that for the proof in \cite{figalli2022sobolev},
the nonquadratic quantity analogous to $\tilde Q_n(w)$  only needs to control terms without derivatives.
This is due to the fact that the denominator in the deficit of the Sobolev inequality contains no derivatives,
and is a key feature for the compactness argument in \cite[Proposition~3.8(iii)]{figalli2022sobolev}. Instead,  the denominator in our deficit \eqref{eq:deficit} does contain derivatives,
and the highest order term cannot absorb all other gradient terms.
The precise structure of these gradient terms in the last line of \eqref{eq:quantitative_proof} is crucial in order to make our compactness argument work.

%

\subsection{Outline of the paper} After fixing some notation in the next subsection, in Section \ref{sec:prelims} we collect some basic facts about the notion of topological degree for Sobolev mappings (adapted to our setting), the volume functional $\mc V_n$, and also the geometry of M\"obius transformations of $\S$. In the final subsection \ref{subsec: 1st_2nd_variation} therein, we derive the first variation of the deficit $\mc E_{n-1}$ introduced in \eqref{eq:deficit}, discuss some standard properties of its critical points, and we also include the linear stability estimate obtained originally in \cite[Theorem 1.5]{zemas2022rigidity}, which as described will be crucial in obtaining the nonlinear estimate. Sections \ref{sec:param_conformal}, \ref{sec:compactness} and \ref{sec:quantitative_stability} are devoted to the proofs of Proposition \ref{prop:isopintro} and Theorems \ref{thm:compactintro}, \ref{thm:qttiveintro} respectively, the proofs being split in several different steps of possibly independent interest in each case.


\subsection{Notation}\label{Notation}

The following standard notation will be adopted throughout the paper.\\[-0.5cm]

{\renewcommand{\cellalign}{cl}
\setlength{\extrarowheight}{15pt}
\begin{longtable}[l]{l l }
$\{e_i\}_{i=1}^n,\  \langle \cdot, \cdot \rangle,\ |\cdot| $ & the standard orthonormal basis, inner product, norm in $\R^n$\\[-3pt]
$A:B,\ |A|$ & the Frobenius inner product and norm  for matrices in $\R^{n\times m}$\\[-3pt]
$A^t$ & the transpose of a matrix  or the adjoint of a linear map \\[-1pt]
$a\otimes b$ &  the matrix with entries $(a\otimes b)_{ij}:=a_ib_j$, where $a\in \R^n, b\in \R^m$\\[-3pt]
$\S, \B$ &  the subsets $\{x\in \R^n\colon |x|=1\}$ and $\{x\in \R^n\colon |x|\leq 1\}$ of $\R^n$\\[-3pt]
$\omega_n$ & the Euclidean volume of the unit ball in $\R^n$\\[-3pt]
$\{\tau_1,\dots,\tau_{n-1}\}$ & \makecell{a positively oriented orthonormal frame for $T_x\S$ so that, for all \\  $x\in \S$, $\{\tau_1(x),\cdots,\tau_{n-1}(x),x\}$ is a positively oriented frame of $\R^n$}\\[-3pt]
$\mathcal{H}^k, \L^d$ & the $k$-dimensional Hausdorff and $d$-dimensional Lebesgue measures\\[-3pt]
$\textup O(n),\ \textup{SO}(n)$ & the orthogonal and special orthogonal groups of $\R^n$\\[-3pt]
$I_n,\ I_x$ & the identity matrix in $\R^n$ and the identity transformation on $T_x\S$ \\[-3pt]
$\t u$ &\makecell{ the tangential gradient of $u\colon \S\to \R^n$, represented in local\\ coordinates by the $n\times(n-1)$ matrix with entries $(\t u)_{lm}=\partial_{\tau_m}u^l$}\\[-3pt]
$J(u)$ &\makecell{  the tangential Jacobian of $u$, $J(u):=\partial_{\tau_1} u \wedge \dots \wedge \partial_{\tau_{n-1}} u$ } \\[-3pt]
$\mathrm{id}_{\S}, P_T$ & the identity map on $\S$ and its gradient $\t\mathrm{id}_{\S}$\\[-3pt]
$\mathrm{div}_{\S}$, $\Delta_{\S}$& the tangential divergence and the Laplace-Beltrami operator on $\S$\\[-3pt]
$C^0_b(\R^n;\R)$ & the space of continuous bounded functions from $\R^n$ to $\R$\\[-3pt]
$C^k,\ C^{k,\alpha}$ & \makecell{the spaces of $k$-times continuously differentiable  maps
 and of maps\\ with $\alpha$-Hölder continuous derivatives of order $k$}\\[-3pt]
$L^p, W^{s,p}$ & \makecell{ the standard Lebesque or fractional Sobolev spaces on $\S$,  where\\ the norms are taken with respect to the normalized $\mathcal{H}^{n-1}$-measure}\\[1pt]
$\lesssim_{M_1,M_2,\dots}, \sim_{M_1,M_2,\dots}$ & \makecell{ the corresponding inequality or equality holds up to a constant multiplicative factor \\
 that  depends only on the parameters $M_1, M_2,\dots$, or only on the \\
 dimension  when the subscripts are absent}\\[-3pt]
$c,C>0$ & \makecell{constants whose value is allowed to vary from line to line but which\\ depend only on the dimension}
\\[-3pt]
$ \mathbf 1_D$ &  the indicator function of a set $D\subset \R^n$
\end{longtable}
}

\section{Preliminaries}\label{sec:prelims}

\subsection{Multilinear algebra and calculus on the sphere}
In what follows, we always suppose that $n\geq 3$. Given vectors $v_1,\dots, v_{n-1}\in \R^n$, we define their \textit{vector product} $v_1\wedge \dots \wedge v_{n-1}\in \R^n$ (identified with its Hodge dual) as the unique $v\in \R^n$ such that
$$\langle v,w\rangle = \det(v_1,\dots, v_{n-1},w) \quad \text{for all } w\in \R^n,$$
where the expression on the right-hand side is the determinant of the $n\times n$ matrix with columns $v_1, \dots, v_{n-1}, w$.
The vector product is intimately connected to the cofactor matrix, which for $A\in \R^{n\times n}$ satisfies the identity 
\begin{equation}
\label{eq:defcof}
(\cof A)^t A = (\det A)\, I_n\,.
\end{equation} 
In particular, we have
\begin{equation}
\label{eq:cof}
(\cof A)(v_1\wedge \dots \wedge v_{n-1})=Av_1 \wedge \dots \wedge A v_{n-1} \ \forall v_1,\dots, v_{n-1}\in \R^n\,.
\end{equation}
With these notations, for the local orthonormal frame $\{\tau_1,\dots, \tau_{n-1}\}$ of $T_x\S$,
we have
\begin{equation} \label{eq:nu}
\tau_1\wedge \dots \wedge \tau_{n-1} = x\,.
\end{equation}

In order to facilitate many calculations in the paper, it will be useful to consider extensions of maps $u\in W^{1,n-1}(\mb S^{n-1};\R^n)$.
To that end, we note that $W^{1,n-1}(\S;\R^n)\subset W^{1-\frac{1}{n},n}(\S;\R^n)$ (see \textit{e.g.} \cite[Theorem~17.85]{Leoni2017}), 
and the space $W^{1-\frac{1}{n},n}(\S;\R^n)$
is
the image of $W^{1,n}(\B;\R^n)$ under the trace operator \cite[Theorems~18.27-28]{Leoni2017}.
In particular, for any $u\in W^{1,n-1}(\mb S^{n-1};\R^n)$,  there exists $U\in W^{1,n}(\B;\R^n)$ such that
\begin{align}\label{eq:Sobolev_bound_for_U}
U|_{\S}=u,
\quad
\text{and }
\quad
\|U\|_{W^{1,n}(\B)}\lesssim \|u\|_{W^{1,n-1}(\S)}\,.
\end{align} 
Applying \eqref{eq:cof} with $A=\n U$ and using also \eqref{eq:nu}, we see that for $\H^{n-1}$-a.e. $x\in \S$,
\begin{equation}\label{eq:cofeqwedge}
\cof(\n U) x = (\n U) \tau_1 \wedge\dots \wedge (\n U) \tau_{n-1}  = \bigwedge_{i=1}^{n-1} \p_{\tau_i} u =: J(u)\,.
\end{equation}
Let us also note here, for later use, the Hadamard-type inequality
\begin{equation}\label{eq:Hadamard}
|J(u)| =|\cof(\nabla U)x | \leq \left(\frac{|\n_T u|^2}{n-1}\right)^{\frac{n-1}{2}} \ \H^{n-1}\tp{-a.e.\ on \ } \S\,,
\end{equation}
which is a straightforward consequence of the Cauchy--Schwarz and arithmetic mean--geometric mean inequalities.

\subsection{On the topological degree and the volume functional}
\label{sec:degree}

The reader may have noticed that it is not immediately clear how to make sense of \eqref{eq:vol} for maps $u\in W^{1,n-1}(\S;\R^n)$, although 
 the integral therein 
is clearly meaningful if in addition $u$ is essentially bounded.  In this subsection we 
recall
 how to interpret \eqref{eq:vol} and, on the way to do so, we will gather some useful properties of the volume functional and of related objects, in particular the notion of (local) topological degree.  In fact, the issue of giving a meaning to \eqref{eq:vol} for maps in the critical space is closely related to the theory of 
 VMO-degree developed by Brezis and Nirenberg \cite{Brezis1995,Brezis1996}.

We begin by stating a standard identity relating bulk integrals of $\det \n U$, with $U$ as in \eqref{eq:Sobolev_bound_for_U}, 
 with boundary integrals of $J(u)$; this identity follows from \textit{Piola's identity}
 \begin{equation}\label{eq:Piola}
 \mathrm{div}\cof\nabla U=0 \ \mathrm{ in } \ \mathcal{D}'(\mb B^n)\,,
 \end{equation}  
the chain rule and \eqref{eq:defcof}, see \textit{e.g.}\ \cite[Lemma 3.14]{Hencl2014a} for further details.

\begin{lemma}\label{lemma:IBPdet}
For all $U\in W^{1,n}(\mb B^n;\R^n)$ with $u:=U|_{\S}\in W^{1,n-1}(\S;\R^n)$ 
 as in \eqref{eq:Sobolev_bound_for_U},  we have
\begin{equation}\label{eq:bulk_surface}
\int_{\mb B^n} (\ddiv V)(U(x)) \det \n U(x) \d x = \int_{\mb S^{n-1}} \langle V (U(x)) ,  \cof(\n U) x \rangle \d \H^{n-1}\,,
\end{equation}
whenever  $V\in C^1_c(\R^n; \R^n)\,.$
\end{lemma}

If $U\in  (  W^{1,n}\cap L^\infty ) (\B; \R^{n})$ then,  by taking suitable approximations of the vector field $V(x)=x$, Lemma \ref{lemma:IBPdet} combined with \eqref{eq:cofeqwedge} yield the identity
\begin{equation}\label{eq:bulkbdry}
\fint_{\mb B^n} \det \n U \d x = \fint_{\S} \langle u, J(u)\rangle \d \H^{n-1}= \mc V_n(u)\,,
\end{equation}
where the last equality is simply definition \eqref{eq:vol}.  
The integral on the left-hand side is meaningful for $U\in W^{1,n}(\B;\R^n)$, and hence \eqref{eq:bulkbdry} can be used to define the volume  $\mc V_n(u)$ for $u\in W^{1,n-1}(\S;\R^n)$.
In fact, the same reasoning allows  to extend $\mc V_n$ to the trace space of $W^{1,n}(\B;\R^n)$
(similar considerations are exploited in \cite[Section~2]{BrezisNguyen2011} to study fine properties of Jacobian determinants).
Note that, since the extension $U\in W^{1,n}(\B;\R^n)$ depends continuously on $u \in W^{1,n-1}(\S;\R^n)$,  see 
\eqref{eq:Sobolev_bound_for_U},
this definition of $\mathcal V_n(u)$ coincides with the one used \textit{e.g.} in \cite[Corollary~3.5]{Mou1996} by combining compensation properties of Jacobians \cite{Coifman1993} and duality properties of Hardy spaces \cite{FeffermanStein1972}.

We now discuss the notion of local topological degree, which is closely related to the volume functional. Let us first explain how to define the \textit{local degree} for maps $U\in W^{1,n}(\B;\R^n)$.  First recall that, in the case of $U\in C^1(\overline\B;\R^n)$, for every \textit{regular value} $y$ of $U$ such that $y\notin U(\S)$, the \textit{local degree} at $y$ is defined as 
\begin{equation*}
\mathrm{deg}(U, \B; y):=\sum_{z\in U^{-1}(y)} \mathrm{sgn}(\mathrm{det}\nabla U)(z)\,.
\end{equation*}
By the fact that $\L^n(U(\S))=0$ and Sard's theorem, $\mathrm{deg}(U, \B; y)$ is well-defined for $\L^n$-a.e. $y\in \R^n$. In this case,  the generalized area formula immediately yields the identity
\begin{equation}
\label{eq:deg=jac}
\int_{\mb B^n} \eta(U(x)) \det \n U(x) \d x= \int_{\R^n} \eta(y) \deg(U,\mb B^n;y) \d y,
\end{equation}
whenever $\eta \in C_b^0(\R^n;\R)$, see \cite[Corollary 1]{Sverak1988}.

We now want to define the local degree for a general map $U\in W^{1,n}(\B;\R^n)$
with trace $u:=U|_{\S}\in W^{1,n-1}(\S;\R^n)$,  
 precisely in such a way that \eqref{eq:deg=jac} stays true,
and so we will essentially use this identity \textit{as a definition}. 
To be precise, consider the distribution $T_U\in \mathcal D'(\R^n)$ defined as
\begin{align*}
 ( T_U,\varphi ):= \int_{\B} \varphi(U(x)) \det\nabla U(x)\, \mathrm{d}x\,,
\end{align*}
 where $\varphi \in C^\infty_c(\R^n)$  and $(\cdot,\cdot)$ stands for the duality pairing.
 This distribution can be 
 identified with a finite Radon measure, but we have in fact:

\begin{lemma}\label{lem:TUBV}
For any $U\in W^{1,n}(\B;\R^n)$
with trace $u:=U|_{\S}\in W^{1,n-1}(\S;\R^n)$, 
we have  $T_U\in \tp{BV}(\R^n;\mathbb Z)$.
\end{lemma}
\begin{proof}
Consider an approximating sequence $(U_j)_{j\in \N}\subset C^\infty(\overline\B;\R^n)$ such that 
\begin{equation}\label{eq:approx}
U_j\rightarrow U\ \L^n\text{-a.e. in } \B 
\text{ and  strongly   in } W^{1,n}(\B;\R^n), \   
\sup_{j\in \N}\|u_j\|_{W^{1,n-1}(\S)}<+\infty\,,
\end{equation} 
where as usual $u_j := U_j|_\S$.
It suffices to show that
\begin{enumerate}
\item\label{it:wstarconv} $T_{U_j} \wstar T_U$ in $\mathcal D'(\R^n)$,
\item\label{it:equibdd} $(T_{U_j})_{j\in \N}$ is equibounded in $BV(\R^n)$,
\end{enumerate}
as then, by compactness, $T_U\in BV(\R^n;\Z)$, since $T_{U_j}\to T_U$ in  $L^1_\tp{loc}(\R^n)$ and $T_{U_j}$ is integer-valued 
 for every $j\in \N$.

We note that \ref{it:wstarconv} follows easily from the fact that $\det\nabla U_j\to \det\nabla U$ in $L^1(\B)$ and $U_j\to U$ $\L^n$-a.e. on $\B$. For \ref{it:equibdd},  by the smoothness of $(U_j)_{j\in \N}$ and \eqref{eq:approx}, we have
\begin{align*}
\sup_{j\in\N}\|T_{U_j}\|_{L^1(\R^n)}
&
\lesssim
   \sup_{j\in\N}\int_{\B} |\nabla U_j|^n\, \mathrm{d}x<+\infty\,,   
\end{align*}
and by \eqref{eq:bulk_surface}, \eqref{eq:Hadamard} and \eqref{eq:approx}, 
\begin{align*}
\sup_{j\in\N}|DT_{U_j}|(\R^n)&=\sup_{j\in\N}\sup_{\psi\in C^1_c(\R^n;\R^n)}\bigg\{\int_{\B} (\ddiv \psi)(U_j(x)) \det\nabla U_j(x)\, \mathrm{d}x\colon \|\psi\|_{L^\infty(\R^n)}\leq 1\bigg\}  \\
&= \sup_{j\in\N}\sup_{\psi\in C^1_c(\R^n;\R^n)} \bigg\{\int_{\mb S^{n-1}} \langle \psi\circ U_j,  \cof(\n U_j) x \rangle\, \d \H^{n-1}\colon \|\psi\|_{L^\infty(\R^n)}\leq 1\bigg\} \\
&\leq \sup_{j\in\N}\int_{\S} |J(u_j)|\,\d \H^{n-1} \\
&  \lesssim   \sup_{j\in\N}\int_{\S} |\nabla_T u_j|^{n-1}\,\d \H^{n-1}<+\infty\,,
\end{align*}
completing the proof.
\end{proof}

We now define the local degree of $U\in W^{1,n}(\B;\R^n)$ 
 with $U|_{\S}\in W^{1,n-1}(\S;\R^n)$, 
  as
\begin{equation}\label{eq:general_degreee}
\deg(U,\B;\cdot ):=T_U\,,
\end{equation}
and by Lemma \ref{lem:TUBV} we see that the local degree is essentially integer-valued and satisfies the identity \eqref{eq:deg=jac},  as the latter is valid along the sequence $(U_j)_{j\in \N}\subset C^\infty( \overline{\B} ;\R^n)$,   and $T_{U_j}\wstar T_U$ in $\tp{BV}_{\mathrm{loc}}(\R^n)$, as we saw in \ref{it:wstarconv}-\ref{it:equibdd} 
above. Summarizing this discussion, we have shown:

\begin{theorem}\label{thm:integral_degree_formula}
For any $U\in W^{1,n}(\mb B^n;\R^n)$ with trace $u:=U|_{\S}\in W^{1,n-1}(\S;\R^n)$ and the local degree being defined as in \eqref{eq:general_degreee}, 
the identity \eqref{eq:deg=jac} holds,  whenever  $\eta\in C_b^0(\R^n;\R)$.
\end{theorem}

\begin{remark}[Comparison with  VMO-degree]
Let $U\in W^{1,n}(\B;\R^n)$. Due to the inclusion $W^{1,n}(\B;\R^n)\subset\tp{VMO}(\B;\R^n)$, if 
$y\not \in \overline{ U(\S)}$
 then $\deg(U,\B;y)$ can be defined as in  \cite[Section II.2]{Brezis1996}, and it coincides with the limit of $ ( \deg(U_j,\B;y) )_{j\in \N} $ for a smooth approximating sequence; in particular, it also coincides with the notion of degree defined in \eqref{eq:general_degreee}. 
 But $\overline{U(\S)}$ may 
  generically 
   have positive 
   $\L^n$-measure, so this does not define $\deg(U,\B;\cdot)$ 
   $\L^n$-a.e.\ in $\R^n$. See also \cite[page 114]{Sverak1988} for related results and discussion.

In fact, under our additional assumption that $u:=U|_{\S}\in W^{1,n-1}(\S;\R^n)$,
the tools developed in \cite{Brezis1995,Brezis1996} do provide a pointwise characterization
of the local topological degree 
$\deg(U,\B;\cdot)$ 
$\L^n$-a.e.\ in $\R^n$.
Specifically, we show in Appendix~\ref{a:deg},
following ideas from \cite{JerrardSoner2002,CanevariOrlandi2019},
that for $\L^n$-a.e. $y\in\R^n$,
\begin{equation*}
u_y:=\frac{u-y}{|u-y|}\in W^{1,n-1}(\S;\S)
\ \
\text{and } \ \
\deg(U,\B;y)=\deg(u_y,\S;\S)\,,
\end{equation*}
where  $\deg(u_y ,\S;\S)$ is the $\mathrm{VMO}$-degree defined in \cite{Brezis1995}.
This generalizes \cite[Section II.4]{Brezis1996} where
the identity  $\deg(U,\B;y)=\deg(u_y,\S;\S)
$ is proved for all $y\notin\overline{U(\S)}$.
\end{remark}

\begin{remark}[Sphere-valued maps]\label{rmk:vol=deg}
The discussion of the previous remark, combined with \eqref{eq:bulkbdry}, also shows that, as claimed in the introduction,  if $u\in W^{1,n-1}(\S;\S)$ then \[\mc V_n(u) = \deg(u):=\deg(u,\S;\S)\] 
is the usual $\tp{VMO}$-degree.
\end{remark}

We conclude this subsection by recalling some useful facts from \cite[Section 2.3]{GuerraLamyZemas2023} regarding the expansion of $\mc V_n$ around $\mathrm{id}_{\S}$, as a polynomial in the derivatives of the perturbation. For convenience of the reader, we recall the notation used therein. 
Let $A\in \R^{n\times n}$ with its set of eigenvalues 
 (be them real or complex) 
  being labeled as $\{\mu_1,\dots,\mu_n\}$.   We then have  
\begin{align*}
\det(I_n+A)=1+\sum_{k=1}^n \sigma_k(A)\,,
\end{align*}
where $\sigma_k(A)$ denotes the $k$-th elementary symmetric polynomial in the eigenvalues of $A$:
\begin{equation}\label{eq:symmetric_polies}
\sigma_k(A):=\sum_{1\leq i_1<\dots<i_k\leq n}\mu_{i_1}\dots\mu_{i_k}\,.
\end{equation}
Note that the $k$-homogeneity of $\sigma_k$ implies the Euler identity
\begin{align*}
\sigma_k(A)=\frac 1k  \sigma_k'(A):A\,,
\end{align*}
where $\sigma_k'(A)\in\R^{n\times n}$ is the gradient of $\sigma_k$ with respect to the $A$-variable.
With this notation, we have the following:

\begin{lemma}\label{lemma_on_null_Lagrangians}
Let $w\in W^{1,n-1}(\S;\R^n)$. Then
\begin{equation}\label{eq:boundary_integrals}
\mc V_n(\mathrm{id}_{\S}+w)= 1+\sum_{k=1}^n \frac{n}{k}\fint_{\S}\langle w,[\sigma'_k(\nabla_T wP_T^t)]^tx\rangle\,\mathrm{d}\H^{n-1}\,,
\end{equation}
where $\sigma_k$ are as in \eqref{eq:symmetric_polies}.
Moreover, the first and last summands have the simple forms
\begin{align}
n \fint_\S \langle w,[\sigma'_1(\nabla_T wP_T^t)]^tx\rangle\,\mathrm{d}\H^{n-1}
& = n\fint_{\mathbb S^{n-1}} \langle w, x\rangle\,\mathrm{d}\H^{n-1}\,,
\label{eq:first_term_in_the_expansion}
\\
\fint_{\S}\langle w,[\sigma'_n(\nabla_T wP_T^t)]^tx\rangle\,\mathrm{d}\H^{n-1}
&=\fint_{\S}\big\langle w,J(w)\big\rangle\,\mathrm{d}\H^{n-1}=\mc V_n(w)\,,
\label{last_term_in_the_expansion}
\end{align}
and the intermediate summands for $k\in\lbrace 2,\ldots n-1\rbrace$ are estimated by
\begin{equation}\label{eq:trivial_estimate_for_intermediate_terms}
\frac{n}{k}\bigg|\fint_{\S}\langle w,[\sigma'_{  k }(\nabla_T wP_T^t)]^tx\rangle\,\mathrm{d}\H^{n-1}\bigg|
\leq C_n
\fint_{\S}|w||\nabla_Tw|^{k-1}\,\mathrm{d}\H^{n-1}\,,
\end{equation}
where $C_{n}>0$ is a constant depending only on $n$.
\end{lemma}

The proof of Lemma~\ref{lemma_on_null_Lagrangians} follows directly from \cite[Lemma 2.6]{GuerraLamyZemas2023} combined with identity \eqref{eq:bulkbdry}, and \cite[Remark~2.7]{GuerraLamyZemas2023}.

\subsection{M\"obius transformations}\label{sec:Mobius maps}

As mentioned in the Introduction,  
we define $\Mob(\mb S^{n-1})$ to be the group of M\"obius transformations of $\mb S^{n-1}$, 
\begin{equation}\label{def: Mobius_group} 
\Mob(\mb S^{n-1}) := \{O \phi_{\xi,\lambda}: O \in \tp O(n),\ \xi \in \S, \ \lambda>0\}\,,
\end{equation}
where 
\[
\phi_{\xi,\lambda}:=\sigma_{\xi}^{-1}\circ i_{\lambda}\circ\sigma_{\xi}\,,
\] 
with $\sigma_{\xi}\colon \S\to \overline{\R^{n-1}} $ the  stereographic projection from $-\xi\in \S$\,,  and $i_{\lambda}\colon \R^{n-1}\to \R^{n-1}$ the dilation by factor $\lambda>0$. 
Explicitly, we may write
\begin{align}\label{moebius_analytically}
\begin{split}	
\sigma_\xi(x)&=\frac{x-\langle x,\xi \rangle\xi}{1+\langle x,\xi\rangle}
\,, \\
\phi_{\xi,\lambda} (x)&= \frac{-\lambda^2 (1-\langle x,\xi\rangle )\xi + 2 \lambda(x-\langle x,\xi\rangle \xi) + (1+\langle x,\xi\rangle) \xi}{\lambda^2 (1-\langle x,\xi\rangle) + (1+ \langle x,\xi\rangle)}\
\end{split}
\end{align}
for all $x\in\S$. 
This coincides with the classical characterization of M\"obius transformations of $\overline{\R^{n-1}}$ in terms of inversions, see \cite[Section~2.1]{Reshetnyak1994} and \cite[Remark~A.1]{zemas2022rigidity} for more details.
For instance, the inversion $\psi\colon x\mapsto x/|x|^2$ in $ \overline{\R^{n-1}}$ corresponds simply to the orthogonal reflection $S_\xi \in \tp O(n)$ with respect to $\xi^\perp :=\{y\in \R^n:\langle y,\xi\rangle=0\}$, \textit{i.e.}, 
 $S_\xi \mathrm{id}_{\S}=\sigma_\xi^{-1}\circ\psi\circ\sigma_\xi$.

We also denote by $\Mob_\pm(\mb S^{n-1})$ the subfamilies of orientation preserving and orientation-reversing M\"obius transformations, corresponding to multiplication by $O\in \SO(n)$ or by $O\in \tp O(n)\setminus \SO(n)$ respectively.

We note that $\Mob(\mb S^{n-1})$ is a \textit{Lie group} of dimension $n(n+1)/2$. Using \eqref{moebius_analytically}, it is an elementary calculation to show that the corresponding \textit{Lie algebra} $T_{\mathrm{id}_{\S}}(\Mob(\mb S^{n-1}))$ can be identified with 
\begin{align}\label{Lie_algebra}
\mob(\S):=\left\{X_{S,\xi,\mu}:  S+S^t=0,\ \xi\in\S,\ \mu\in \R\right\}\,,
\end{align}
where $X_{S,\xi,\mu}\colon  \S
\to \R^n$ is defined by $X_{S,\xi,\mu}(x) := Sx+\mu\big(\langle x,\xi\rangle x-\xi\big)$.

\subsection{The first and second variations of the deficit}\label{subsec: 1st_2nd_variation}

In this subsection, we will consider without restriction maps $u\in W^{1,n-1}(\S;\R^n)$ for which  $\mc V_n(u)>0$, 
without mentioning it further in the sequel. In the case $\mc V_n(u)<0$, all the subsequent results can be retrieved by composing with the \textit{flip} in $\R^n$, \textit{i.e.}, the map $(x_1,\dots,x_{n-1},x_n)\mapsto(x_1,\dots,x_{n-1},-x_n)$. 

For the next result we use the notation
\begin{equation*}
\Delta_{n-1}(u):= \mathrm{div}_{\S}(|\nabla_T u|^{n-3}\nabla_Tu)
\end{equation*}
for the $(n-1)$-Laplace operator on $\S$.
 Recalling \eqref{eq:deficit},  
  we have:
\begin{lemma}\label{lemma: first_variation}
A critical point $u\in W^{1,n-1}(\S;\R^n)$ for $\mc E_{n-1}$ with $\mc V_n(u)>0$, satisfies the system
\begin{equation}\label{eq:EL_H_system}
\Delta_{n-1}(u)+H_uJ(u)=0\,,\ \ \text{where } H_u:=(n-1)^{\frac{n-1}{2}}\frac{\mc D_{n-1}(u)}{\mc V_n(u)}\,,
\end{equation}
in the sense of distributions.
\end{lemma}
\begin{proof}
For every $\psi\in W^{1,n-1}(\S;\R^n)$, we compute the Gateaux derivative
\begin{align*}
\mc E'_{n-1}(u)[\psi]&
:=\frac{d}{dt}\Big|_{t=0}\mc E_{n-1}(u+t\psi)\\
&
=\frac{n}{n-1}\frac{\mc D_{n-1}(u)^{\frac{1}{n-1}}}{\mc V_n(u)^2}\bigg(\mc V_n(u)\mc D_{n-1}'(u)[\psi]-\frac{n-1}{n}\mc D_{n-1}(u)\mc V_n'(u)[\psi]\bigg)\,.
\end{align*} 
We then compute each of the Gateaux derivatives separately. For the first derivative, we have
\begin{align*}
\mc D_{n-1}'(u)[\psi]=\frac{1}{(n-1)^{\frac{n-3}{2}}}\fint_{\S}|\nabla_Tu|^{n-3}\nabla_Tu:\nabla_T\psi\d \H^{n-1}\,.
\end{align*}
For the derivative of the volume, we pick extensions $U,\Psi\in W^{1,n}(\B;\R^n)$ of $u, \psi$,  and use  \textit{Piola's identity} \eqref{eq:Piola} together with the divergence theorem, to get 
\begin{align*}
\mc V_n'(u)[\psi] &
=\fint_{\B} \cof\nabla U 
:\nabla\Psi\,\mathrm{d}x\\	&=n\fint_{\S}\langle \cof(\nabla U)x,\psi \rangle\d \H^{n-1}=n\fint_{\S}\langle J(u),\psi\rangle\d \H^{n-1}\,,	
\end{align*}
where in the last equality we used also \eqref{eq:cofeqwedge}. Collecting all the above identities, we obtain 
\[\frac{\mc V_n(u)}{(n-1)^\frac{n-1}{2}}\fint_{\S}|\nabla_T u|^{n-3}\nabla_Tu\colon \nabla_T\psi\d \H^{n-1}-\mc D_{n-1}(u)\fint_{\S}\langle J(u),\psi\rangle\d \H^{n-1}=0\,,\]
which is precisely the distributional form of \eqref{eq:EL_H_system}.
\end{proof}

Solutions to \eqref{eq:EL_H_system} enjoy the following regularity properties, shown by Mou and Yang in \cite[Theorem 3.6]{Mou1996}:

\begin{theorem}\label{thm:H-system}
Let $u\in W^{1,n-1}(\mb S^{n-1};\R^n)$ be a solution to \eqref{eq:EL_H_system}. Then $u\in C^{1,\alpha}(\mb S^{n-1};\R^n)$ for some $\alpha \in (0,1)$.
\end{theorem}

\begin{remark}
\normalfont 
It is expected that continuity of solutions holds provided that the constant $H_u$ in  \eqref{eq:EL_H_system} 
is replaced by a   bounded and Lipschitz 
function $H\colon \mb S^{n-1}\to \R$, cf.\ \cite[Problem 2.4]{Schikorra2017}.
\end{remark}

Let $\mathcal L_n :=\mathcal E_{n-1}''(\mathrm{id}_\S)$ denote the second variation of the deficit $\mathcal{E}_{n-1}$ at $\mathrm{id}_\S$, and let $Q_n$ the associated quadratic form
on $W^{1,2}(\mathbb S^{n-1};\R^n)$, which is given by 
\begin{align}\label{eq:Qn} \begin{split}
\hspace{-0.5em}Q_n(w)
&:=\frac 12 \frac{n}{n-1}\fint_{\mathbb S^{n-1}} |\nabla_T w|^2
+\frac{n(n-3)}{2(n-1)^2}\fint_{\mathbb S^{n-1}}(\mathrm{div}_{\S}w)^2-\frac n2\fint_{\mathbb S^{n-1}} \langle w, A(w)\rangle\,,
\end{split}
\end{align}
where 
\begin{align}\label{eq:opearator_A}
A(w):= (\mathrm{div}_{\S}w)x-\sum_{j=1}^{n}x_j\nabla_T w^j\,,
\end{align}
cf.\ \cite[(1.12) and Subsection 5.1]{zemas2022rigidity}.
In particular, for every $u\in W^{1,n-1}(\S;\R^n)$ with 
\[\fint_{\S}u=0\,,\qquad \fint_{\S}\langle u,x\rangle=1\,,\]
conditions that for our purposes can be ensured by a translation and a rescaling, if one writes $u:=\mathrm{id}_{\S}+w$, then a formal Taylor expansion gives (cf.\  \cite[Appendix B]{zemas2022rigidity} 
 and the discussion in Subsection~\ref{quantitative_description})
\[
\mathcal{E}_{n-1}(u)=Q_n(w)+\mathcal{O}\Big(\fint_{\S}|\nabla_T w|^3\Big)\,.
\]

The following coercivity result is proven in \cite[Theorem 1.5]{zemas2022rigidity}.

\begin{theorem}[Linear stability]\label{thm:linstab}
There is a constant $C_n>0$ such that  for all $w\in H_n$,
where
\begin{equation}\label{eq:H_n}
H_n:=\bigg\{w\in W^{1,2}(\S;\R^n)\colon \fint_{\S}w=0\,,\quad \fint_{\S}\langle w,x\rangle=0\bigg\}\,,\\
\end{equation}
the following estimate holds:
\begin{align}\label{eq:lin_stab}
Q_n(w) \geq C_n \fint_{\mathbb S^{n-1}}\big|\nabla_T w-\nabla_T(\Pi_{n,0}w)\big|^2\d \H^{n-1}\,,
\end{align}
where
$\Pi_{n,0}\colon H_n\to H_{n,0}$ is the $W^{1,2}$-orthogonal projection
onto the kernel $H_{n,0}$ of $Q_n$ in $H_n$, which is isomorphic to $\mob(\S)$ defined in \eqref{Lie_algebra}.
\end{theorem}

\begin{remark}
\normalfont
We note that the proof of Theorem \ref{thm:linstab} given in \cite[Subsection 5.1]{zemas2022rigidity} is substantially more involved for $n>3$ than for $n=3$: indeed, in the latter case the middle term in \eqref{eq:Qn} is absent, 
and the quadratic form $Q_3$ and the operator $A$ commute, leading to an explicit diagonalization of $Q_3$. In higher dimensions this is no longer the case, and the optimal constant in \eqref{eq:lin_stab} is not known, see  \cite[Remark 5.1]{zemas2022rigidity} for further discussion.
%
\end{remark}

\section{The parametric conformal isoperimetric inequality}\label{sec:param_conformal}

The purpose of this section is to prove Proposition \ref{prop:isopintro}, which we restate here for the reader's convenience.

\begin{proposition}\label{prop:isop}
For all maps $u\in W^{1,n-1}(\mb S^{n-1};\R^n)$ we have 
$$\mc E_{n-1}(u)\geq 0\,,$$ with equality if and only if $(u-y_0)/|\mc V_n(u)|^{1/n}\in \Mob(\mb S^{n-1})$ for some $y_0\in \R^n$.
\end{proposition}

As mentioned in the Introduction,  Proposition \ref{prop:isop} is based on an optimal parametric form of the  isoperimetric inequality. 
The non-parametric form of the isoperimetric inequality is well-known, see \textit{e.g.}\ \cite[(2.10)]{Osserman1978} and the references therein. The parametric isoperimetric inequality is an immediate consequence of the corresponding inequalities for currents, see \textit{e.g.} \cite[Corollary 6.5]{Federer1960} or \cite{Almgren1986}. Here we give a simple proof of this result, starting from the usual isoperimetric inequality and following the approach of \cite[Lemma 1.3]{Muller1990}; we begin by proving a related statement in Euclidean space, for which we recall \eqref{eq:cofeqwedge} and the definition of local degree in \eqref{eq:general_degreee}.

\begin{proposition}\label{prop:isoperimetricinRn} For all $U\in W^{1,n}(\mb B^n;\R^n)$ with $u:=U|_{\S} \in W^{1,n-1}(\S;\R^n)$ we have
\begin{equation}
\label{eq:isopRn}
n\,\omega_n^{1/n}   \left(\|\deg(U,\mb B^n;\cdot)\|_{L^1(\R^n)}\right)^{\frac{n-1}{n}} \leq \int_{\S} |J(u)|\, \mathrm{d}\haus\,.
\end{equation}
If equality holds, then either $\haus\big(\{J(u)\neq 0\}\big)=0$, or there is a ball $B\subset \R^n$ such that
\begin{gather}\label{eq:equalitydeg}
\deg(U,\mb B^n ;y)  =  \pm \mathbf 1_{B}(y) \quad \text{for } \L^n\text{-a.e.\ }y\in \R^n\,,\\[2pt]
\label{eq:equalitycof}
u(x)\in \partial B \quad \text{for } \haus\text{-a.e.\ } x  \in   \{J(u)\neq 0\}\,.
\end{gather}
\end{proposition}

\begin{proof}
The proof of \eqref{eq:isopRn} follows essentially \cite[Lemma~1.3]{Muller1990}, 
which we reproduce here to make the discussion of equality cases transparent.
For every $V\in C^1_c(\R^n;\R^n)$ with $\|V\|_{L^\infty(\R^n)}\leq 1$, using \eqref{eq:deg=jac},  \eqref{eq:bulk_surface} and \eqref{eq:cofeqwedge}, we estimate 
\begin{align*}
\int_{\R^n}(\ddiv V)(y)\mathrm{deg}(U,\B;y)\,\mathrm{d}y&=\int_{\B}(\ddiv V)(U(x))\det\nabla U(x)\,\mathrm{d}x\\
& =\int_{\S}\langle V (  U(x) ),
\cof(\nabla U)x\rangle \d   \haus (x)
\\
&\leq \|V\|_{L^\infty(\R^n)}\int_{\S}|\cof(\nabla U)x| \d \haus \leq \int_{\S}|J(u)|\,\mathrm{d}\haus\,,
\end{align*}
so that supremizing over all such $V$ and using the duality between finite Radon measures and continuous bounded functions, this estimate leads us to
\begin{equation}\label{eq:BVdeg}
\big|D[\deg(U,\mb B^n;\cdot)]\big|(\R^n) \leq \int_{\S} |J(u)|\, \mathrm{d}\haus\,.
\end{equation}
By the Sobolev inequality for $BV$-functions, see \textit{e.g.}\ \cite[(1.4)]{Fusco2007}, \eqref{eq:BVdeg} yields 
\begin{equation}\label{eq: Sobolev_with_J_u}
n\, \omega_n^{1/n} \|\deg(U,\mb B^n;\cdot)\|_{L^{\frac{n}{n-1}}(\R^n)} \leq \int_{\S} |J(u)|\,\mathrm{d}\haus\,,
\end{equation}
with equality  if and only if 
\begin{equation}
\label{eq:degball}
\deg(U,\mb B^n;y)  =   \alpha   
\mathbf{1}_{B}(y) \quad \text{for } \L^n\text{-a.e.\ }y\in \R^n\,,
\end{equation}
for some ball $B:=B_r(y_0)\subset \R^n, r>0, y_0\in \R^n$ and some $  \alpha   
\in \R$.  In fact, since the local degree is integer-valued, cf.\ Lemma \ref{lem:TUBV}, we must have $  \alpha   
 \in \Z$. Thus, $|\deg(U,\mb B^n;\cdot)| \leq |\deg(U,\mb B^n;\cdot)|^{\frac{n}{n-1}}$, and estimate \eqref{eq:isopRn} follows.

We now characterize the equality cases, so assume that \eqref{eq:isopRn} holds with equality.  Clearly we must then have $  \alpha   
\in \{-1,0,1\}$. In the case $ \alpha   
=0$, by the equality cases in the above inequalities we get 
$$\int_{\S} |J(u)|\,\mathrm{d}\haus=n\, \omega_n^{1/n} \|\deg(U,\mb B^n;\cdot)\|_{L^{\frac{n}{n-1}}(\R^n)}=0\,,$$
and so $J(u)=0$ $\haus$-a.e.\ on $\S$.
Otherwise, we see from \eqref{eq:degball} that the local degree of $U$ is (up to a sign) the indicator function of the ball $B$, thus \eqref{eq:equalitydeg} holds. We must also have equality in \eqref{eq:BVdeg} and  \eqref{eq: Sobolev_with_J_u}, and thus by \eqref{eq:degball},
\begin{equation}\label{eq:equality_1}
\int_{\S}|J(u)|\,\mathrm{d}\haus=\big| D [\deg(U,\mb B^n;\cdot)]\big|(\R^n)=\haus(\partial B)=n\omega_nr^{n-1}\,.
\end{equation}
Notice however that, by \eqref{eq:degball}, and denoting by $\nu_{\partial B}$ the outward pointing unit normal to $\partial B$, 
\begin{align*}
\begin{split}
\big| D [\deg(U,\mb B^n;\cdot)]\big|(\R^n)& =\sup_{\underset{\|V\|_{L^\infty(\R^n)}\leq 1}{V\in C^1_c(\R^n;\R^n)}}\int_{\R^n}\deg(U,\B;y)(\ddiv V)(y)\,
\mathrm{d}y
\\
&=\sup_{\underset{\|V\|_{L^\infty(\R^n)}\leq 1}{V\in C^1_c(\R^n;\R^n)}}- \alpha \int_{\partial B}\langle V,\nu_{\partial B}\rangle
\leq \haus(\partial B)= n\omega_nr^{n-1}\,,
\end{split}
\end{align*}
and by \eqref{eq:equality_1} we see that the supremum in the above estimate is achieved for any vector field
\begin{equation}\label{eq: V_0}
V_0\in C^1_c(\R^n;\R^n) \ \text{with } \|V_0\|_{L^\infty(\R^n)}\leq 1 \ \text{\ and \ } V_0|_{\partial B}=- \alpha \nu_{\partial B}\,.
\end{equation}
Using again \eqref{eq:equality_1}, \eqref{eq:deg=jac} and \eqref{eq:bulk_surface}, we deduce that for every such $V_0$,
\begin{align*}
\begin{split}
\int_{\S}|J(u)| \d \haus &= \int_{\R^n}\mathrm{deg}(U,\B;y)(\ddiv V_0)(y)\,\mathrm{d}y\\
& =\int_{\S}\langle V_0\circ u, J(u)\rangle \d \haus\leq \int_{\S} |V_0\circ u| |J(u)| \d \haus\\
&\leq \|V_0\circ u\|_{L^\infty(\S)}\int_{\S}|J(u)|\d \haus \leq \int_{\S}|J(u)|\d \haus\,,
\end{split}
\end{align*}
which further implies that
\begin{equation}\label{eq:equality_3}
|V_0(u(x))|=\sup_{\S} |V_0\circ u|=1 \quad \text{for } \haus\text{-a.e.\ } x 
\in 
 \{ J(u)\neq 0 \}\,.
\end{equation}
Picking for instance a vector field $V_0$ satisfying \eqref{eq: V_0} and moreover being such that 
\[|V_0(y)|<1 \ \ \forall y\in \R^n\setminus \partial B\,,\]
we deduce that \eqref{eq:equality_3} readily implies \eqref{eq:equalitycof}.
\end{proof}

As an immediate consequence of Proposition \ref{prop:isoperimetricinRn}, we obtain:

\begin{corollary}[Parametric isoperimetric inequality]
For all $u\in W^{1,n-1}(\S;\R^n)$ we have
\begin{equation}
\label{eq:param_isoperimetric_ineq}
\fint_\S |J(u)| \d \haus \geq |\mc V_n(u)|^{\frac{n-1}{n}}.
\end{equation}
\end{corollary}

\begin{proof}
 Clearly we may assume that $\mc V_n(u)\neq 0$, as otherwise there is nothing to prove. Let $U\in W^{1,n}(\mb B^n;\R^n)$ be an extension of $u$ as in \eqref{eq:Sobolev_bound_for_U}. Combining 
 \eqref{eq:deg=jac} (for $\eta\equiv 1$)   and \eqref{eq:isopRn}, we see that
$$\omega_n^{1/n} \left|\int_{\mb B^n} \det\n U \d x\right|^{\frac{n-1}{n}} \leq \frac 1 n\int_{\S} |J(u)|   \d \haus  
\,,
$$
and the claim follows by  applying \eqref{eq:bulkbdry} and rearranging. 
\end{proof}


\begin{proof}[Proof of Proposition \ref{prop:isop}]
The non-negativity of the deficit $\mc E_{n-1}$ follows at once from \eqref{eq:param_isoperimetric_ineq} and Hadamard's inequality \eqref{eq:Hadamard}, so to complete the proof it remains to characterize the maps $u\in W^{1,n-1}(\S;\R^n)$ such that 
\begin{equation}\label{eq:zero_energy}
\mc E_{n-1}(u)=0\,.
\end{equation} Since such a map $u$ is in particular a critical point for $\mathcal{E}_{n-1}$, by Theorem \ref{thm:H-system} it is necessarily in $C^{1,\alpha}(\S;\R^n)$ for some $\alpha \in (0,1)$. Moreover, by \eqref{eq:deg=jac} (applied with $\eta\equiv 1$), \eqref{eq:bulkbdry} and \eqref{eq:zero_energy}, we have that $u$ satisfies \eqref{eq:isopRn} with equality, and hence by Proposition  \ref{prop:isoperimetricinRn} also \eqref{eq:equalitydeg} and \eqref{eq:equalitycof} hold. Note that for the ball $B:=B_r(y_0)$ therein, by \eqref{eq:equality_1}, \eqref{eq:param_isoperimetric_ineq} and \eqref{eq:zero_energy}, we have
\begin{equation}
\label{eq:radius}
r=\bigg(\fint_{\S}|J(u)|\,\mathrm{d}\haus\bigg)^{\frac{1}{n-1}}=|\mc V_n(u)|^\frac{1}{n}\,,
\end{equation}
so that by replacing $u$ with $(u-y_0)/|\mc V_n(u) |^{1/n}$, we can assume without restriction that $B=\mb B^n$.  We claim that the open set 
$$\Omega:=\{x\in \S: u(x)\not \in \S\}$$ is empty.  To see this, note that by \eqref{eq:equalitycof} we have
\begin{equation}
\label{eq:zerocof}
J(u) = 0 \ \ \text {in } \Omega\,.
\end{equation}
Since \eqref{eq:zero_energy}, \eqref{eq:param_isoperimetric_ineq} and \eqref{eq:Hadamard} imply that equality must hold in \eqref{eq:Hadamard} pointwise, \eqref{eq:zerocof} in turn implies that
$$ \n_T u  = 0  \ \ \text{ in } \Omega\,.$$ 
But this means that $u$ is constant in each connected component of $\Omega$. Since $u$ is continuous and $u(\S \setminus \Omega)\subseteq \S$, we see that $\Omega=\emptyset$, and therefore $u(\S)\subseteq \S$.

As mentioned in the previous paragraph, equality holds in \eqref{eq:Hadamard} pointwise,  which  by the equality cases in the Cauchy--Schwarz and arithmetic mean-geometric mean inequalities  means that the vectors $\{\p_{\tau_1} u,\dots, \p_{\tau_{n-1}} u\}$ are all pairwise orthogonal and have the same norm. 
Equivalently,  at every $x\in \S$ the singular values of the $(n-1)\times (n-1)$ matrix $(\n_T u(x))^t \n_T u(x)$ are all the same, which means precisely that $u$ solves \eqref{eq:weakconf}. Moreover, by \eqref{eq:equality_1} (for $r=1$ and $y_0=0$ as we have without restriction assumed) and \eqref{eq:radius}, we deduce that 
\begin{equation*}
|\mc V_n(u)|=\bigg(\fint_{\S}|J(u)|\bigg)^{\frac{n}{n-1}} =1 \,,
\end{equation*}
\textit{i.e.}, the weakly conformal map $u\in W^{1,n-1}(\S;\S)$ has  topological degree  $\pm 1$ on $\S$, cf.\ Remark \ref{rmk:vol=deg}. The conclusion now follows from Liouville's Theorem.
\end{proof}


\section{Compactness of sequences with vanishing deficit}\label{sec:compactness}

The purpose of this section is to prove Theorem \ref{thm:compactintro}. To do so, and as mentioned in the beginning of Subsection \ref{subsec: 1st_2nd_variation}, we can assume without restriction that 
\begin{equation}\label{eq:vol_positive}
\mc V_n(u_j)>0 \ \ \forall j\in \N\,.
\end{equation}
Moreover, all the subsequent statements hold up to extraction of a subsequence, which we will  neither  relabel nor mention further in the sequel. 


\subsection{Some auxiliary results of general character}

Before 
getting to
the core of the proof of Theorem \ref{thm:compactintro}, we begin by stating and proving some general results from the Calculus of Variations.

The first result we will need is essentially well-known, and its proof is a variant of that of \cite[Theorem 2.2]{Strzelecki2004}, see also \cite[Theorem~3 on page 40]{Evans1990}, \cite[Theorem on page 3]{Hardt1994a} and \cite{Courilleau2001}: it essentially says that a good control on the $p$-Laplace operator of a sequence of maps yields a subcritical form of compactness. For the statement, recall the notation
$\Delta_p u:=\dv_{ \S } (|\nabla_T u|^{p-2}\nabla_{ T} u)$.

\begin{proposition}[Subcritical compactness]\label{suboptimal compactness}
Let $p\geq 2$ and let $(v_j)_{j\in \N}\subset W^{1,p}(\S;\R^n)$ be weak solutions to 
\begin{equation}\label{eq:approximate_equation}
\Delta_{p} v_j+  f_j
=g_j\,,
\end{equation}
where
\begin{equation}\label{eq:uniform_energy_L1_estimate}
K:=\sup_{j\in \N} \|v_j\|_{W^{1,p}(\S)} +\sup_{j\in \N} \|  f_j 
\|_{L^1(\S)}<+\infty\,, 
\end{equation}
and, with $p':=\frac{p}{p-1}$, 
\begin{equation}\label{eq:strong_convergence_to_0_in_the_predual}
g_j\to 0 \ \text{strongly in } W^{-1,p'}(\S;\R^n)\,.\\[10pt]
\end{equation}
If $v_j\rightharpoonup v$ in $W^{1,p}(\S;\R^n)$ then $v_j \rightarrow v$ strongly in $W^{1,q}(\S;\R^n)$ whenever $1\leq q<p$.
\end{proposition}

\begin{remark}
Proposition \ref{suboptimal compactness} holds also for $1<p<2$, but in this case the proof has a few additional technicalities, as \eqref{eq:algineq} below does not hold. As we will ultimately apply the proposition with $p:=n-1\geq 2$, we decided to omit this case. 
\end{remark}

\begin{proof}[Proof of Proposition \ref{suboptimal compactness}]
The proof is essentially the same as in \cite[ Theorem  on  page 3]{Hardt1994a}, the only difference is the presence of the right-hand side $g_j$ in estimate \eqref{eq:2nd_estimate_for_I} below.
It suffices to prove that for $q\in[1,p)$ fixed, and for each $\delta\in (0,1]$,
\begin{equation}\label{eq:delta_closeness}
\int_{\S}|\nabla_T v_j-\nabla_T v|^q\, \d \haus=o(\delta)+o_\delta(1/j)\,,
\end{equation} 
where, for a function $M(\delta,j)$, we say that
\begin{enumerate}
\item $M(\delta,j)= o(\delta)$ if $ \lim_{\delta\searrow 0}\sup_{j\in \N}|M(\delta,j)|=0$\,,\\[-10pt]
\item $M(\delta,j)= o_\delta(1/j)$ if for each $\delta\in (0,1]$ we have $\lim_{j\to\infty}|M(\delta,j)|=0$\,.
\end{enumerate}
The assertion then follows by first passing to the limit in \eqref{eq:delta_closeness} as $j\to \infty$ and then $\delta\searrow 0$. 

Since $v_j\to v$ strongly in $L^{p}(\S;\R^n)$, up to passing to a non-relabeled subsequence we may assume that $v_j\to v$ also pointwise $\haus$-a.e. on $\S$. 
We then consider the sets
\begin{equation}\label{eq: sets_of_d_large_diff}
E_{\delta,j}:=\{x\in \S\colon |v_j(x)-v(x)|\geq \delta\}\,, \text{  so that  } \haus(E_{\delta,j})=o_\delta(1/j)\,.  	
\end{equation}
By weak lower semicontinuity and the assumption that $v_j\rightharpoonup v$ in $W^{1,p}(\S;\R^n)$, we have 
\[\|v\|_{W^{1,p}(\S)}\leq \liminf_{j\to\infty}\|v_j\|_{W^{1,p}(\S)}\leq K\,,\] 
hence, by H\"older's inequality,
\begin{align*}
\int_{E_{\delta,j}}|\nabla_Tv_j-\nabla_Tv|^q\,\d \haus
&
\leq [\haus(E_{\delta,j})]^{\frac{p-q}{p}}\|\nabla_T v_j-\nabla_T v\|_{L^p(E_{\delta,
 j
 })}^{
q
}
\\
&
\lesssim K^q \EEE [\haus(E_{\delta,j})]^{\frac{p-q}{p}}\,,
\end{align*}
so that, by \eqref{eq: sets_of_d_large_diff},
\begin{equation}\label{eq:estimate_on_large_dev_set}
\int_{E_{\delta,j}}|\nabla_Tv_j-\nabla_Tv|^q\,\d \haus=o_\delta(1/j)\,.
\end{equation}
The rest of the proof is dedicated to proving that 
\begin{equation}\label{eq:estimate_on_small_dev_set}
\int_{\S\setminus E_{\delta,j}}|\nabla_Tv_j-\nabla_Tv|^{p}\,\d \haus=o(\delta)+o_\delta(1/j)\,,
\end{equation}	
as then \eqref{eq:delta_closeness} follows from \eqref{eq:estimate_on_large_dev_set}, \eqref{eq:estimate_on_small_dev_set} and another application of H\"older's inequality. 
In order to prove \eqref{eq:estimate_on_small_dev_set}, we define the auxiliary function 
\begin{equation}\label{eq:auxiliary_test_function}
\eta:\R^n\to \R^n, \quad \eta(y):=\min\big\{\delta/|y|,1\big\}y\,, 
\end{equation}
for which $\|\eta\|_{L^\infty(\R^{n})}\leq \delta$. Let us recall here the elementary algebraic inequality
\begin{equation}
\label{eq:algineq}
\langle |a|^{p-2}a-|b|^{p-2}b, a-b\rangle\geq 2^{2-p}|a-b|^{p} \ \ \text{for all }  a,b\in \R^m \text{ and  }p\geq 2\,,
\end{equation}
which follows from the simple observation
\begin{align*}
\langle |a|^{p-2}a-|b|^{p-2}b, a-b\rangle&=\frac{|a|^{p-2} + |b|^{p-2}}{2}|a-b|^2+ \frac{(|a|^{p-2} - |b|^{p-2})(|a|^2 - |b|^2)}{2}\\
&\geq \frac{|a|^{p-2} + |b|^{p-2}}{2}|a-b|^2\geq 2^{2-p}|a-b|^{p}\,.
\end{align*}
Since $\eta(y)=y$ when $|y|\leq \delta$, using \eqref{eq:algineq} we can estimate
\begin{align}
\begin{split}
\label{eq:first_estimate_on_bad_set}
\int_{\S\setminus E_{\delta,j}} |\nabla_Tv_j-\nabla_Tv|^{p}&\lesssim \int_{\S\setminus E_{\delta,j}}(|\nabla_Tv_j|^{p-2}\nabla_Tv_j -  |\nabla_Tv|^{p-2}\nabla_Tv):\nabla_T(v_j-v)\\
&= \int_{\S\setminus E_{\delta,j}} |\nabla_Tv_j|^{p-2}\nabla_T v_j:\nabla_T(\eta\circ(v_j - v))\\
& \quad -\int_{\S\setminus E_{\delta,j}} |\nabla_Tv|^{p-2}\nabla_T v: \nabla_T(v_j-v)\\
&=:\int_{\S\setminus E_{\delta,j}} \tp I_j- \int_{\S\setminus E_{\delta,j}} \tp{II}_j\,.
\end{split}
\end{align}
For the second term in the last line of \eqref{eq:first_estimate_on_bad_set} we can easily estimate
\begin{align}\label{second_term_small}
\begin{split}
\left|\int_{\S\setminus E_{\delta,j}}\tp{II}_j\right|
&\leq \left|\int_{\S}|\nabla_Tv|^{p-2}\nabla_T v: \nabla_T(v_j-v)\right|+\int_{E_{\delta,j}} |\nabla_Tv|^{p-1}(|\nabla_Tv|+|\nabla_Tv_j|)\\
&\lesssim \left|\int_{\S}|\nabla_Tv|^{p-2}\nabla_T v: \nabla_T(v_j-v)\right|+K\left(\int_{E_{\delta,j}} |\nabla_Tv|^{p}\right)^{\frac{1}{p'}}\\
& = o_\delta(1/j)\,,
\end{split}
\end{align}
where we used $\nabla_T(v_j-v)\rightharpoonup 0$ weakly in $L^{p}(\S;\R^n)$, $|\nabla_Tv|^{p-2}\nabla_Tv\in L^{p'}(\S;\R^n)$, and that by \eqref{eq: sets_of_d_large_diff}, $\haus (E_{\delta,j})= o_\delta(1/j)$. We now focus on estimating the first term in the last line of \eqref{eq:first_estimate_on_bad_set}, for which clearly
\begin{align}\label{eq:estimate_for_I}
\int_{\S\setminus E_{\delta,j}} \tp I_j=\int_{\S} \tp I_j-\int_{E_{\delta,j}} \tp I_j\,,
\end{align}
and we will estimate the two integrals separately. Testing \eqref{eq:approximate_equation} against $\eta\circ(v_j-v)$, the bounds \eqref{eq:uniform_energy_L1_estimate} and \eqref{eq:strong_convergence_to_0_in_the_predual} yield
\begin{align}\label{eq:2nd_estimate_for_I}
\begin{split}
\left|\int_{\S} \tp I_j \right|&\leq\left|\int_{\S} \big\langle   f_j,\eta\circ(v_j-v)\big\rangle\right|+\left|(g_j,\eta\circ(v_j-v))_{W^{-1,p'},W^{1,p}}\right|\\
&\leq \|\eta\|_{L^\infty(\R^n)}\sup_{j\in \N}\big\|f_j\big\|_{L^1(\S)}+\left|(g_j,\eta\circ(v_j-v))_{W^{-1,p'},W^{1,p}}\right|\\
&\leq K\delta+o_\delta (1/j)=  o(\delta)+o_\delta(1/j)\,.
\end{split}
\end{align}
On the other hand, note that $\eta\circ(v_j-v)=\delta\frac{v_j-v}{|v_j-v|}$ on $E_{\delta,j}$ by \eqref{eq: sets_of_d_large_diff} and \eqref{eq:auxiliary_test_function}, and therefore, $\haus$-a.e. on this set,
\begin{align}\label{eq: calculation_of_I_on_E_delta_i}
\begin{split}	
\tp I_j &=\delta|\nabla_Tv_j|^{p-2}\nabla_Tv_j:\nabla_T \frac{v_j-v}{|v_j-v|}\\
&=\delta|\nabla_Tv_j|^{p-2}\sum_{l=1}^n\bigg\langle \nabla_Tv_j^l,\frac{|v_j-v|\nabla_T(v_j^l-v^l)-(v_j^l-v^l)\nabla_T|v_j-v|}{|v_j-v|^2} \bigg\rangle\\
&=: \tp{I}^{'}_{j,\delta} -\tp{I}^{''}_{j,\delta}\,,
\end{split}
\end{align}
where
\begin{align}\label{eq: I'}
\begin{split}
\tp{I}^{'}_{j,\delta}&:=\delta \frac{|\nabla_Tv_j|^{p-2}}{|v_j-v|} \sum_{l=1}^n\Big\langle \nabla_Tv_j^l,\nabla_Tv_j^l-\frac{(v_j^l-v^l)}{|v_j-v|^2}\sum_{m=1}^n(v_j^m-v^m)\nabla_T v_j^m\Big\rangle\\[2pt]
&=\delta \frac{|\nabla_Tv_j|^{p-2}}{|v_j-v|}\Big(|\nabla_T v_j|^2-\frac{\big|\sum_{l=1}^n (v_j^l-v^l)\nabla_Tv_j^l\big|^2}{|v_j-v|^2}\Big)\,,
\end{split}
\end{align}
and
\begin{align}\label{eq: I''}
\begin{split}	
\tp{I}^{''}_{j,\delta}&:= \delta \frac{|\nabla_Tv_j|^{p-2}}{|v_j-v|}\sum_{l=1}^n\Big\langle \nabla_Tv_j^l,\nabla_Tv^l-\frac{(v_j^l-v^l)}{|v_j-v|^2}\sum_{m=1}^n(v_j^m-v^m)\nabla_T v^m\Big\rangle\,.
\end{split}
\end{align} 
By the Cauchy-Schwarz inequality we see that $\tp{I}'_{j,\delta}\geq 0$.  
Thus,  combining \eqref{eq:first_estimate_on_bad_set}--\eqref{eq: I''}, we get
\begin{align}\label{eq:second_estimate_on_bad_set}
\begin{split}	
\int_{\S\setminus E_{\delta,j}} |\nabla_Tv_j-\nabla_Tv|^{p}&\lesssim o(\delta)+o_\delta(1/j)-\int_{E_{\delta,j}}\tp{I}'_{j,\delta}+\int_{E_{\delta,j}}\tp{I}^{''}_{j,\delta}\\
&\leq o(\delta)+o_\delta(1/j)+\int_{E_{\delta,j}}\tp{I}^{''}_{j,\delta}\,.
\end{split}
\end{align}
Finally, for the last integral in \eqref{eq:second_estimate_on_bad_set}, recalling \eqref{eq: sets_of_d_large_diff}, 
we have  the following pointwise estimate 
 $\haus$-a.e.   on $E_{\delta,j}$:
\begin{align*}
|\tp{I}^{''}_{j,\delta}|&
\leq  |\nabla_Tv_j|^{p-2}\bigg(|\nabla_Tv_j||\nabla_Tv|+\frac{|v_j-v|^2|\nabla_Tv_j||\nabla_Tv|}{|v_j-v|^2}\bigg)
\leq 2|\nabla_Tv_j|^{p-1}|\nabla_Tv|\,.
\end{align*}
Hence, by Hölder's inequality and \eqref{eq:uniform_energy_L1_estimate}, we obtain
\begin{align}\label{eq:last_estimate_for_bad_set}
\begin{split}	
\left|\int_{E_{\delta,j}}\tp{I}^{''}_{j,\delta}\right|&\leq 2\int_{E_{\delta,j}} |\nabla_Tv_j|^{p-1}|\nabla_Tv|\\ & \leq 2\|\nabla_Tv_j\|^{p-1}_{L^{p}(E_{\delta,j})} \|\nabla_Tv\|_{L^{p}(E_{\delta,j})}
\leq  2K^{ p-1}
\|\nabla_Tv\|_{L^{p}(E_{\delta,j})}= o_\delta(1/j)\,,
\end{split}
\end{align}
since $\nabla_T v\in L^{p}(\S)$ and $\haus(E_{\delta,j})=o_\delta(1/j)$. Combining \eqref{eq:second_estimate_on_bad_set} and \eqref{eq:last_estimate_for_bad_set} we arrive at \eqref{eq:estimate_on_small_dev_set}, which together with \eqref{eq:estimate_on_large_dev_set} finally yields \eqref{eq:delta_closeness}, completing the proof.
\end{proof}

We next state the \textit{Brezis--Lieb Lemma}, which is a sharper version of \textit{Fatou's Lemma}, see \textit{e.g.} \cite[Chapter 1, Theorem 8]{Evans1990} for the simple proof.

\begin{lemma}\label{lem:brezis_lieb}
Let $p\geq 1$, $(v_j)_{j\in \N}\subset W^{1,p}(\S;\R^n)$ be such that $v_j\rightharpoonup v$ weakly in $W^{1,p}(\S;\R^n)$ and $\nabla_Tv_j \rightarrow \nabla_Tv$ pointwise $\haus$-a.e. on $\S$. Then,
\begin{equation}\label{eq:brezis_lieb}
\lim_{j\to \infty} \left[\fint_{\S}|\nabla_Tv_j|^{p} - \fint_{\S}|\nabla_Tv_j-\nabla_Tv|^{p} \right]= \fint_{\S}|\nabla_Tv|^{p}\,.
\end{equation}
\end{lemma} 

The next lemma is also well-known and gives an analogous statement to \eqref{eq:brezis_lieb} for the volume functional,  relying on its multilinear structure.

\begin{lemma}\label{volume_brezis_lieb}
Let $(v_j)_{j\in \N}\subset W^{1,n-1}(\S;\R^n)$ be such that $v_j\rightharpoonup v$ weakly in $W^{1,n-1}(\S;\R^n)$. Then,
\begin{equation}\label{eq:vol_brezis_lieb}
\lim_{j\to \infty} \mc V_n(v_j)=\mc V_n(v)+ \lim_{j\to \infty}  \mc V_n(v_j-v)\,.
\end{equation}
\end{lemma}	
\begin{proof}
We consider  
extensions $(V_j)_{j\in \N}\subset W^{1,n}(\B;\R^n)$ of $(v_j)_{j\in \N}$ as in \eqref{eq:Sobolev_bound_for_U}, as well as the 
 corresponding
  extension $V$ of $v$.
By \eqref{eq:Sobolev_bound_for_U} and $v_j\rightharpoonup v$ in $W^{1,n-1}(\S;\R^n)$ it is easy to see that up to a subsequence, 
\begin{equation*}
V_j\rightharpoonup V \ \text{ weakly in } W^{1,n}(\B;\R^n)\,.
\end{equation*} 
Recalling \eqref{eq:bulkbdry}, \eqref{eq:vol_brezis_lieb} is equivalent to
\begin{equation}\label{eq:vol_brezis_lieb_bulk}
\lim_{j\to \infty}  \fint_{\mb B^n} \det \n V_j\d x= \fint_{\mb B^n} \det \n V \d x+ \lim_{j\to \infty}  \fint_{\mb B^n} \det \n (V_j-V)\d x\,.
\end{equation}
Let us set for brevity $W_j:=V_j-V\rightharpoonup 0$ in $ W^{1,n}(\mb B^n;\R^n)$.  The calculation required to obtain \eqref{eq:vol_brezis_lieb_bulk} can be concisely expressed using differential forms, since $\det \n V \d x= \tp d V^1\wedge \dots \wedge \tp d V^n$. Indeed, we calculate
\begin{align}\label{eq: det_of_sum}
\begin{split}	
\fint_{\mb B^n} \det \n V_j \d x &= \fint_{\mb B^n} \det \n (V+W_j) \d x\\
& = \fint_{\B} \tp d(V^1+W_j^1)\wedge\dots \wedge \tp{d}(V^n+W_j^n)\\
& = \fint_{\B} (\det \n V + \det \n W_j )\d x+ \sum_{\#K,\# L>0} \tp{sgn}(\sigma_{K,L}) \fint_{\B} \tp d V^K \wedge \tp d W^L_j \,, 
\end{split}
\end{align}
where $K=\{k_1,\dots, k_m\}$ and $L=\{l_1,\dots l_{n-m}\}$ are complementary subsets of $\{1,\dots, n\}$ such that $k_1<\dots <k_m,\ l_1<\dots <l_{n-m}$, $\sigma_{K,L}$ is the permutation which naturally orders $(K,L)$,
$\tp{d}V^K: = \d V^{k_1}\wedge \dots \wedge \tp{d} V^{k_m}$, and similary for $\tp{d}W^L_j$. Since $\#L>0$ and $W_j\rightharpoonup 0$ in $ W^{1,n}(\B;\R^n)$, each of the terms $ \tp dW_j^L$ in the sum in \eqref{eq: det_of_sum}  converges weakly to 0 in $L^{\frac{n}{n- m}}(\B)$, 
see for instance \cite[Theorem 1.1]{Robbin1987}. As $\tp{d}V^K\in L^{\frac{n}{m}}(\B)$, \eqref{eq:vol_brezis_lieb_bulk} 
 and thus also \eqref{eq:vol_brezis_lieb}
  follow.
\end{proof}

The final auxiliary ingredient we need for the proof of Theorem \ref{thm:compactintro} is \textit{Ekeland's variational principle} \cite{Ekeland1974} applied to the functional $\mc E_{n-1}$, which is continuous in the strong $W^{1,n-1}(\S;\R^n)$-topology. 

\begin{lemma}[Ekeland's variational principle for $\mc E_{n-1}$]\label{Ekeland'slemma}
Let $(u_j)_{j\in \N}\subset W^{1,n-1}(\S;\R^n)$ be a sequence satisfying \eqref{eq:vol_positive} which is minimizing for $\mc E_{n-1}$, \textit{i.e.},
\begin{equation}\label{eq:min_seq_1}
\lim_{j\to \infty}\mc E_{n-1}(u_j)=0\,.
\end{equation}
Then, there exists another minimizing sequence $(v_j)_{j\in \N}\subset W^{1,n-1}(\S;\R^n)$, \textit{i.e.},
\begin{equation}\label{eq:min_seq_2}
\lim_{j\to \infty}\mc E_{n-1}(v_j)=0\,,
\end{equation}
satisfying \eqref{eq:vol_positive}, with the following further properties:
\begin{equation}\label{eq:dif_of_min_seq_negligible}
\lim_{j\to \infty} \|u_j-v_j\|_{W^{1,n-1}(\S)}=0\,,
\end{equation}
and, in the sense of distributions, 
\begin{equation}\label{eq:Palais_Smale}
\Delta_{n-1}(v_j)+H_{v_j}J(v_j)=g_j\,, 
\end{equation}
where $g_j\to 0 $ strongly in
$(W^{1,n-1})^*(\S;\R^n)$\,.
\end{lemma}

Lemma \ref{Ekeland'slemma} asserts that any minimizing sequence can be replaced with another minimizing sequence of approximate solutions to the Euler--Lagrange system \eqref{eq:EL_H_system}, where the perturbations are 
 small  in the dual
of $W^{1,n-1}(\S;\R^n)$; sequences verifying this type of condition are usually referred to as \textit{Palais--Smale sequences} \cite[Section~II.2]{Struwe2008}.

\subsection{The proof of Theorem \ref{thm:compactintro}}

We are now ready to begin the core of the proof of Theorem \ref{thm:compactintro}, 
which is
strongly
 inspired by the proof of \cite[Lemma 2.1]{Caldiroli2006}  for the case $n=3$. 
The main difference, related to the genuinely nonlinear nature of $\mathcal E_{n-1}$ for $n\geq 4$, is that the expansion of $\mathcal D_{n-1}(u_j)$ in \eqref{eq:qual_expansions} is in the latter case not straightforward: we are forced to consider Palais--Smale sequences through Lemma \ref{Ekeland'slemma} and to carefully combine all the previous auxiliary results.
  
We begin with a key lemma which asserts that weakly convergent minimizing sequences are compact  in the strong sense, provided that their limit is non-constant.

\begin{lemma}\label{weak_to_strong}
Let $(u_j)_{j\in \N}\subset W^{1,n-1}(\S;\R^n)$ satisfy \eqref{eq:vol_positive}, \eqref{eq:min_seq_1},  and assume that
\begin{equation}\label{eq:u_jweak_convergence}
u_j\rightharpoonup u\ \text{weakly in } W^{1,n-1}(\S;\R^n)\,,
\end{equation} 
for some non-constant $u$. Then actually,
 \begin{equation}\label{eq:u_j_strong_convergence}
 u_j\rightarrow u\ \text{strongly in } W^{1,n-1}(\S;\R^n)\,,
 \end{equation} 
and in particular $\mc E_{n-1}(u)=0$.
\end{lemma} 
\begin{proof}
For $(u_j)_{j\in \N}$ as in the statement of the lemma, let $(v_j)_{j\in \N} \subset  W^{1,n-1}(\S;\R^n)$ be the sequence provided by Lemma \ref{Ekeland'slemma},  which satisfies \eqref{eq:min_seq_2}--\eqref{eq:Palais_Smale}. In particular, by \eqref{eq:u_jweak_convergence} also
\begin{equation}\label{eq:w_j_weakly_conv}  
v_j\rightharpoonup u \ \text{weakly in } W^{1,n-1}(\S;\R^n)\,,	
\end{equation}
and by \eqref{eq:vol_positive} and the fact that $\lim_{j\to \infty}|\mc V_n(u_j)-\mc V_n(v_j)|=0$, we also have that $\mc V_{n}(v_j)>0$ for all $j\in \N$ large enough.  We prove next that actually 
\begin{equation}\label{eq:w_j_strongly_conv}  
v_j\rightarrow u \ \text{strongly in } W^{1,n-1}(\S;\R^n)\,,
\end{equation}
which clearly yields \eqref{eq:u_j_strong_convergence} by \eqref{eq:dif_of_min_seq_negligible}.

To prove \eqref{eq:w_j_strongly_conv},  we first observe that $(v_j)_{j\in \N}$ satisfies the conditions for the application of Proposition \ref{suboptimal compactness}. Indeed, in view of \eqref{eq:Palais_Smale}, and since $\sup_{j\in \N}\|v_j\|_{W^{1,n-1}(\S)}<+\infty$ by \eqref{eq:w_j_weakly_conv}, setting
\begin{equation*}
 f_j:=H_{v_j}J(v_j)\,, 
\end{equation*}  
with $H_{v_j}$ as in \eqref{eq:EL_H_system}\,,
we only need to verify that 
\begin{equation}\label{eq:f_j_equibdd}
\sup_{j\in \N}\|f_j\|_{L^1(\S)}<+\infty\,.
\end{equation}
To see this,  by \eqref{eq:Hadamard} and \eqref{eq:deficit}, we estimate
\begin{align*}
\|f_j\|_{L^1(\S)}\leq H_{v_j}\fint_{\S} |J(v_j)|\lesssim  \frac{[\mc D_{n-1}(v_j)]^2}{\mc V_{n}(v_j)}= [\mc D_{n-1}(v_j)]^{\frac{n-2}{n-1}}\big(\mc \E_{n-1}(v_j)+1\big)\,,
\end{align*}
so that \eqref{eq:f_j_equibdd} follows from \eqref{eq:w_j_weakly_conv} and \eqref{eq:min_seq_2}. Applying Proposition \ref{suboptimal compactness} with $p:=n-1$ to the sequence $(v_j)_{j\in \N}$, we deduce that 
\begin{equation*}
v_j\rightarrow u \ \text{strongly in } W^{1,q}(\S;\R^n) \ \text{\ for all } 1\leq q<n-1\,,
\end{equation*}
and so, up to a subsequence, $\nabla_T v_j\rightarrow \nabla_T u$ pointwise $\haus$-a.e. on $\S$. Hence, by \eqref{eq:w_j_weakly_conv} and applying Lemmata \ref{lem:brezis_lieb} and \ref{volume_brezis_lieb} we deduce that
\begin{align}\label{eq:qual_expansions}
\begin{split}
 \mc D_{n-1}(u)& =\mc D_{n-1}(v_j)-\mc D_{n-1}(v_j-u)+o(1)\,,\\
\mc V_n(u)&=\mc V_n(v_j) - \mc V_n(v_j-u)+o(1)\,,
\end{split}
\end{align}
where $o(1)$ simply denotes a quantity which vanishes as $j\to \infty$.

The rest of the proof is identical to that of \cite[Lemma 2.1, Step 2]{Caldiroli2006} but, for the sake of completeness, we include the details here. Note that \eqref{eq:qual_expansions} and \eqref{eq:min_seq_2} imply that
\begin{align}\label{eq:chain_of_ineq}
\begin{split}
\mc D_{n-1}(u) & = \mc D_{n-1}(v_j)- \mc D_{n-1}(v_j-u)+o(1)\\
& =|\mc V_{n}(v_j)|^{\frac{n-1}{n}}- \mc D_{n-1}(v_j-u) +o(1)\\
&=|\mc V_n(u)+\mc V_n(v_j-u)|^{\frac{n-1}{n}}- \mc D_{n-1}(v_j-u)+o(1)\\
&\leq |\mc V_n(u)|^{\frac{n-1}{n}}+|\mc V_n(v_j-u)|^{\frac{n-1}{n}}- \mc D_{n-1}(v_j-u)+o(1)\\
&\leq \mc D_{n-1}(u)+o(1)\,,
\end{split}
\end{align}	
the inequality between the third and fourth line following from the algebraic inequality 
\begin{equation}\label{eq:algebraic_ineq}
|t+s |^{\frac{n-1}{n}}\leq 
|t|^{{\frac{n-1}{n}}}
+|s|^{\frac{n-1}{n}}\, \ \forall   t,s \in \R\,,
\end{equation}
while the last inequality follows simply from the nonnegativity of the deficit 
 applied to $u$ and $v_j-u$, cf.\ Proposition \ref{prop:isop}. Passing to a further non-relabeled subsequence for which
\[\lim_{j\to \infty}\mc V_n(v_j-u)=:\mc V\in \R\,,\]
and taking the limit superior as $j\to \infty$ in \eqref{eq:chain_of_ineq}, we infer that actually
\[|\mc V_n(u)+\mc V|^{\frac{n-1}{n}}=|\mc V_n(u)|^{\frac{n-1}{n}}+|\mc V|^{\frac{n-1}{n}}\implies \mc V_n(u)=0 \ \text{or } \mc V=0\,,\]
by the cases of equality in \eqref{eq:algebraic_ineq}.
If $\mc V_n(u)=0$, then instead of \eqref{eq:chain_of_ineq} we would simply have 
\begin{align*}
\mc D_{n-1}(u) &
 = \mc D_{n-1}(v_j)- \mc D_{n-1}(v_j-u)+o(1)\\
& =|\mc V_{n}(v_j)|^{\frac{n-1}{n}}- \mc D_{n-1}(v_j-u)+o(1)\\
&=|\mc V_n(v_j-u)|^{\frac{n-1}{n}}- \mc D_{n-1}(v_j-u)+o(1)\leq o(1)\,,
\end{align*}	
thus $\mc D_{n-1}(u)=0$, which is a contradiction to the fact that $u$ is non-constant. Hence, we must have $\mc V=0$. Starting with the lower semicontinuity of the conformal Dirichlet energy and then arguing as in the derivation of \eqref{eq:chain_of_ineq}, we get
\begin{align*}
\mc D_{n-1}(u)  \leq \liminf_{j\to \infty} \mc D_{n-1}(v_j)
& \leq  \limsup_{j\to \infty} \mc D_{n-1}(v_j)\\
& =\limsup_{j\to \infty} |\mc V_{n}(v_j)|^{\frac{n-1}{n}}\\
&=\limsup_{j\to \infty} |\mc V_n(u)+\mc V_n(v_j-u)|^{\frac{n-1}{n}}
=|\mc V_n(u)|^{\frac{n-1}{n}}
\leq \mc D_{n-1}(u)\,.
\end{align*}
Thus, $\lim_{j\to \infty}\mc D_{n-1}(v_j)=\mc D_{n-1}(u)$ which, together with \eqref{eq:w_j_weakly_conv}, implies \eqref{eq:w_j_strongly_conv}, 
 concluding the proof.
\end{proof}

We are finally ready to complete the proof of our compactness theorem which, after all our preparations, is a direct adaptation of the argument in Step 3 of the proof of \cite[Lemma 2.1]{Caldiroli2006}, now in general dimension $n\geq 3$.

\begin{proof}[Proof of Theorem \ref{thm:compactintro}]
 As the proof is somewhat long, we split it into several steps.

\medskip
\textit{Step 1: Construction of the normalizing maps $(\phi_j)_{j\in\N}$.} Let $(u_j)_{j\in \N}\subset W^{1,n-1}(\S;\R^n)$ be a minimizing sequence as in \eqref{eq:minim_seq} and \eqref{eq:vol_positive}, and let us set for brevity
\begin{equation*}
\mu_j:=\big[\mc V_n(u_j)\big]^{-1/n}>0\,.
\end{equation*}
By the scaling invariance of $\mc E_{n-1}$, the sequence $(\mu_j u_j)_{j\in \N}\subset W^{1,n-1}(\S;\R^n)$ is again energy minimizing, with 
\begin{equation}\label{normalization_condition}
\mc V_n(\mu_ju_j)=1\,, \ \text{hence } \lim_{j\to \infty} \mc D_{n-1}(\mu_ju_j)=1\,. 
\end{equation}
In what follows, for $t\in [-1,1]$ we denote by 
\[\mathbb{S}_{t, \pm}^{n-1}:=\Big\{x\in \S: x_n \gtrless t\Big\}\,,\]
omitting the dependence on $t$ when $t=0$.

The main purpose of this step is to construct maps $(\phi_j)_{j\in \N}\subset \Mob(\S)$ (recall \eqref{def: Mobius_group}), such that if we set
\begin{equation}\label{eq:rescaled_reparam_sequence}
\tilde u_j:=\frac{u_j  \circ \phi_j-\fint_{\mb S^{n-1}} u_j \circ \phi_j }{ \big[\mc V_n(u_j)\big]^{1/n}}\,, 
\end{equation}
then, up to a subsequence, we have
\begin{align}\label{eq:prop_of_w_j}
\begin{split}	
\rm{(i)}&\quad\fint_{\S}\tilde u_j=0\,, \quad \mc V_n(\tilde u_j)=1\,,\\[2pt]
\rm{(ii)}&\quad \lim_{j\to\infty} \mc D_{n-1}(\tilde u_j)=1\,,\\[2pt]
 \mathrm{(iii)}& \quad \frac{n \omega_n}{k+1}\leq 
 \int_{\mb S^{n-1}_{s,+}} \Big(\frac{|\n_T \tilde u_j|^2}{n-1}\Big)^{\frac{n-1}{2}}\,\mathrm{d}\H^{n-1}\leq \frac{n \omega_n k }{k+1} \quad   \text{for any }  
s\in [-r_0,0]\,, 
\end{split}
\end{align}
where both $k:=k(n)\in \N$ and $r_0>0$ will be specified later. 
Note that (i) and (ii) follow from the translation, scaling and conformal invariance of the deficit and \eqref{normalization_condition}, so in particular $(\tilde u_j)_{j\in \N}\subset W^{1,n-1}(\S;\R^n)$ is another minimizing sequence. In turn, property (iii) is not obvious and depends on the construction of the sequence $(\phi_j)_{j\in \N}$.  We also note that this last property will be crucial to ensure non-concentration of the energy in the next steps. 

To construct the sequence $(\phi_j)_{j\in\N}$, it is more convenient to work in $\R^{n-1}$, after stereographically projecting, so we consider the maps
\begin{equation*}
v_j :=  (\mu_ju_j)\circ \sigma_{e_n}^{-1} \colon \R^{n-1}\to \R^n\,.
\end{equation*}
We can find $k:=k(n)\in\N$ and points $\{p_1,\dots, p_k\}\subset \overline{\mb B^{n-1}}$ such that 
\begin{equation}
\label{eq:defUk}
\overline{\mb B^{n-1}}\subset \Omega_k := \bigcup_{\ell=1}^k B_1^{ n-1}(p_\ell)\,,
\end{equation}
where all balls are in $\R^{n-1}$ for this part of the argument. 
Note that $\sigma_{e_n}(\mb{S}^{n-1}_+)=\mb B^{n-1} \subset  \R^{n-1}$.  Since $\Omega_k$ in \eqref{eq:defUk} is an open cover,  there exists $\rho_0>1$ such that
$ B_{\rho_0}^{n-1}(0)  \subset \Omega_k$.  In particular, for  $r_0:=\frac{\rho_0^2-1}{\rho_0^2+1}>0$,  we have  
$$\mb S^{n-1}_{-r_0,+} = \sigma_{e_n}^{-1}\big( B_{\rho_0}^{n-1}(0)\big)\,.$$ 
We now choose points $(q_j)_{j\in \N} \subset \R^{n-1}$ and  radii $(\rho_j)_{j\in \N} \subset \R_+$   such that
\begin{equation}\label{eq:good_maximal_points}
\int_{B_{\rho_j}^{ n-1}(q_j)} \bigg(\frac{|\n v_j|^2}{n-1}\bigg)^{\frac{n-1}{2}} \d y=
\sup_{ q\in \R^{n-1}} \int_{B_{\rho_j}^{ n-1}( q)} \bigg(\frac{|\n v_j|^2}{n-1}\bigg)^{\frac{n-1}{2}} \d y = \frac{n\omega_n}{k+1}\,.
\end{equation}
The supremum in \eqref{eq:good_maximal_points} is attained for any fixed $\rho\geq 0$ in the place of $\rho_j$,
as the supremum of a nonnegative continuous function  in $\R^{n-1}$ tending to zero at infinity. 
Moreover, the value of this supremum grows continuously from $0$ at $\rho=0$, to $n\omega_n \mathcal D_{n-1}(\mu_j u_j)$ as $\rho\to\infty$,
so that a choice of $\rho_j$  satisfying the last equality in \eqref{eq:good_maximal_points} is possible thanks to \eqref{normalization_condition}.

Let now $T_j\colon \R^{n-1}\to \R^{n-1}$ be defined via \EEE $T_j(y):=\rho_j y + q_j$, \EEE and consider the new sequence
$$\tilde v_j := v_j \circ T_j\,. $$
By a  change of variables in \eqref{eq:good_maximal_points} this new sequence satisfies
\begin{equation}
\label{eq:choicevjtilde}
\int_{\mb B^{n-1} } \bigg(\frac{|\n \tilde v_j|^2}{n-1}\bigg)^{\frac{n-1}{2}} \d y=
\sup_{q\in \R^{n-1}} \int_{B_1^{ n-1 }(q)} \bigg(\frac{|\n \tilde v_j|^2}{n-1}\bigg)^{\frac{n-1}{2}} \d y = \frac{n\omega_n}{k+1}\,.
\end{equation}
Finally, for the definition in \eqref{eq:rescaled_reparam_sequence}, we take 
$$\phi_j:= \sigma_{e_n}^{-1} \circ T_j \circ \sigma_{e_n} \in \Mob_{ +}(\S)\,.$$ 
With this definition, the lower bound in \eqref{eq:prop_of_w_j}(iii) follows from the conformal invariance of the energy, \eqref{eq:choicevjtilde}, and the fact that $\mb B^{n-1}\subset \sigma_{e_n}(\mb S^{n-1}_{s,+})$  whenever  $s\leq 0$.  For the upper bound in \eqref{eq:prop_of_w_j}(iii), we argue similarly: since $s\in [-r_0,0]$,  using again conformal invariance,  the simple observation below \eqref{eq:defUk} and \eqref{eq:choicevjtilde}, we have
\begin{align*}
\int_{\mb S^{n-1}_{s,+}}\bigg(\frac{|\n_T \tilde u_j|^2}{n-1}\bigg)^{\frac{n-1}{2}}\,\mathrm{d}\H^{n-1} & \leq
\int_{\mb S^{n-1}_{-r_0,+}}\bigg(\frac{|\n_T \tilde u_j|^2}{n-1}\bigg)^{\frac{n-1}{2}}\,\mathrm{d}\H^{n-1} \\
& =\int_{ B_{\rho_0}^{n-1}(0)}\bigg(\frac{|\n \tilde v_j|^2}{n-1}\bigg)^{\frac{n-1}{2}}\,\mathrm{d}y  
\\
&
\leq \sum_{\ell=1}^k \int_{ B_1^{n-1}(p_\ell)}\bigg(\frac{|\n \tilde v_j|^2}{n-1}\bigg)^{\frac{n-1}{2}}\,\mathrm{d}y\leq \frac{n\omega_nk}{k+1}\,,
\end{align*}
as claimed.\EEE

To conclude this step we note that, by \eqref{eq:prop_of_w_j}(i),(ii) and Poincar{\' e's} inequality, we have $\sup_{j\in \N}\|\tilde u_j\|_{W^{1,n-1}(\S)}<+\infty$. Hence, up to a non-relabeled subsequence, 
\begin{equation}\label{eq: w_j_weakly_converges}
\tilde u_j	\rightharpoonup u \ \text{weakly in } W^{1,n-1}(\S;\R^n)\,.
\end{equation}
We claim that $u\in W^{1,n-1}(\S;\R^n)$ is non-constant.  Note that, if $u$ is constant, then by \eqref{eq:prop_of_w_j}(i) and the fact that $\tilde u_j\to u$ strongly in $L^{n-1}(\S;\R^n)$, we must have 
\begin{align}\label{eq: u_is_0}
u=\fint_{\S}u=\lim_{j\to \infty}\fint_{\S}\tilde u_j=0\,.
\end{align}
We now assume for the sake of contradiction that \eqref{eq: u_is_0} is the case.

\medskip
\textit{Step 2: Construction of $(n-2)$-spheres where no energy concentration occurs.} 
We now construct a particular sequence $(s_j)_{j\in \N} \subset [-r_0,0]$  such that
\begin{equation}\label{eq:fractional_norm_to_0}
\lim_{j\to\infty} \|\tilde u_j\|_{W^{1-\frac{1}{n-1},n-1}(\partial \mathbb{S}^{n-1}_{s_j,+})}=0\,.	
\end{equation}	
By the general slicing properties of Sobolev functions, for $\L^1$-a.e.\ $s\in (-r_0,0)$ we have that 
\begin{equation}\label{eq:sobolev_slices}
\tilde u_j|_{\mathbb{S}^{n-2}_{s}}\in W^{1,n-1}(\mathbb{S}^{n-2}_{s};\R^n)\,, \ \text{where } \mathbb{S}^{n-2}_s:=\partial \mathbb{S}^{n-1}_{s,+}\,,
\end{equation}
the latter intended of course in the sense of relative boundary with respect to $\S$. 
Moreover, denoting by $\nabla_{\mathbb{S}_s^{n-2}}$ the tangential gradient on the intermediate $(n-2)$-sphere $\mathbb{S}_s^{n-2}$, we define the sets 
\begin{align}\label{eq:aux_sets_I_1_2}
\begin{split}
I_{j,1}&:=\bigg\{s\in(-r_0,0)\colon \eqref{eq:sobolev_slices} \ \text{holds, } \int_{\mathbb{S}_s^{n-2}}\big|\nabla_{\mathbb{S}^{n-2}_s}\tilde u_j\big|^{n-1}\leq \frac{3}{r_0}\int_{\S}\big|\nabla_{T}\tilde u_j\big|^{n-1}\bigg\}\,,\\
I_{j,2}&:=\bigg\{s\in(-r_0,0)\colon \eqref{eq:sobolev_slices} \ \text{holds, } \int_{\mathbb{S}_s^{n-2}}|\tilde u_j|^{n-1}\leq \frac{3}{r_0}\int_{\S}|\tilde u_j|^{n-1}\bigg\}\,.
\end{split}
\end{align}
By Fubini's theorem and since $|\nabla_{T}\tilde u_j| \geq  |\nabla_{\mathbb{S}^{n-2}_s}\tilde u_j|$ $\mathcal{H}^{n-2}$-a.e. on $\mathbb{S}^{n-2}_s$, we obtain
\begin{align*}
\int_{\S}|\nabla_T\tilde u_j|^{n-1}&\geq \int_{-r_0}^{0}\bigg(\int_{\mathbb{S}_s^{n-2}}|\nabla_{T}\tilde u_j|^{n-1}\bigg)\,\mathrm{d}s\geq \int_{(-r_0,0)\setminus I_{j,1}}\bigg(\int_{\mathbb{S}_s^{n-2}}|\nabla_{\mathbb{S}^{n-2}_s}\tilde u_j|^{n-1}\bigg)\,\mathrm{d}s\\
&>\frac{3}{r_0}\L^1((-r_0,0)\setminus I_{j,1}) \int_{\S}|\nabla_T\tilde u_j|^{n-1}\,,
\end{align*}
hence $\L^1(I_{j,1})>\frac{2r_0}{3}$. A similar argument shows that $\L^1(I_{j,2})>\frac{2r_0}{3}$. Thus, 
\begin{align*}
\L^1(I_{j,1}\cap I_{j,2})= \L^1(I_{j,1})+\L^1(I_{j,2})-\L^1(I_{j,1}\cup I_{j,2})\geq \frac{2r_0}{3}+\frac{2r_0}{3}-r_0=\frac{r_0}{3}\,.
\end{align*} 
In particular, we can pick $s_j\in I_{j,1}\cap I_{j,2}$, and using the Gagliardo--Nirenberg--Sobolev inequality in fractional Sobolev spaces \cite[Theorem 1]{BrezisMironescu2018} and \eqref{eq:aux_sets_I_1_2}, we deduce that
\begin{align}\label{gagliardo_nirenberg}
\begin{split}	
\|\tilde u_j\|_{W^{1-\frac{1}{n-1},n-1}(\mathbb{S}_{s_j}^{n-2})}&\lesssim_{r_0} 	\|\tilde u_j\|_{W^{1,n-1}(\mathbb{S}_{s_j}^{n-2})}^{\frac{n-2}{n-1}}	  \|\tilde u_j\|_{L^{n-1}(\mathbb{S}_{s_j}^{n-2})}^{\frac{1}{n-1}}\\
&\lesssim_{r_0} \|\tilde u_j\|_{W^{1,n-1}(\S)}^{\frac{n-2}{n-1}}	  \|\tilde u_j\|_{L^{n-1}(\S)}^{\frac{1}{n-1}}\,,
\end{split}
\end{align}
so that \eqref{eq:fractional_norm_to_0} follows from \eqref{eq: w_j_weakly_converges},  \eqref{eq: u_is_0} and \eqref{gagliardo_nirenberg}.

\medskip
\textit{Step 3: Construction of extensions of the traces on the good $(n-2)$-spheres.}
For every $j\in \N$, let $\zeta_j\in W^{1,n-1}(\mathbb{S}_{s_j,+}^{n-1};\R^n)$ be an extension of $\tilde u_j|_{\partial \mathbb{S}^{n-1}_{s_j,+}}$ satisfying
\[\|\zeta_j\|_{W^{1,n-1}(\mathbb{S}_{s_j,+}^{n-1})}\lesssim 
\|\tilde u_j\|_{W^{1-\frac{1}{n-1},n-1}(\partial \mathbb{S}^{n-1}_{s_j,+})}\,,\]
similarly to  \eqref{eq:Sobolev_bound_for_U}. In particular,  \eqref{eq:fractional_norm_to_0} yields
\begin{equation}\label{eq: n_harmonic_tends_to_0}
\lim_{j\to \infty} \|\zeta_j\|_{W^{1,n-1}(\mathbb{S}_{s_j,+}^{n-1})}=0\,. 
\end{equation}

Let $\psi_j\in \Mob_{-}(\S)$ be the M\"obius transformation that maps conformally $\mathbb{S}_{s_j,+}^{n-1}$ onto $\mathbb{S}_{s_j,-}^{n-1}$, keeping $\partial \mathbb{S}_{s_j,+}^{n-1}$ invariant, while reversing the orientation. For every $j\in \N$ we now define
\begin{equation*}
u_{1,j}(x):=\begin{cases}
\tilde u_j(x)\,  \ \ \qquad \text{if } x\in \mathbb{S}_{s_j,+}^{n-1}\,,\\[4pt]
\zeta_j(\psi_j^{-1}(x))\, \ \text{if } x\in \mathbb{S}_{s_j,-}^{n-1}\,,
\end{cases}
\text{and}
\quad u_{2,j}(x):=\begin{cases}
\tilde u_{j}(\psi_j(x))\,  \ \quad \text{if } x\in \mathbb{S}_{s_j,+}^{n-1}\,,\\[4pt]
\zeta_j(\psi_j^{-1}(x))\, \ \  \ \text{if } x\in \mathbb{S}_{s_j,-}^{n-1}\,.
\end{cases}
\end{equation*}
By the invariance of the $(n-1)$-Dirichlet energy under conformal reparametrizations, we get
\begin{align}\label{eq:split_of_Dir}
\begin{split}
\mc D_{n-1}(u_{1,j})&=\frac{1}{n\omega_n}\bigg(\int_{\mathbb{S}_{s_j,+}^{n-1}}\bigg(\frac{|\nabla_T \tilde u_j|^2}{n-1}\bigg)^{\frac{n-1}{2}}+\int_{\mathbb{S}_{s_j,+}^{n-1}}\bigg(\frac{|\nabla_T \zeta_j|^2}{n-1}\bigg)^{\frac{n-1}{2}}\bigg)\,,\\
\mc  D_{n-1}(u_{2,j})&=\frac{1}{n\omega_n}\bigg(\int_{\mathbb{S}_{s_j,-}^{n-1}}\bigg(\frac{|\nabla_T \tilde u_j|^2}{n-1}\bigg)^{\frac{n-1}{2}}+\int_{\mathbb{S}_{s_j,+}^{n-1}}\bigg(\frac{|\nabla_T \zeta_j|^2}{n-1}\bigg)^{\frac{n-1}{2}}\bigg)\,,
\end{split}
\end{align}
hence
\begin{equation}\label{eq:sum_of_energies}
\mc D_{n-1}(u_{1,j})+\mc D_{n-1}(u_{2,j})=\mc {D}_{n-1}(\tilde  u_j)+\frac{2}{n\omega_n}\int_{\mathbb{S}_{s_j,+}^{n-1}}\bigg(\frac{|\nabla_T \zeta_j|^2}{n-1}\bigg)^{\frac{n-1}{2}}\,.
\end{equation}
For the volume term, recalling the notation in \eqref{eq:cofeqwedge}, using the change of variables formula and the fact that $\psi_j$ reverses the orientation, we have
\begin{align*}
\mc V_{n}(u_{1,j})&=\frac{1}{n\omega_n}\bigg(\int_{\mathbb{S}_{s_j,+}^{n-1}}\langle \tilde u_j,J(\tilde u_j)\rangle-\int_{\mathbb{S}_{s_j,+}^{n-1}}\langle \zeta_j,J(\zeta _j)\rangle\bigg)\,,\\[2pt]
\mc V_{n}(u_{2,j})&=\frac{1}{n\omega_n}\bigg(-\int_{\mathbb{S}_{s_j,-}^{n-1}}\langle \tilde u_j,J(\tilde u_j)\rangle-\int_{\mathbb{S}_{s_j,+}^{n-1}}\langle \zeta_j,J(\zeta _j)\rangle\bigg)\,,
\end{align*}
hence 
\begin{equation}\label{eq:dif_of_volumes}
\EEE\mc V_{n}(u_{1,j})-\mc V_{n}(u_{2,j})=\mc V_{n}(\tilde u_{j})\,.\\[4pt]
\end{equation} 
\medskip
\textit{Step 4: The limit $u$ is non-constant.} We now proceed similarly to the proof of Lemma \ref{weak_to_strong}.  First note that, since $(\tilde u_j)_{j\in \N}\subset W^{1,n-1}(\S;\R^n)$ is a minimizing sequence, combining \eqref{eq:sum_of_energies}, \eqref{eq: n_harmonic_tends_to_0},  \eqref{eq:dif_of_volumes}, \eqref{eq:algebraic_ineq}, and the non-negativity of $\mc E_{n-1}$, we have 
\begin{align}\label{eq:convexity_argument}
\begin{split}	
\mc D_{n-1}(u_{1,j})+\mc D_{n-1}(u_{2,j})&=\mc D_{n-1}(\tilde 
u_{j})+o(1)\\
& =|\mc V_n(\tilde  
u_{j})|^{\frac{n-1}{n}}+o(1)\\
&=|\mc V_{n}(u_{1,j})-\mc V_{n}(u_{2,j})|^{\frac{n-1}{n}}+o(1)\\
&\leq|\mc V_{n}(u_{1,j})|^{\frac{n-1}{n}}+|\mc V_{n}(u_{2,j})|^{\frac{n-1}{n}}+o(1)\\
&\leq \mc D_{n-1}(u_{1,j})+\mc D_{n-1}(u_{2,j})+o(1)\,.
\end{split}
\end{align}
Passing to further subsequences if necessary, we can assume that the following limits exist:
\begin{align*}
\begin{split}	
&\lim_{j\to \infty} \mc D_{n-1}(u_{1,j})=:D_1\in \R_+\,, \quad \lim_{j\to \infty} \mc D_{n-1}(u_{2,j})=:D_2\in \R_+\,,\\
&\lim_{j\to \infty} \mc V_{n}(u_{1,j})=:V_1\in \R\,, \qquad\quad \lim_{j\to \infty} \mc V_{n}(u_{2,j})=:V_2\in \R\,,
\end{split}
\end{align*} 
and taking the limit as $j\to \infty$ in \eqref{eq:convexity_argument}, we infer in particular that
\begin{equation*}
|V_{1}-V_{2}|^{\frac{n-1}{n}}=|V_{1}|^{\frac{n-1}{n}}+|V_{2}|^{\frac{n-1}{n}}\implies V_1=0 \text{ or } V_2=0\,.
\end{equation*}
Both cases lead to a contradiction. Indeed, if $\lim_{j\to \infty}\mc V_{n}(u_{1,j})=V_1=0$, \eqref{eq:convexity_argument} actually gives
\begin{align*}
\begin{split}	
\mc D_{n-1}(u_{1,j})+\mc D_{n-1}(u_{2,j})&\leq|\mc V_{n}(u_{1,j})|^{\frac{n-1}{n}}+|\mc V_{n}(u_{2,j})|^{\frac{n-1}{n}}+o(1)\\
&=|\mc V_{n}(u_{2,j})|^{\frac{n-1}{n}}+o(1)\\
& \leq \mc D_{n-1}(u_{2,j})+o(1)\,,
\end{split}
\end{align*} 
so that passing to the limit $j\to \infty$, we get $D_1=0$. This, 
combined with \eqref{eq: n_harmonic_tends_to_0} and \eqref{eq:split_of_Dir},  gives
\begin{equation*}
\lim_{j\to \infty}\int_{\mathbb{S}^{n-1}_{ s_j,+}}\bigg(\frac{|\nabla_T \tilde u_j|^2}{n-1}\bigg)^{\frac{n-1}{2}}\,\d\haus=0\,,
\end{equation*}
contradicting the first inequality in \eqref{eq:prop_of_w_j}(iii). 
In the second case, if  $\lim_{j\to \infty}\mc V_{n}(u_{2,j})=V_2=0$, by the very same argument, we get 
\begin{equation*}
\lim_{j\to \infty}\int_{\mathbb{S}^{n-1}_{ s_j,-}}\bigg(\frac{|\nabla_T  \tilde u_j|^2}{n-1}\bigg)^{\frac{n-1}{2}}\,\d\haus=0 \iff \lim_{j\to \infty}\int_{\mathbb{S}^{n-1}_{s_j,+}}\bigg(\frac{|\nabla_T  \tilde u_j|^2}{n-1}\bigg)^{\frac{n-1}{2}}\,\d\haus=n\omega_n\,,
\end{equation*}
by \eqref{eq:prop_of_w_j}(ii), which contradicts the second inequality in  \eqref{eq:prop_of_w_j}(iii). 
Thus, in either case we obtain a contradiction, and this shows that the limiting $u$ in \eqref{eq: w_j_weakly_converges} is non-constant, as claimed.

\medskip
\textit{Step 5: Conclusion.} By the previous step, the limit $u$ is non-constant. Hence,  applying Lemma \ref{weak_to_strong}, we deduce that actually 
\begin{equation}\label{eq: w_j_strongly_converges}
\tilde u_j	\rightarrow u \ \text{strongly in } W^{1,n-1}(\S;\R^n)\,,
\end{equation} 
and $u$ is a minimizer for $\mc E_{n-1}$ with 
\begin{equation}\label{eq:normalization_for_u}
\mc D_{n-1}(u)=1\,,\quad \mc V_n(u)=1\,,\quad \fint_{\S}u=0\,,
\end{equation}
as a consequence of \eqref{eq:prop_of_w_j}(i),(ii) and \eqref{eq: w_j_strongly_converges}. In particular, by \eqref{eq:normalization_for_u} and Proposition \ref{prop:isop} we have that $u\in  \Mob_+(\mb S^{n-1})$, see also Remark \ref{rmk:vol=deg}. The zero-average condition in  \eqref{eq:normalization_for_u} implies by a standard argument that in fact
\begin{equation}\label{eq:lim_is_rotation}
u=R\id_{\S} \ \text{for some } R\in \SO(n)\,,
\end{equation}
cf.\ \cite[Lemma A.3]{zemas2022rigidity} or \cite[Lemma 2.4]{GuerraLamyZemas2023}. 
Indeed, using \eqref{eq:normalization_for_u}, Jensen's inequality,  and the sharp Poincar\'e inequality on $\S$ for functions with zero average, cf.\ \cite[(2.7)]{GuerraLamyZemas2023}, we have
\begin{align*}
1 &=\fint_{\S}\bigg(\frac{|\nabla_T u|^2}{n-1}\bigg)^{\frac{n-1}{2}} \geq \bigg(\fint_{\S}\frac{|\nabla_T u|^2}{n-1}\bigg)^{\frac{n-1}{2}} 
\geq \Big(\fint_{\S} |u|^2 \Big)^{\frac{n-1}{2}} = 1\,,
\end{align*}
the last equality following from the fact that $u$ is $\S$-valued. Hence equality holds throughout in the above chain of inequalities and, in particular, $u$ satisfies the sharp Poincar\'e inequality on $\S$ with equality, so it must be linear, \textit{i.e.}, $u(x)=Rx$ for some $R\in \R^{n\times n}$. Finally,  since $u \in W^{1,n-1}(\mb S^{n-1}; \mb S^{n-1})$ has degree one, it must be that $R\in\SO(n)$. 

Combining now \eqref{eq:rescaled_reparam_sequence}, \eqref{eq: w_j_strongly_converges} and \eqref{eq:lim_is_rotation}, we conclude the proof of the theorem.  \qedhere
\end{proof}

\section{Quantitative stability}\label{sec:quantitative_stability}
\subsection{Reduction to the local stability problem}
\label{sec: quant_stability_local}

The purpose of this section is to prove the quantitative Theorem \ref{thm:qttiveintro}. With the qualitative analogue at hand, \textit{i.e.},  Theorem \ref{thm:compactintro},  a standard \textit{contradiction$\slash$compactness argument} shows that it suffices to prove the \textit{$W^{1,n-1}(\S;\R^n)$-local} version of the desired estimate, as claimed in the next lemma. 

\begin{lemma}\label{from_local_to_global}
It suffices to prove the $W^{1,n-1}(\S;\R^n)$-local version of Theorem \ref{thm:qttiveintro}, \textit{i.e}, it suffices to prove it for maps  in the class 
\begin{equation}\label{eq:local_family}
\mathcal{B}_{\delta_n,\varepsilon_n}
:=\left\{ u 
\in W^{1,n-1}(\S;\R^n)\colon
\begin{array}{lr}
{\rm(i)}\ \ \ \fint_{\S}  u 
=0\\[7pt]
{\rm(ii)}\ \ \fint_{\S}|\t  u 
-P_T|^{n-1}\leq \delta_n \\[7pt]
{\rm(iii)}\ \ \mc E_{n-1}( u 
)< \varepsilon_n
\end{array}\right\},
\end{equation}
for some $\delta_n, \varepsilon_n\in (0,1)$ sufficiently small constants.
\end{lemma}
\begin{proof}
That without loss of generality we can assume $\rm (i)$ is obvious, since the deficit $\mc E_{n-1}$ is translation-invariant. That we can assume property $\rm (iii)$ is also immediate, because for every $u\in W^{1,n-1}(\S;\R^n)$ with $\mc V_n(u)\neq 0$ and $\mc E_{n-1}(u)\geq \varepsilon_n$, and for every $\phi\in \Mob(\mb S^{n-1})$, by the conformal invariance of the $(n-1)$-Dirichlet energy, we trivially obtain
\begin{align}\label{eq:trivial_deficit_bound}
\begin{split}
\fint_{\mb S^{n-1}} \bigg| \frac{\n_T u}{ |\mc V_n(u)|^{1/n}} - \n_T \phi\bigg|^{n-1}&\leq 2^{n-2}\bigg(|\mc V_n(u)|^{-\frac{n-1}{n}}\fint_{\mb S^{n-1}} |\n_T u|^{n-1}+\fint_{\S}|\n_T \phi|^{n-1}\bigg)\\[2pt]
&=2^{n-2}(n-1)^{\frac{n-1}{2}}\bigg[\bigg(\frac{[\mc D_{n-1}(u)]^{\frac{n}{n-1}}}{|\mc V_n(u)|}\bigg)^{\frac{n-1}{n}}+1\bigg]\\[2pt]
&=c_n \bigg[\bigg(1+\mc E_{n-1}(u)\bigg)^{\frac{n-1}{n}}+1\bigg]\\[2pt]
&\leq c_n \bigg(2+\frac{n-1}{n}\mc E_{n-1}(u)\bigg)
\leq c_n \bigg(\frac{2}{\e_n}+\frac{n-1}{n}\bigg)
\mc E_{n-1}(u)
\,,
\end{split}
\end{align} 
where $c_n:=2^{n-2}(n-1)^{\frac{n-1}{2}}$, and in the last line we used the algebraic inequality $(1+t)^{\gamma}\leq 1+\gamma t$, valid for all $t \geq  0$ and $\gamma\leq 1$.
Therefore, it suffices to prove Theorem~\ref{thm:qttiveintro} in the $\varepsilon_n$-small deficit regime. 
	
Suppose now that we have proven Theorem \ref{thm:qttiveintro} for maps in $\mathcal{B}_{\delta_n,\varepsilon_n}$, but for the sake of contradiction that the theorem fails to hold globally. Then, for every $j\in \mathbb{N}$ there exists  $u_j\in W^{1,n-1}(\S;\R^n)$ with $\mc V_n(u_j)\neq 0$ such that for every $\phi\in \Mob(\mb S^{n-1})$, 
\begin{equation}\label{eq:contr_comp_1}
\fint_{\mb S^{n-1}} \bigg| \frac{\n_T u_j}{ |\mc V_n(u_j)|^{1/n}} - \n_T \phi\bigg|^{n-1} \mathrm{d}\haus \geq j \mc E_{n-1}(u_j)\,.
\end{equation}
Combining \eqref{eq:trivial_deficit_bound} and \eqref{eq:contr_comp_1}, we immediately see that for $j\in \N$ large enough,
\[0\leq \mc E_{n-1}(u_j)\leq \frac{2c_n}{j-c_n \frac{n-1}{n}}\to 0\ \ \text{as } j\to \infty\,.\]
We can now use the compactness result from Theorem \ref{thm:compactintro} to obtain a contradiction: by its application, we can find $(\phi_j)_{j\in \N}\subset \Mob(\mb S^{n-1})$ and $O\in \tp O(n)$ so that, up to a non-relabelled subsequence,
\begin{equation*}
v_j:=\frac{u_j  \circ \phi_j-\fint_{\mb S^{n-1}} u_j \circ \phi_j }{ |\mc V_n(u_j)|^{1/n}}\rightarrow O\id_{\S} \ \text{strongly in } W^{1,n-1}(\S;\R^n)\,.
\end{equation*}
Without loss of generality (by considering $O^{-1} v_j$ instead of $v_j$) we can also suppose that $O=I_n$. Then, for the constants $\delta_n, \varepsilon_n\in (0,1)$ in the statement, we can find $j_0:=j_0(\delta_n,\varepsilon_n)\in \mathbb{N}$ such that the sequence $(v_j)_{j\geq j_0}\subset \mathcal{B}_{\delta_n,\varepsilon_n}$. Hence, by the assumption that \eqref{eq:qttive} is valid in this class, we deduce that there exists $(\psi_j)_{j\geq j_0}\subset \Mob(\mb S^{n-1})$ such that
\begin{align*}
\fint_{\S}\bigg|\frac{\n_T u_j}{|\mc V_n(u_j)|^{1/n}} - \n_T (\psi_j\circ\phi^{-1}_j)\bigg|^{n-1}&=\fint_{\S} \big|\n_T v_j - \n_T \psi_j\big|^{n-1}\lesssim  \mc E_{n-1}(v_j)= \mc E_{n-1}(u_j) \,,
\end{align*}
for all $j\geq j_0$, which clearly contradicts \eqref{eq:contr_comp_1}.
\end{proof}

Ultimately, the nonlinear stability estimate \eqref{eq:qttive} is a consequence of its linear analogue quoted in Theorem \ref{thm:linstab}. In order to apply the latter result, we need to 
use the invariance of the deficit under the action of $\Mob(\mb S^{n-1})$ to fix some of the degrees of freedom in the problem.  This is the content of the next lemma, whose proof can be obtained by following verbatim the steps of proof of \cite[Lemma 4.13]{zemas2022rigidity}, just adapted to general dimension $n\geq 3$\,.
See also  \cite[Lemma~4.1.12]{ZemasPhD}, where $\nabla_T(u\circ\psi)$ is actually shown to be close to $P_T$ in $L^2$ 
-- instead of $L^{n-1}$ as in  (ii) of \eqref{eq:local_family} --
but since Lemma~\ref{fixingMobius} perturbs $u$ along the finite-dimensional manifold $\Mob_+(\mb S^{n-1})$, the choice of norm is immaterial.

\begin{lemma}\label{fixingMobius}
Given $\delta_n, \varepsilon_n \in (0,1)$ sufficiently small, there exists $\tilde \delta_n \in (0,1)$ that depends only on $\delta_n$ and such that $\tilde \delta_n\to 0$ as $\delta_n\to 0$, so that for every $u\in \mathcal{B}_{\delta_n,\varepsilon_n}$ there exists $\psi\in \Mob_+(\mb S^{n-1})$ such that
\begin{equation}\label{eq:ucompphi_zero_proj}
\left(u\circ\psi-\fint_{\S} u\circ\psi\right) \in \mathcal{B}_{\tilde\delta_n,\varepsilon_n} \ \ \text{with } \ \ \Pi_{n,0}(u\circ\psi)=0\,,
\end{equation} 	
where $\Pi_{n,0}$ is 
the projection onto the kernel $H_{n,0}$ of $Q_n:=\mathcal E_{n-1}''(\id_\S)$
as in Theorem \ref{thm:linstab}\,.
\end{lemma}

In order to apply Theorem \ref{thm:linstab}, we want to consider perturbations of $\id_\S$ which are $W^{1,2}(\S)$-orthogonal to it.
In that respect, let us set 
\begin{equation}\label{eq:fixing_lambda}
\lambda_{u,\psi}:=\fint_{\S}\langle u\circ \psi,x\rangle\,\d\haus\,,
\end{equation} 
and define $w\in W^{1,n-1}(\S;\R^n)$ via
\begin{equation}\label{eq:correct_w}
w:=\frac{1}{\lambda_{u,\psi}}\Big(u\circ \psi-\fint_{\S}u\circ \psi\Big)-\mathrm{id}_{\S}\,.
\end{equation} 
We then have the following simple result:

\begin{lemma}\label{lem:w}
The map $w$ defined through \eqref{eq:correct_w} satisfies $w\in H_n$ and  $\Pi_{n,0}w=0$. Moreover,  there is a constant $C_n>0$ such that, if $\delta_n \in(0,1)$ is sufficiently small, 
$$\fint_\S |\n_T w|^{n-1} \d \haus \leq C_n \tilde \delta_n =:\theta_n.$$
\end{lemma}

\begin{proof}
The first claims are immediate from \eqref{eq:ucompphi_zero_proj},  \eqref{eq:fixing_lambda} and the definition of $H_n$ in \eqref{eq:H_n}.
To prove the estimate,  we begin by  recalling the convention that the $L^p$-norms are taken with respect to the normalized $\mathcal{H}^{n-1}$-measure. Then, by the cancellation property $\fint_{\S} x=0$, the Poincar{\'e} inequality on $\S$, H{\" older's} inequality and \eqref{eq:ucompphi_zero_proj}, we obtain
\begin{align}
\label{eq:lambda_estimate}
\begin{split}
\big|\lambda_{u,\psi}-1\big|&=\left|\fint_{\S}\big\langle (u\circ\psi-x)-\fint_{\S}(u\circ \psi-x),x\big\rangle\right|\\
&\leq \Big\| (u\circ\psi-x)-\fint_{\S}(u\circ \psi-x)\Big\|_{L^2(\S)}\\
&\lesssim \Big\| \nabla_T(u\circ\psi)-P_T\Big\|_{L^2(\S)}\lesssim \Big\| \nabla_T(u\circ\psi)-P_T\Big\|_{L^{n-1}(\S)}\\
&\leq (\tilde \delta_n)^{\frac{1}{n-1}} \leq \frac{1}{2}\,,
\end{split}
\end{align}
by choosing $\delta_n\in (0,1)$ sufficiently small so that $\tilde \delta_n\in (0,1)$ is also small.
Hence,  as $u\in \mathcal{B}_{\delta_n,\varepsilon_n}$ (cf.\ \eqref{eq:local_family}),  and by \eqref{eq:ucompphi_zero_proj} and the conformal invariance of the $(n-1)$-Dirichlet energy, we get
\begin{align*}
\fint_{\S}|\nabla_T w|^{n-1}&=\fint_{\S}\bigg|\frac{1}{\lambda_{u,\psi}}\nabla_T(u\circ\psi)-P_T\bigg|^{n-1}\nonumber\\
&\lesssim \Big|\frac{1}{\lambda_{u,\psi}}-1\Big|^{n-1}\fint_{\S}|\nabla_T(u\circ \psi)|^{n-1}+\fint_{\S}|\nabla_T(u\circ \psi)-P_T|^{n-1}\\
& \lesssim |\lambda_{u,\psi}-1|^{n-1}\fint_{\S}|\nabla_T u|^{n-1}+\fint_{\S}|\nabla_T(u\circ \psi)-P_T|^{n-1} \\
& \lesssim \tilde \delta_n\left(\fint_{\S}|\nabla_T u-P_T|^{n-1}+ \fint_{\S}|P_T|^{n-1}\right) + \tilde \delta_n\\
& \lesssim \tilde \delta_n(\delta_n+1)\,,
\end{align*}
as wished. \qedhere
\end{proof}


\subsection{Proof of Theorem \ref{thm:qttiveintro}}\label{sec:main_quantitative_local} 

We now begin the proof of the main quantitative estimate of this section:  by the preparations of the previous subsection,  this estimate will imply Theorem \ref{thm:qttiveintro} easily.

\begin{proposition}[Local nonlinear stability]\label{prop:local_nonlinear_stability}
There exists a constant $\tilde c_n>0$ and sufficiently small constants $\theta_n,\varepsilon_n\in (0,1)$ such that
\begin{align}\label{eq:nonlin_stab}
\mathcal E_{n-1}(\mathrm{id}_\S+w)
\geq \tilde c_n \fint_{\mathbb S^{n-1}}|\nabla_T w|^{n-1}\,\mathrm{d}\haus \,,
\end{align}
for all maps $w\in W^{1,n-1}(\S;\R^n)$ such that
\begin{equation}\label{eq:condw}
w\in H_n\,, \ \ \text{ \ } \Pi_{n,0}w=0\,, \ \ \text{and }\ \mathrm{id}_{\S}+w\in \mathcal{B}_{\theta_n,\varepsilon_n}\,,
\end{equation}
where $H_n$, $\Pi_{n,0}$ are as in Theorem \ref{thm:linstab} and $\mathcal{B}_{\theta_n,\varepsilon_n}$ is defined through \eqref{eq:local_family}\,.
\end{proposition}

Before giving the proof of the proposition,
we start with some auxiliary calculations for a Taylor-type expansion of the deficit $\mc E_{n-1}$ around the identity, with perturbations as the ones in \eqref{eq:condw}. We begin with the expansion of the volume.

\begin{lemma}[Expansion of the volume]
\label{lemma:prep_calc}
Let $w\in W^{1,n-1}(\S;\R^n)$ satisfy \eqref{eq:condw}, then
\begin{equation}
\label{eq:exp_vol}
\mc V_n(\id_\S + w) = 1+ \frac n 2 \fint_\S \langle w, A(w)\rangle + \sum_{k=3}^{n-1}\frac{n}{k}\fint_{\mathbb S^{n-1}}\langle w,[\sigma'_k(\nabla_T wP_T^t)]^t x\rangle+ \mc V_n(w)\,,
\end{equation}
where recall that $A$ is the linear operator defined in \eqref{eq:opearator_A}.
In particular, if $\theta_n\in (0,1)$ is chosen sufficiently small, then 
\begin{equation}\label{eq:vol_trivial_bound}
\frac{1}{2}\leq \mc V_n(\mathrm{id}_{\S}+w)\leq \frac{3}{2}\,.
\end{equation}
\end{lemma}
\begin{proof}
The expansion \eqref{eq:exp_vol} is an immediate consequence of  \eqref{eq:boundary_integrals},  using the fact that the $k=1$ term therein vanishes by \eqref{eq:first_term_in_the_expansion} and \eqref{eq:condw},  the  $k=2$ term therein is precisely $\frac{n}{2}\fint_{\S}\langle w,A(w)\rangle$, cf.\ \eqref{eq:opearator_A}, and the $k=n$ term is the volume by \eqref{last_term_in_the_expansion}. 

To prove \eqref{eq:vol_trivial_bound}, we start from \eqref{eq:exp_vol} and we estimate using \eqref{eq:trivial_estimate_for_intermediate_terms} and H\"older's inequality:
\begin{align*}
|\mc V_n(\mathrm{id}_{\S}+w)-1|&\leq \Bigg|\sum_{k=2}^{n-1} \frac{n}{k}\fint_{\S}\langle w,[\sigma'_k(\nabla_T wP_T^t)]^tx\rangle\Bigg|+\big|\mc V_n(w)\big|\\
&\lesssim \sum^{n-1}_{k= 2 }\fint_{\S}|w||\nabla_Tw|^{k-1}+ \big|\mc V_n(w)\big|\\
&\leq \sum^{n-1}_{k= 2 %
}\|w\|_{L^{\frac{n-1}{n-k}}(\S)}\|\nabla_Tw\|_{L^{n-1}(\S)}^{k-1}+ \big|\mc V_n(w)\big|\\
&\leq \sum^{n-1}_{k= 2 
}\|w\|_{L^{n-1}(\S)}\|\nabla_Tw\|_{L^{n-1}(\S)}^{k-1}+ \big[\mc D_{n-1}(w)\big]^\frac{n}{n-1}\\
&\lesssim  \sum^{n}_{k=2}
\|\nabla_Tw\|_{L^{n-1}(\S)}^{k}
 \lesssim \theta_n^{\frac{ 2 \EEE
 	}{n-1}}\leq \frac{1}{2} \,,
\end{align*}
provided that $\theta_n\in(0,1)$ is chosen sufficiently small.  Note that here we have also used the Poincar{\'e} inequality in $W^{1,n-1}(\S;\R^n)$ (which is applicable due to \eqref{eq:condw}), and again the parametric conformal-isoperimetric inequality.
\end{proof}

We next give a Taylor-type lower inequality for the conformal $(n-1)$-Dirichlet energy, which is essentially the one devized in \cite{figalli2022sobolev}.

\begin{lemma}[Expansion of the $(n-1)$-Dirichlet energy]\label{lem_Dirichlet_exp}
Let $w\in W^{1,n-1}(\S;\R^n)$ satisfy \eqref{eq:condw}. 
For every $\kappa\in (0,1)$ there exists $c(\kappa)>0$ such that 
\begin{align}\label{eq:dirichlet_energy_fig_expansion_trivial_bound}
\begin{split}
\mc D_{n-1}(\mathrm{id}_{\S}+w)\geq 1&+\frac{(1-\kappa)}{2}\fint_{\S}|\n_T w|^2+ c(\kappa)\fint_{\S}|\n_T w|^{n-1}\\
&+\frac{(1-\kappa)(n-3)}{2}\fint_{\S}|N(P_T,\n_T w)|^{n-3}\big(|P_T|-|P_T+\n_Tw|\big)^2\,,	
\end{split}
\end{align}
where
\begin{equation}\label{eq:nonlinear_extra_termapplication_of_figalli_zhang}
N(X,Y):=\begin{cases}
\frac{X}{|X|} &\text{if }|X|\leq |X+Y|\,, \\
\Big(\frac{|X+Y|}{|X|}\Big)^{\frac 1{n-3}}\Big(\frac{X+Y}{|X|}\Big)
&
\text{if }|X|\geq |X+Y|\,.
\end{cases}
\end{equation}
\end{lemma}

\begin{proof}
For $\kappa\in(0,1)$ arbitrary, using the algebraic inequality obtained in \cite[Lemma 2.1(ii)]{figalli2022sobolev} with $p:=n-1\geq 2$, we deduce that for some constant $c(\kappa)>0$,
\begin{align}\label{eq:application_of_figalli_zhang}
\begin{split}
|P_T+\nabla_Tw|^{n-1}& \geq |P_T|^{n-1}+(n-1)|P_T|^{n-3}P_T:\n_T w+ c(\kappa)|\n_T w|^{n-1}\\
&\quad +\frac{(1-\kappa)(n-1)}{2}\Big[|P_T|^{n-3}|\n_T w|^2\\
&\qquad +(n-3)|M(P_T,\n_T w)|^{n-3}\big(|P_T|-|P_T+\n_Tw|\big)^2\Big]\,
\end{split}
\end{align}
$\haus$-a.e. on $\S$, where for every $X,Y\in \R^{n\times(n-1)}$ we have denoted
\begin{equation*}
M(X,Y):=\begin{cases}
X &\text{if }|X|\leq |X+Y|\,, \\
\Big(\frac{|X+Y|}{|X|}\Big)^{\frac 1{n-3}}(X+Y)
&
\text{if }|X|\geq |X+Y|\,.
\end{cases}
\end{equation*}
By using the fact that $|P_T|=\sqrt{n-1}$,  integrating \eqref{eq:application_of_figalli_zhang} on $\S$, using
\[\fint_{\S}P_T:\n_T w=\fint_{\S}\mathrm{div}_{\S}w=(n-1)\fint_{\S}\langle w,x\rangle=0\,,\]
(cf. \eqref{eq:condw}) and rearranging terms, we arrive at \eqref{eq:dirichlet_energy_fig_expansion_trivial_bound}. 
\end{proof}

We are now ready to give the proof of Proposition \ref{prop:local_nonlinear_stability},  which is based on the above two lemmata, the coercivity estimate of Theorem \ref{thm:linstab}, and a contradiction$\slash$compactness argument 
inspired by \cite[Proposition~3.8]{figalli2022sobolev}.

\begin{proof}[Proof Proposition \ref{prop:local_nonlinear_stability}] Let us choose $\theta_n\in(0,1)$ small enough so that Lemma \ref{lemma:prep_calc} applies, and let also $\kappa\in (0,1)$ to be chosen later. In what follows, we will tacitly assume that the value of the constant $c(\kappa)>0$ can vary along the proof, as long as it always depends only on $\kappa$ and $n$, but not on $w$. 

Using \eqref{eq:exp_vol}--\eqref{eq:dirichlet_energy_fig_expansion_trivial_bound} together with the algebraic inequality  $(1+t)^{\frac{n}{n-1}}\geq 1+\tfrac{n}{n-1}t,$ valid for $t\geq 0$,
we obtain
\begin{align}\label{eq:basic_ineq}
\begin{split}
\frac{3}{2}
\mathcal E_{n-1}(\mathrm{id}_\S+w)&\geq \big[\mc D_{n-1}(\mathrm{id}_{\S}+w)\big]^{\frac{n}{n-1}}-\mc V_n(\mathrm{id}_{\S}+w)\\
&
\geq
(1-\kappa)
\widetilde Q_n(w) 
+ c(\kappa) \fint_{\mathbb S^{n-1}}|\nabla_T w|^{n-1} 
 - \frac{\kappa n}{2}
\fint_{\mathbb S^{n-1}}\langle w,A(w)\rangle
\\
&\quad
-\left[\sum_{k=3}^{n-1}\frac{n}{k}\fint_{\mathbb S^{n-1}}\langle w,[\sigma_k'(\nabla_T w P_T^t)]^t x\rangle+\mc V_n(w) \right]\,,
\end{split}
\end{align}
where we have set
\begin{align}\label{eq:tildeQn}
\begin{split}
\widetilde Q_n(w)
&
:=
\frac 12 \frac{n}{n-1}\fint_{\mathbb S^{n-1}} |\nabla_T w|^2 - \frac{n}{2}\fint_{\S} \langle w,A(w)\rangle 
\\
&
\qquad
+\frac{n(n-3)}{2(n-1)}\fint_{\mathbb S^{n-1}}|N(P_T,\nabla_T w)|^{n-3}(|P_T+\nabla_T w|-|P_T|)^2\,,
\end{split}
\end{align}
with $N$ as in \eqref{eq:nonlinear_extra_termapplication_of_figalli_zhang}
 and $A$ as in \eqref{eq:opearator_A}.   
Recalling also the definition of $Q_n$ in \eqref{eq:Qn}, note that 
\begin{align*}
\widetilde Q_n(w)=Q_n(w)+R_n(w)\,,
\end{align*}
with $R_n$ given by
\begin{align*}
R_n(w)
&
:=\frac{n(n-3)}{2(n-1)}\fint_{\mathbb S^{n-1}}
\Big(
|N(P_T,\nabla_T w)|^{n-3}(|P_T+\nabla_T w|-|P_T|)^2
-\frac{(\mathrm{div}_{\S}w)^2}{n-1}
\Big)\,.
\end{align*}
We now split the rest of the proof into two steps.

\medskip
\textit{Step 1: Estimating the last term in \eqref{eq:basic_ineq}.} 
The volume term can be simply estimated by the isoperimetric inequality of Proposition \ref{prop:isop}:
\begin{align*}
\mc V_n(w)
\leq 
\Big[\mc D_{n-1}(w)\Big]^{\frac{n}{n-1}} \sim \bigg[\fint_\S |\n_T w|^{n-1}\bigg]^{\frac{n}{n-1}} \leq \frac{c(\kappa)}{2} \fint_\S |\n_T w|^{n-1}\,,
\end{align*}
if $\theta_n\in (0,1)$ is chosen sufficiently small. 
Thus, after redefining $c(\kappa)$, we get
\begin{align}\label{eq:lower_ineq_1}
\begin{split}
\frac{3}{2}
\mathcal E_{n-1}(\mathrm{id}_\S+w)
&
\geq
(1-\kappa)
\widetilde Q_n(w) 
+  c(\kappa) \fint_{\mathbb S^{n-1}}|\nabla_T w|^{n-1} 
-
\frac{\kappa n}{2} \fint_{\mathbb S^{n-1}} \langle w,A(w)\rangle  \\
&
\quad
-\sum_{k=3}^{n-1}\frac{n}{k}\fint_{\mathbb S^{n-1}}\langle w,[\sigma_k'(\nabla_T w P_T^t)]^tx\rangle\,.
\end{split}
\end{align}
To estimate the  last  sum, we claim that, with $\alpha_n:= \frac{1}{2(n-1)}>0$ and $\beta_n:=\frac{n+1}{n}\alpha_n>0$,
for all $k\in\lbrace 3,\ldots, n-1\rbrace$ 
we have
\begin{align}
\label{eq:interpolation}
\begin{split}
\fint_{\mathbb S^{n-1}}\langle w,[\sigma_k'(\nabla_T w P_T^t)]^t x\rangle & \lesssim
\int_{\mathbb S^{n-1}}|w|\, |\nabla_T w|^{ k-1}\\
& \lesssim \Big(\fint_\S |\n_T w|^2\Big)^{1+\alpha_n} + \Big(\fint_\S |\n_T w|^{n-1}\Big)^{1+\beta_n}\,,
\end{split}
\end{align}
where the first inequality is simply \eqref{eq:trivial_estimate_for_intermediate_terms}. In fact,  as we will now see, \eqref{eq:interpolation} holds for any $w$ satisfying $\fint_{\S} w=0$, which in particular is ensured by  \eqref{eq:condw}. 
The proof of \eqref{eq:interpolation} is an application of Sobolev and H\"older's inequalities with appropriate choices of exponents. 
First, by H\"older's inequality we have
\begin{align*}
\fint_{\mathbb S^{n-1}}|w|\, |\nabla_T w|^{ k-1}
\leq \|w\|_{L^{2n}(\S)} \Big( \fint_{\S} 
|\nabla_T w|^{\gamma}\Big)^{\frac{2n-1}{2n}}\,,\ \ \gamma :=\frac{2n (k-1) }{2n-1}\in (2,n-1)\,.
\end{align*}
To estimate the first factor on the right hand side, we use the Sobolev and Hölder's inequalities once again, to obtain 
\[\|w\|_{L^{2n}(\S)}\lesssim\|\nabla_T w\|_{L^{p_{n-1}}(\S)}\leq \|\nabla_T w\|_{L^{n-1}(\S)}\,,\]
where $p_{n-1} \in (1,n-1)$ is such that $p^*_{n-1}:=\frac{(n-1)p_{n-1}}{n-1-p_{n-1}}=2n$.
In order to estimate the second factor, we use the simple algebraic inequality 
\[2\leq \gamma\leq n-1\implies |\nabla_T w|^\gamma\leq |\nabla_Tw|^2 +|\nabla_Tw|^{n-1} \ \ \haus\text{-a.e. on } \S\,.\]
Hence, using also that $$(s+t)^a\leq s^a+t^a \quad \forall s,t\geq 0, a\in[0,1]\,,$$ 
we deduce 
\begin{align*}
\fint_{\mathbb S^{n-1}}|w|\, |\nabla_T w|^{k-1}
&
\lesssim \|\nabla_T w\|_{L^{n-1}(\S)} \Big( \fint_{\S} 
(|\n_T w|^2+|\n_T w|^{n-1})\Big)^{\frac{2n-1}{2n}}\,\\[3pt]
&\lesssim \|\nabla_T w\|_{L^{n-1}(\S)}
\Big( \|\nabla_T w\|_{L^2(\S)}^{\frac{2n-1}{n}}
+\|\nabla_T w\|_{L^{n-1}(\S)}^{(n-1)\frac{2n-1}{2n}}\Big)
\\[3pt]
&
\lesssim \|\nabla_T w\|_{L^{n-1}(\S)}^{n-\frac{n-1}{2n}} +
\|\nabla_T w\|_{L^{n-1}(\S)} \|\nabla_T w\|_{L^{2}(\S)}^{\frac{2n-1}{n}}\,.
\end{align*}
Finally, using Young's inequality to estimate the last term, we obtain
\begin{align*}
\int_{\mathbb S^{n-1}}|w|\, |\nabla_T w|^{ k-1}
&
\lesssim \|\nabla_T w\|_{L^{n-1}(\S)}^{n-\frac{n-1}{2n}} 
+ \|\nabla_T w\|_{L^{2}(\S)}^{\frac{2n-1}{n-1}}\,,
\end{align*}
since $\|\n_T w\|_{L^{n-1}(\S)}^n\leq \|\n_T w\|^{n-\frac{n-1}{2n}}_{L^{n-1}(\S)}$ by  \eqref{eq:condw}; this last estimate is precisely \eqref{eq:interpolation}.
Returning to \eqref{eq:lower_ineq_1}, and since $\alpha_n,\beta_n>0$,  by choosing $\theta_n\in (0,1)$ smaller if necessary and redefining $c(\kappa)$, \eqref{eq:interpolation} yields
\begin{align}\label{eq:basic_ineq_3}
\begin{split}
\frac{3}{2}\mathcal E_{n-1}(\mathrm{id}_\S+w)
&
\geq
(1-\kappa)
\widetilde Q_n(w) 
+  c(\kappa) \fint_{\mathbb S^{n-1}}|\nabla_T w|^{n-1} 
\\
&
\quad
 -\frac{\kappa n}{2} \fint_{\mathbb S^{n-1}} \langle w,A(w)\rangle
-c_n \Big( \fint_{\mathbb S^{n-1}} |\nabla_T w|^2  \Big)^{1+\alpha_n},
\end{split}
\end{align}
for any map $w$ satisfying \eqref{eq:condw}. 

Before continuing with the second part of the proof, let us remark that the choice of the exponent $\alpha_n:=\frac{1}{2(n-1)}$ in \eqref{eq:interpolation} is arbitrary, in the sense that it could be replaced by any exponent $\alpha\in(0,\frac{1}{n-1})$, up to adjusting the value of the multiplicative constant in front of the right hand side of the estimate. Indeed, one could alternatively apply the Sobolev inequality for any $p\in (1,n-1)$ in the argument below \eqref{eq:interpolation} and adjust the exponents with which Hölder's and Young's inequalities are applied afterwards. 

\medskip
\textit{Step 2:  Stability estimate for $\tilde Q_n$.}
From the estimate \eqref{eq:basic_ineq_3} we see that, if the linear stability estimate \eqref{eq:lin_stab} is still valid when $Q_n$ is replaced by the nonlinear form $\widetilde Q_n$ with a possibly smaller $c_n>0$ and for maps satisfying \eqref{eq:condw},  
then for small enough $\theta_n, \kappa\in (0,1)$ the terms in the last line of \eqref{eq:basic_ineq_3} 
can be absorbed into $(1-\kappa)\widetilde Q_n(w)$ and this would conclude the proof of \eqref{eq:nonlin_stab}; compare with \cite[Proposition~3.8(iii)]{figalli2022sobolev}. In fact, it may be that  $\widetilde Q_n$ does not satisfy such an estimate,  but  a weaker estimate suffices: it is sufficient to prove that, for any given $C,\alpha>0$
and
$c \in \R$
 with $|c|$ 
 small enough,
there exists $\theta:=\theta(n,C,\alpha,c)>0$,  such that
\begin{align}\label{eq:nonlin_spectral_gap}
&
\widetilde Q_n(w) 
\geq c  \fint_{\mathbb S^{n-1}}\langle w,A(w)\rangle+C\Big( \fint_{\mathbb S^{n-1}}|\nabla_T w|^2\Big)^{1+\alpha}, 
\end{align}
whenever \eqref{eq:condw} holds with $\theta$ in the place of $\theta_n$.
Recall again here that $A$ 
is the first order differential operator defined in \eqref{eq:opearator_A}.
 Indeed, once the above estimate is shown, we can apply it for 
\[c:=
\frac{\kappa n}{2(1-\kappa)}\,, \quad C:=\frac{c_n}{{1-\kappa}}>0\,, \quad \alpha:=\alpha_n=\frac{1}{2(n-1)}\]
for $\kappa\in(0,1)$ sufficiently small, and $\theta_n\in (0,1)$ small accordingly, and combining it with \eqref{eq:basic_ineq_3}, we are led to \eqref{eq:nonlin_stab}.

Thus it remains to  prove \eqref{eq:nonlin_spectral_gap}, and for that we argue by contradiction: for fixed $C,\alpha>0$ and
$c \in \R$ with $|c|\in(0,1)$ sufficiently small,
 assume that the estimate fails. In this case,  there exists a sequence $(w_j)_{j\in \N}\subset W^{1,n-1}(\S;\R^n)$ satisfying \eqref{eq:condw}, such that
\begin{gather}\label{eq:absurd_assumption}
\begin{split}
\delta_j& := \fint_{\mathbb S^{n-1}}|\nabla_T w_j|^{n-1} \to 0 \ \ \text{as } j\to \infty\,,
\\
\widetilde Q_n(w_j) 
& < c  \fint_{\mathbb S^{n-1}}\langle w_j,A(w_j)\rangle
+C\Big( \fint_{\mathbb S^{n-1}}|\nabla_T w_j|^2\Big)^{1+\alpha}\,.
\end{split}
\end{gather}
Recalling the definition of $\widetilde Q_n$ in  \eqref{eq:tildeQn}, the latter inequality is equivalent to
\begin{align}\label{eq:absurd_assumption_2}
\begin{split}
&
\frac 12 \frac{n}{n-1}\fint_{\mathbb S^{n-1}} |\nabla_T w_j|^2
+
\Xi_n(w_j)\\
&\qquad <\Big(\frac{n}{2}+c\Big)\fint_{\mathbb S^{n-1}}
\langle w_j,A(w_j)\rangle
+C\Big( \fint_{\mathbb S^{n-1}}|\nabla_T w_j|^2\Big)^{1+\alpha},
\end{split}
\end{align}
where
\begin{align}\label{eq:Xi_n}
\Xi_n(w)
&
:=
\frac{n(n-3)}{2(n-1)}\fint_{\mathbb S^{n-1}}|N(P_T,\nabla_T w)|^{n-3}(|P_T+\nabla_T w|-|P_T|)^2\,,
\end{align}
and $N$ as in \eqref{eq:nonlinear_extra_termapplication_of_figalli_zhang} Now let 
\begin{align*}
\e_j:=\Big( \fint_{\mathbb S^{n-1}}|\nabla_T w_j|^2\Big)^{\frac 12} 
\leq \delta_j^{\frac{1}{n-1}}\to 0 \ \ \text{as } \ j\to \infty\,,
\ \ \hat w_j:=\frac{w_j}{\e_j}\,,
\end{align*}
so that $\fint_\S |\nabla_T \hat w_j|^2=1$. Rewriting 
 \eqref{eq:absurd_assumption_2} in terms also of $\hat w_j$, we get 
\begin{align}\label{eq:absurd_assumption_3}
&\frac 12 \frac{n}{n-1}
+\frac{\Xi_n(w_j)}{\e_j^2}
<\Big(\frac{n}{2}+c\Big)\fint_{\mathbb S^{n-1}}
\langle \hat w_j,A(\hat w_j)\rangle 
+C\,  \e_j^{2\alpha}\,.
\end{align}
Extracting a subsequence,  we can assume
$\hat w_j\rightharpoonup \hat w$ weakly in $W^{1,2}(\S ;\R^n )$, so $\hat w_j\to \hat w$ strongly in $L^2(\S ;\R^n )$, where $\hat w\in W^{1,2}(\S ;\R^n )$ also satisfies 
\begin{equation}\label{eq:normalization_for_w}
\fint_{\S}\hat w=0\,, \quad \ \Pi_{n,0}\hat w=0\,.	
\end{equation}
In particular, taking the liminf as $j\to \infty$ in  \eqref{eq:absurd_assumption_3}, we deduce that
\begin{align}\label{eq:absurd_assumption_4}
\frac 12 \frac{n}{n-1}+\liminf_{j\to +\infty}
\Bigg[\frac{\Xi_n(w_j)}{\e_j^2}
\Bigg]
\leq \Big(\frac{n}{2}+c\Big)
\fint_{\mathbb S^{n-1}} \langle \hat w,A(\hat w)\rangle\,.
\end{align}
Since $\Xi_n\geq 0$, the left-hand side of \eqref{eq:absurd_assumption_4} is bounded from below by $n/(2n-2)$, so we deduce in particular that 
\begin{align}\label{eq:hatw_nontrivial}
\fint_{\mathbb S^{n-1}}|\nabla_T \hat w|^2 > 0\,.
\end{align}

We now claim that
\begin{align}\label{eq:liminf_quad}
\begin{split}
&
\frac 12 \frac{n}{n-1}+\liminf_{j\to +\infty}
\Bigg[\frac{\Xi_n(w_j)}{\e_j^2}
\Bigg]
\geq
\frac 12 \frac{n}{n-1}\fint_{\mathbb S^{n-1}} |\nabla_T \hat w|^2
+\frac{n(n-3)}{2(n-1)^2}
\fint_{\mathbb S^{n-1}}
(\mathrm{div}_{\S}\hat w)^2\,.
\end{split}
\end{align}
Once this is shown,  and recalling the definition \eqref{eq:Qn}
 of $Q_n$, we see that \eqref{eq:absurd_assumption_4} implies 
\begin{align*}
Q_n(\hat w)\leq c\fint_{\mathbb S^{n-1}}\langle \hat w,A(\hat w) \rangle
\lesssim 
|c| 
\fint_{\S}|\nabla_T\hat w|^2\,,
\end{align*}
the last estimate following from the Cauchy-Schwarz inequality, \eqref{eq:opearator_A},  and the Poincar{\' e} inequality on $\S$ for functions with zero mean, which is applicable by \eqref{eq:normalization_for_w}.
But according to \eqref{eq:hatw_nontrivial} $\hat w$ is a nontrivial map satisfying \eqref{eq:normalization_for_w}, so this last estimate is a clear contradiction to the coercivity estimate \eqref{eq:lin_stab} for small enough  $|c|>0$.

Therefore, to complete the proof of \eqref{eq:nonlin_spectral_gap} 
it suffices to prove \eqref{eq:liminf_quad}.
By the fundamental theorem of calculus, we have that $\haus$-a.e. on $\S$,
\begin{align*}
\frac{|P_T+\nabla_T w_j|-|P_T|}{\e_j}
=\int_0^1 
\frac{P_T +t\nabla_T w_j}{|P_T+t\nabla_T w_j|}\colon \nabla_T  \hat w_j
\d t
=f_j +g_j\,,
\end{align*}
where
\begin{align*}
f_j & :=
\int_0^1 
\frac{P_T +t\nabla_T w_j}{|P_T+t\nabla_T w_j|}\colon\nabla_T \hat w
\d t\,,
\\[3pt]
g_j
&:=
\int_0^1 
\frac{P_T +t\nabla_T w_j}{|P_T+t\nabla_T w_j|}\colon(\nabla_T \hat w_j-\nabla_T \hat w)
\d t\,.
\end{align*}
Recall by \eqref{eq:absurd_assumption} that  $\nabla_T w_j\to 0$ strongly in $L^{n-1}(\S)$, hence also $\haus$-a.e. up to a non-relabeled subsequence, and therefore for every $t\in[0,1]$
\begin{align*}
\frac{P_T +t\nabla_T w_j}{|P_T+t\nabla_T w_j|}
\to \frac{P_T}{|P_T|}\quad \haus\text{-a.e.\ on } \S.
\end{align*}
By the dominated convergence theorem and the fact that $\hat w_j\rightharpoonup \hat w$ in $W^{1,2}(\S;\R^n)$, we deduce
\begin{align*}
&f_j\to \frac{P_T\colon\nabla_T\hat w}{|P_T|}=\frac{\mathrm{div}_{\S}\hat w}{\sqrt{n-1}}\ 
\text{ strongly in }L^2(\S)\,,
\quad
g_j
\rightharpoonup 0 \text{ weakly in }L^2(\S)\,.
\end{align*}
Recalling the definition \eqref{eq:Xi_n} of $\Xi_n$, we have
\begin{align*}
\frac{\Xi_n(w_j)}{\e_j^2}
&
=\frac{n(n-3)}{2(n-1)}\fint_{\mathbb S^{n-1}}
|N(P_T,\nabla_T w_j)|^{n-3}(f_j+g_j)^2 \\
&
\geq
\frac{n(n-3)}{2(n-1)}\fint_{\mathbb S^{n-1}}
|N(P_T,\nabla_T w_j)|^{n-3}f_j^2
+
\frac{n(n-3)}{n-1}\fint_{\mathbb S^{n-1}}
|N(P_T,\nabla_T w_j)|^{n-3}f_j g_j\,.
\end{align*}
Moreover, recalling  \eqref{eq:nonlinear_extra_termapplication_of_figalli_zhang}, we have 
\[1\geq |N(P_T,\nabla_T w_j)|\to 1 \ \haus\text{-a.e. on }\S\,,\]
so that by the dominated convergence theorem and the above convergences of $f_j$ and $g_j$ we infer
\begin{align*}
\liminf_{j\to +\infty} \frac{\Xi_n(w_j)}{\e_j^2} \geq 
\frac{n(n-3)}{2(n-1)^2}
\fint_{\mathbb S^{n-1}} (\mathrm{div}_{\S}\hat w)^2\,.
\end{align*}
This inequality,  combined with the weak lower semicontinuity of the Dirichlet energy, implies \eqref{eq:liminf_quad}, thus concluding the proof of the proposition.
\end{proof}

The proof of our main theorem follows now easily by combining all the above auxiliary steps.

\begin{proof}[Proof of Theorem \ref{thm:qttiveintro}]
In view of Lemma \ref{from_local_to_global}, it suffices to prove the theorem for maps $u\in \mathcal{B}_{\delta_n,\varepsilon_n}$ as in \eqref{eq:local_family}, for $\delta_n,\varepsilon_n\in (0,1)$ sufficiently small to be chosen in the sequel, so that also Lemma \ref{fixingMobius} is applicable. Then, by Lemma \ref{lem:w} the map $w\in W^{1,n-1}(\S;\R^n)$ defined in \eqref{eq:correct_w} satisfies the hypotheses of Proposition \ref{prop:local_nonlinear_stability}, for some $\theta_n\in (0,1)$ that can be chosen sufficiently small depending on $\delta_n$. In particular, by the scaling, translation, and conformal invariance of the deficit $\mc E_{n-1}$,  \eqref{eq:correct_w} and \eqref{eq:nonlin_stab}, we obtain
\begin{align}\label{eq:almost_final_nonlinear}
\begin{split}
\mc E_{n-1}(u)& =\mc E_{n-1}(\mathrm{id}_{\S}+w)\\
&\geq \tilde c_n\fint_{\S}|\nabla_T w|^{n-1}= \tilde c_n\fint_{\S}\bigg|\frac{1}{\lambda_{u,\psi}}\nabla_T (u\circ\psi)-P_T\bigg|^{n-1}\,,
\end{split}
\end{align}
so that the only thing that remains to be justified is why we can replace $\lambda_{u,\psi}$ with $[\mc V_n(u)]^{\frac{1}{n}}$ in \eqref{eq:almost_final_nonlinear}. In that respect, we prove that 
\begin{equation}\label{eq:lambda_vs_vol}
\bigg|\frac{1}{[\mc V_n(u)]^{\frac{1}{n}}}-\frac{1}{\lambda_{u,\psi}}\bigg|\lesssim \big[\mc E_{n-1}(u)\big]	^{\frac{1}{n-1}}\,.
\end{equation}
Once \eqref{eq:lambda_vs_vol} is established, the assertion follows by setting $\phi:=\psi^{-1}\in \Mob_+(\mb S^{n-1})$, so that using the conformal invariance of the $(n-1)$-Dirichlet energy and \eqref{eq:almost_final_nonlinear},  we get
\begin{align*}
\fint_{\S}\bigg|\frac{\nabla_T u}{[\mc V_n(u)]^{\frac{1}{n}}}-\nabla_T\phi\bigg|^{n-1}&\lesssim \bigg|\frac{1}{[\mc V_n(u)]^{\frac{1}{n}}}-\frac{1}{\lambda_{u,\psi}}\bigg|^{n-1}\fint_{\S}|\nabla_T u|^{n-1}+\fint_{\S}\bigg|\frac{\nabla_T (u\circ \psi)}{\lambda_{u,\psi}}-P_T\bigg|^{n-1}\\[2pt]
&\lesssim \mc E_{n-1}(u)\,,
\end{align*} 	
which proves \eqref{eq:qttive}. 

Thus, as a final step, we prove \eqref{eq:lambda_vs_vol}.  For simplicity, let us write $\tilde u := (u \circ \psi)/\lambda_{u,\psi}$. We recall that, for $a,b\in \R^m$, by convexity of the function $a\mapsto |a|^{n-1}$ we have the local Lipschitz estimate
\begin{align*}
\big||a|^{n-1}-|b|^{n-1}\big|\lesssim (|a|^{n-2}+|b|^{n-2})|a-b|\,.
\end{align*}
Applying this inequality with $a:=\nabla_T \tilde u$ and $b:=P_T$ ($\haus$-a.e. on $\S$), and using \eqref{eq:ucompphi_zero_proj},  \eqref{eq:local_family}(iii) and \eqref{eq:almost_final_nonlinear}, we obtain
\begin{align}\label{eq: lambda_vs_dirichlet}
\begin{split}	
\big|\mc D_{n-1}(\tilde u)-1\big|&\sim \bigg|\fint_{\S}\Big(\big|\nabla_T\tilde u \big|^{n-1}-|P_T|^{n-1}\Big)\bigg|\\
&\lesssim \fint_{\S}\left(|\nabla_T\tilde u|^{n-2}+|P_T|^{n-2}\right)\big|\nabla_T\tilde u-P_T\big|\\
&\lesssim \left(\|\nabla_T\tilde u\|^{n-2}_{L^{n-1}(\S)}+1\right)\big\|\nabla_T\tilde u-P_T\big\|_{L^{n-1}(\S)}\\
&\lesssim_{\tilde \delta_n} [\mc E_{n-1}(u)]^{\frac{1}{n-1}}\lesssim_{\tilde \delta_n} \varepsilon_n^{\frac{1}{n-1}}\leq \frac{1}{2}\,.
\end{split}
\end{align} 
Since the function $t\mapsto t^\frac{1}{n-1}$ is smooth around $t=1$, \eqref{eq: lambda_vs_dirichlet} in particular implies
\begin{equation*}
\left|[\mc D_{n-1}(\tilde u)]^{\frac{1}{n-1}}-1\right|\lesssim \left|\mc D_{n-1}(\tilde u)-1\right|\lesssim [\mc E_{n-1}(u)]^{\frac{1}{n-1}}\,.
\end{equation*}
Taking $\delta_n\in(0,1)$ sufficiently small, we can of course ensure that $\mc D_{n-1}(u)\in [\tfrac 1 2, \tfrac 3 2]$,  so that
\begin{align}\label{eq: lambda_vs_dirichlet_2}
\begin{split}
\bigg|\frac{1}{\lambda_{u,\psi}}-\frac{1}{[\mc V_{n}(u)]^{\frac{1}{n}}(1+\mc E_{n-1}(u))^{\frac{1}{n}}}\bigg| & =
\bigg|\frac{1}{\lambda_{u,\psi}}-\frac{1}{[\mc D_{n-1}(u)]^{\frac{1}{n-1}}}\bigg| \\
& \sim \left|[\mc D_{n-1}(\tilde u)]^{\frac{1}{n-1}}-1\right| \lesssim [\mc E_{n-1}(u)]^{\frac{1}{n-1}}\,,
\end{split}
\end{align}	
since $\lambda_{u,\psi}^{n-1} \mc D_{n-1}(\tilde u ) = \mc D_{n-1}(u)$.
Notice that,  similarly to \eqref{eq:vol_trivial_bound}, we also have $\mc V_n(u) \in [\tfrac 1 2, \tfrac 3 2]$, provided that 
$\delta_n \in (0,1)$ 
is sufficiently small: indeed, this follows easily from  definition \eqref{eq:correct_w} and estimate \eqref{eq:lambda_estimate}. Hence, by elementary analysis, we obtain
\begin{align}\label{eq: lambda_vs_dirichlet_3}
\frac{1}{[\mc V_{n}(u)]^{\frac{1}{n}}}\Bigg|\frac{1}{(1+\mc E_{n-1}(u))^{\frac{1}{n}}}-1\Bigg|\lesssim (1+\mc E_{n-1}(u))^{\frac{1}{n}}-1\lesssim \mc E_{n-1}(u)\,. 
\end{align}	 
Combining \eqref{eq: lambda_vs_dirichlet_2} and \eqref{eq: lambda_vs_dirichlet_3}, we arrive at \eqref{eq:lambda_vs_vol}, which completes the proof. \qedhere
\end{proof}

\appendix

\section{A pointwise characterization of the local topological degree}\label{a:deg}

In this appendix we prove the following.

\begin{proposition}\label{p:deg_uy}
Let
$U\in W^{1,n}(\mb B^n;\R^n)$ such that $u:=U|_{\S}\in W^{1,n-1}(\S;\R^n)$.
Then
\begin{align*}
u_y := \frac{u-y}{|u-y|}\in \mathrm{VMO}(\mathbb S^{n-1};\mathbb S^{n-1})
\quad\text{for $\L^n$-a.e. }y\in\R^n,
\end{align*}
and
\begin{align}\label{eq:degUuy}
\deg(U,\B;y)=\deg(u_y,\S;\S)\quad\text{for $\L^n$-a.e. }y\in\R^n,
\end{align}
where the local degree $\deg(U,\B;y)$ is defined by \eqref{eq:general_degreee} and $\deg(u_y,\S;\S)$ is the $\mathrm{VMO}$-degree defined in \cite[Section~I.3]{Brezis1995}.
\end{proposition}

The main tools to prove Proposition~\ref{p:deg_uy} are the properties of the local degree from Subsection~\ref{sec:degree} and of the $\mathrm{VMO}$-degree from \cite{Brezis1995}, together with the following lemma.

\begin{lemma}\label{l:uy}
For any $u\in W^{1,n-1}(\mathbb S^{n-1};\R^n)$ we have
\begin{align}\label{eq:uysob}
u_y = \frac{u-y}{|u-y|}\in W^{1,n-1}(\mathbb S^{n-1};\mathbb S^{n-1})
\quad\text{for $\L^n$-a.e. }y\in\R^n,
\end{align}
and the mapping
\begin{align}\label{eq:Phi}
\Phi\colon W^{1,n-1}(\mathbb S^{n-1};\R^n)
&
\to
L^1_{\mathrm{loc}}(\R^n;W^{1,n-1}(\mathbb S^{n-1},\mathbb S^{n-1}))\,,
\nonumber
\\
u & \mapsto \left( y\mapsto u_y \right)
\end{align}
is continuous.
\end{lemma}

The proof of
 Lemma~\ref{l:uy}
 will be given below.

\begin{proof}[Proof of Proposition~\ref{p:deg_uy}]
The fact that $u_y\in \mathrm{VMO}(\S;\S)$ for $\L^n$-a.e. $y\in \R^n$ follows from \eqref{eq:uysob} and the embedding $W^{1,n-1}(\S;\S)\subset \mathrm{VMO}(\S;\S)$ \cite[Example~2, page 209]{Brezis1995}.
Then note that the identity \eqref{eq:degUuy} is true if $U\in C^1(\overline \B;\R^n)$,
since in that case $\mathcal L^n(u(\mathbb S^{n-1}))=0$ and \eqref{eq:degUuy} is satisfied for all $y\notin u(\S)$ \cite[Remark~1.5.10]{Nirenberg2001}).

Now fix a sequence $(U_j)_{j\in\N}\subset C^1(\overline {\mathbb{B}^n};\R^n)$ such that $U_j\to U$ $\mathcal L^n$-a.e. in $\B$ and strongly in $W^{1,n}(\B;\R^n)$, and $u_j:=U_j|_{\S}\to u$ strongly in $W^{1,n-1}(\S;\R^n)$.
By the above, we have
\begin{align}\label{eq:degUjujy}
\deg(U_j,\B;y)=\deg((u_j)_y,\S;\S)\ \ \text{for }\L^n\text{-a.e. }y\in\R^n,
\end{align}
and all $j\in\N$.
Moreover, the proof of Lemma~\ref{lem:TUBV} ensures that
\begin{align}\label{eq:convdegUj}
\deg(U_j,\B;y)\to \deg(U,\B;y)\ \ \text{for }\L^n\text{-a.e. }y\in\R^n
\end{align}
as $j\to\infty$.
Thanks to the continuity of the map $\Phi$ in Lemma~\ref{l:uy}
and Fatou's lemma,
for $\L^n$-a.e. $y\in\R^n$ there is a subsequence $j'\to\infty$ such that $(u_{j'})_y \to u_y$ strongly in $W^{1,n-1}(\S;\S)$.
By the embedding 
$W^{1,n-1}(\S;\S)\subset \mathrm{VMO}(\S;\S)$  and \cite[Section I.3, Theorem~1]{Brezis1995}, this implies
\begin{align*}
\deg((u_{j'})_y,\S;\S)\to \deg(u_y,\S;\S)\ \ \text{as }j'\to\infty.
\end{align*}
Combining this with \eqref{eq:degUjujy}-\eqref{eq:convdegUj} proves \eqref{eq:degUuy}.
\end{proof}

\begin{proof}[Proof of Lemma~\ref{l:uy}]
The proof is inspired by \cite[Lemma~4.3]{JerrardSoner2002} and \cite[Lemma~3.8]{CanevariOrlandi2019}, with slight modifications due to the fact that we do not impose $u\in L^\infty(\S;\R^n)$.

Let  $u\in W^{1,n-1}(\mathbb S^{n-1};\R^n)$.
First note that
 the set $\lbrace y\in\R^n\colon \mathcal H^{n-1}(\lbrace u=y\rbrace)>0\rbrace$
 is at most countable since 
 $\lbrace y\in\R^n\colon \mathcal H^{n-1}(\lbrace u=y\rbrace)\geq 1/N\rbrace$ is finite for any integer $N\geq 1$. 
 Therefore $u_y$ is well-defined $\mathcal H^{n-1}$-a.e. on $\mathbb S^{n-1}$, for $\L^n$-a.e. $y\in\R^n$.

For $\delta>0$ and $y\in\R^n$ define
\begin{align*}
P_{y,\delta}&   
\colon\R^n \to \R^n,
\quad
z
\mapsto \frac{z-y}{\max{|z-y|,\delta}}\,,
\\
P_y &
\colon \R^n\setminus \lbrace y\rbrace \to \R^n\,,
\quad
z
\mapsto \frac{z-y}{|z-y|}\,,
\end{align*}
so that $u_y=P_y(u)$,
and
since $P_{y,\delta}\to P_y$ pointwise on $\R^n\setminus \lbrace y\rbrace$ as $\delta\to 0$ and $|P_{y,\delta}|\leq 1$, by the dominated convergence theorem we have
\begin{align*}
u_{y,\delta}:=P_{y,\delta}(u)\longrightarrow u_y\ \ \text{ in }\mathcal D'(\mathbb S^{n-1};\R^n) \ \ \text{as }\delta\to 0\,.
\end{align*}
Since $P_{y,\delta}$ is Lipschitz we have $u_{y,\delta}\in W^{1,n-1}(\mathbb S^{n-1};\R^n)$, and
\begin{align}\label{eq:convgraddistr}
\nabla P_{y,\delta}(u)\nabla_Tu = \nabla_Tu_{y,\delta} \longrightarrow \nabla_Tu_y
\ \ \text{ in }\mathcal D'(\mathbb S^{n-1};\R^n)
\ \ \text{as }\delta\to 0\,.
\end{align}
For all $z\in\R^n\setminus \lbrace y\rbrace$ we have
\begin{align*}
\nabla P_y(z)
&
=\frac{1}{|z-y|} 
\left( 
I_n -\frac{z-y}{|z-y|}\otimes
 \frac{z-y}{|z-y|}\right),
 \\[3pt]
\nabla P_{y,\delta}(z)
&
=\mathbf 1_{|z-y|>\delta} \nabla P_y(z) +\mathbf 1_{|z-y|\leq \delta}\bigg(\frac{1}{\delta} I_n\bigg)\,,
\end{align*}
and therefore
\begin{align*}
|
\nabla P_{\delta,y}(u)\nabla_Tu - \nabla P_y(u) \nabla_Tu
|
\leq
2
\frac{|\nabla_Tu|}{|u-y|} \ \ \haus\text{-a.e. on } \S\,.
\end{align*}
For any $y\in\R^n$ satisfying
\begin{align}\label{eq:integrDu_y}
\frac{|\nabla_Tu|}{|u-y|}\in L^{n-1}(\mathbb S^{n-1})\,,
\end{align}
then by the dominated convergence theorem we would get
\begin{align*}
\nabla_Tu_{y,\delta}=\nabla P_{y,\delta}(u)\nabla_Tu \longrightarrow \nabla P_y(u)\nabla_Tu\quad\text{strongly in }L^{n-1}(\mathbb S^{n-1};\R^{n\times (n-1)})\,,
\end{align*}
and therefore, thanks to \eqref{eq:convgraddistr},
\begin{align*}
\nabla_T u_y =\nabla P_y(u)\nabla_Tu\ \ \haus\text{-a.e. on }\mathbb S^{n-1}\,.
\end{align*}
Thus, since also $|\nabla P_y(u)\nabla_Tu|\leq |\nabla_Tu|/|u-y|$ pointwise $\haus$-a.e. on $\S$, we would infer that 
\begin{align*}
\nabla_Tu_y\in L^{n-1}(\mathbb S^{n-1};\R^{n\times(n-1)})\,.
\end{align*}
Hence, in order to prove \eqref{eq:uysob} it suffices to show the integrability \eqref{eq:integrDu_y} for $\L^n$-a.e. $y\in\R^n$.
For any $R>0$ we have by  Fubini's theorem,
\begin{align*}
&\int_{|y|\leq R}\int_{\mathbb S^{n-1}} 
\frac{|\nabla_Tu(x)|^{n-1}}{|u(x)-y|^{n-1}}\,\d\mathcal H^{n-1}(x)\, \dy
\leq M_R \int_{\mathbb S^{n-1}}|\nabla_Tu(x)|^{n-1}
\,\d\mathcal H^{n-1}(x)\,,
\end{align*}
where
\begin{align*}
M_R :=
\sup_{z\in\R^n} \int_{|y|\leq R}\frac{\dy}{|z-y|^{n-1}}\,.
\end{align*}
For $|z|\geq 2R$ we have
\begin{align*}
\int_{|y|\leq R}\frac{\dy}{|z-y|^{n-1}}\leq \int_{|y|\leq R}\frac{\dy}{R^{n-1}} \lesssim R\,,
\end{align*}
and similarly for $|z|\leq 2R$ we also have
\begin{align*}
\int_{|y|\leq R}\frac{\dy}{|z-y|^{n-1}}\leq \int_{|y|\leq 3R}\frac{\dy}{|y|^{n-1}} \lesssim R\,.
\end{align*}
Hence, $M_R\lesssim R$ and also
\begin{align*}
&\int_{|y|\leq R}\int_{\mathbb S^{n-1}} 
\frac{|\nabla_Tu(x)|^{n-1}}{|u(x)-y|^{n-1}}\,\d\mathcal H^{n-1}(x)\, \dy
\lesssim R\int_{\mathbb S^{n-1}}|\nabla_Tu|^{n-1}
\,\d\mathcal H^{n-1} <\infty\,,
\end{align*}
for any $R>0$,
which implies that \eqref{eq:integrDu_y} is satisfied for $\L^n$-a.e. $y\in\R^n$ and concludes the proof of \eqref{eq:uysob}.

Next we show that the mapping $\Phi$ defined in \eqref{eq:Phi}
is continuous. To that end, we fix a sequence $(u_j)_{j\in \N}\subset W^{1,n-1}(\S;\R^n)$ such that $u_j\to u$ strongly in $W^{1,n-1}(\mathbb S^{n-1};\R^n)$,
and assume without loss of generality that also $u_j\to u$ pointwise $\haus$-a.e. on $\mathbb S^{n-1}$.
We compute
\begin{align*}
&\int_{|y|\leq R}\int_{\mathbb S^{n-1}}|\nabla_T(u_j)_y-\nabla_Tu_y|^{n-1}\, \d\mathcal H^{n-1}\,\dy
\\[2pt]
&
=
\int_{|y|\leq R}\int_{\mathbb S^{n-1}}
 |\nabla P_y(u_j)\nabla_Tu_j-\nabla P_y(u)\nabla_Tu|^{n-1}
 \, \d\mathcal H^{n-1}\,\dy
\\[2pt]
&
\lesssim I_1^j +I_2^{j,\delta} +I_3^{j,\delta}\,,
\end{align*}
where we introduce an arbitrary $\delta>0$ and set
\begin{align*}
I_1^j
&:=
\int_{|y|\leq R}\int_{\mathbb S^{n-1}}
 |\nabla P_y(u_j)|^{n-1} |\nabla_Tu_j-\nabla_Tu|^{n-1}
 \, \d\mathcal H^{n-1}\dy\,,
\\
I_2^{j,\delta}
&
:=
\int_{|y|\leq R}\int_{\mathbb S^{n-1}}
|\nabla P_{y,\delta}(u_j)-\nabla P_{y,\delta}(u)|^{n-1} |\nabla_Tu|^{n-1}
 \, \d\mathcal H^{n-1}\dy\,,
\\
I_3^{j,\delta}
&
:=\int_{|y|\leq R}\int_{\mathbb S^{n-1}}
\left(|\nabla P_{y,\delta}-\nabla P_y|(u)
+
|\nabla P_{y,\delta}-\nabla P_y|(u_j)
\right)^{n-1}
|\nabla_Tu|^{n-1}
 \, \d\mathcal H^{n-1}\dy\,.
\end{align*}
We estimate the first term using Fubini's theorem, the fact that $|\nabla P_y(z)|\lesssim 1/|z-y|$ and the above estimate on $M_R$, 
 and obtain
\begin{align*}
I_1^j
&
\lesssim R \int_{\mathbb S^{n-1}}|\nabla_T u_j-\nabla_Tu|^{n-1}\,\d\mathcal H^{n-1}\,,
\end{align*}
and the right hand side of the above estimate tends to zero as $j\to\infty$.
Since $|\nabla P_{y,\delta}|\lesssim 1/\delta$, for fixed $\delta>0$ the second term $I_2^{j,\delta}$ tends to zero as $j\to\infty$ by the dominated convergence theorem. Finally, by Fubini's theorem, the third term $I_3^{j,\delta}$ satisfies
\begin{align*}
I_3^{j,\delta}\lesssim M_{R,\delta} \int_{\mathbb S^{n-1}}|\nabla_Tu|^{n-1}\,\d\mathcal H^{n-1}\,,
\end{align*}
where
\begin{align*}
M_{R,\delta}
&
:=
\sup_{z\in\R^n} \int_{|y|\leq R} |\nabla P_{y}(z)-\nabla P_{y,\delta}(z)|^{n-1} \, \dy\,.
\end{align*}
All this implies that, for any $\delta>0$,
\begin{align}\label{eq:basic_estimate_for_Phi_cont}
\limsup_{j\to \infty}
\int_{|y|\leq R}\int_{\mathbb S^{n-1}}|\nabla_T (u_j)_y-\nabla_Tu_y|^{n-1}\, \d\mathcal H^{n-1}\dy
\lesssim M_{R,\delta} \int_{\mathbb S^{n-1}}|\nabla_Tu|^{n-1}\,\d\mathcal H^{n-1}\,.
\end{align}
Finally, using the explicit expressions of $\nabla P_y$ and $\nabla P_{y,\delta}$,
and in particular the fact that 
\[\nabla P_y(z)-\nabla P_{y,\delta}(z)=0 \ \text{for }\ |z-y|\geq \delta\,,\] 
we find that
\begin{align*}
M_{R,\delta}
&
\lesssim \sup_{|z|\leq R+\delta}\int_{|z-y|\leq \delta}\frac{\dy}{|z-y|^{n-1}}
\lesssim \delta\,.
\end{align*}
Letting $\delta\searrow 0$ in \eqref{eq:basic_estimate_for_Phi_cont}, shows that for every $R>0$ fixed, 
\begin{align*}
\lim_{j\to\infty}\int_{|y|\leq R}\int_{\mathbb S^{n-1}}|\nabla_T(u_j)_y-\nabla_Tu_y|^{n-1}\, \d\mathcal H^{n-1}\,\dy=0\,.
\end{align*}
This, together with the fact that by the dominated convergence theorem once again,
\begin{align*}
\lim_{j\to\infty}\int_{|y|\leq R}\int_{\mathbb S^{n-1}}|(u_j)_y-u_y|^{n-1}\, \d\mathcal H^{n-1}\dy=0\,,
\end{align*}
concludes the proof that the map $\Phi$ defined in \eqref{eq:Phi} is continuous in the appropriate topologies.
\end{proof}

\section*{Acknowledgements} 

We would like to thank Giacomo Canevari for pointing out to us the proof of Proposition~\ref{p:deg_uy}. 

AG was supported by Dr.\ Max R\"ossler, the Walter H\"afner Foundation and ETH Z\"urich Foundation. 
 XL was supported by 
the ANR project ANR-22-CE40-0006. 
KZ was supported by the Deutsche Forschungsgemeinschaft (DFG, German Research Foundation) under Germany's Excellence Strategy EXC-2044-390685587, Mathematics M\"unster: Dynamics--Geometry--Structure, and currently by the Sonderforschungsbereich 1060 and the Hausdorff Center for Mathematics (HCM) under Germany's Excellence Strategy -EXC-2047/1-390685813.

We would also like to thank the Hausdorff Institute for Mathematics (HIM) in Bonn,
funded by the Deutsche Forschungsgemeinschaft (DFG, German Research Foundation) under Germany's Excellence Strategy – EXC-2047/1 – 390685813,  and the organizers of the Trimester Program ``Mathematics for Complex Materials'' (03/01/2023-14/04/2023, HIM, Bonn) for their hospitality during the period that this work was initiated. 
AG and KZ would also like to thank the hospitality of the Institut de Math\'ematiques de Toulouse 
(09/10/2023-13/10/2023), where part of this research was developed.

\bibliographystyle{abbrv-andre}
\bibliography{rigid_conform}

\end{document}